\documentclass[a4paper,11pt]{amsart}
\usepackage[utf8]{inputenc}
\usepackage{amsmath,amsthm,amssymb}
\usepackage{amsfonts}
\usepackage{mathtools}
\usepackage{bbm}
\usepackage{stmaryrd}
\usepackage{mathrsfs}
\usepackage{caption, subcaption}
\usepackage{enumitem}
\usepackage{tikz}\usetikzlibrary{decorations.markings,decorations.pathreplacing, patterns, calc,math,arrows.meta,positioning}
\usepackage{circuitikz}
\usepackage[hyperindex,hypertexnames=true,linktocpage=true]{hyperref}
\hypersetup{
    colorlinks=true,
    linkcolor=blue,
    citecolor=red,
    filecolor=magenta,
    urlcolor=cyan
    }
\usepackage{xcolor}
\usepackage{comment}
\usepackage{cite}

\numberwithin{equation}{section}

\usepackage[T3,T1]{fontenc}
\DeclareSymbolFont{tipa}{T3}{cmr}{m}{n}
\DeclareMathAccent{\invbreve}{\mathalpha}{tipa}{16}

\newenvironment{new}
  {\color{black}}
  {}

\def\sideremark#1{\ifvmode\leavevmode\fi\vadjust{\vbox to0pt{\vss 
    \hbox to 0pt{\hskip\hsize\hskip1em           
 \vbox{\hsize3.5cm\tiny\raggedright\pretolerance10000
 \noindent #1\hfill}\hss}\vbox to8pt{\vfil}\vss}}}%

\newcommand{\Cal}{\mathcal}

\newcommand{\re}{\operatorname{Re}}

\renewcommand{\i}{\mathrm{i}\mkern.25\thinmuskip}
\renewcommand{\d}{\mathrm{d}\mkern.25\thinmuskip}
\newcommand{\e}{\varepsilon}
\newcommand{\esslimsup}{\text{ess}\limsup}

\newcommand{\dist}{\operatorname{dist}}

\newcommand{\var}{\operatorname{var}}
\newcommand{\sH}{\mathscr{H}}
\newcommand{\E}{\mathrm{e}}

\newcommand\equidistantPoints[4]{%
    \draw[#1] (0, 0) rectangle (#2, #2);
    \pgfmathsetmacro{\scale}{#4}
    \pgfmathsetmacro{\step}{#2 / (#3 + 1)}
    \pgfmathsetmacro{\last}{#3-1}
    \foreach \i in {0,...,\last} {
      \foreach \j in {0,...,\last} {
        \fill[#1] (\i*\step+\step, \j*\step+\step) circle [radius=0.03/\scale];
      }
    }
}

\newtheorem{theorem}{Theorem}[section]
\newtheorem{lemma}[theorem]{Lemma}
\newtheorem{proposition}[theorem]{Proposition}
\newtheorem{corollary}[theorem]{Corollary}

\theoremstyle{definition}
\newtheorem{definition}[theorem]{Definition}

\newtheorem{remark}[theorem]{Remark}
\newtheorem{examp}[theorem]{Example}

\newcommand{\ov}[1]{\overline{#1}}
\newcommand{\un}[1]{\underline{#1}}

\renewcommand{\o}{\omega}
\newcommand{\pa}{\partial}
\newcommand{\I}{\mathbbm{i}}

\newcommand{\iy}{\infty}

\newcommand{\stq}{\subseteq}

\newcommand{\R}{\mathbb{R}}
\newcommand{\N}{\mathbb{N}}
\newcommand{\Z}{\mathbb{Z}}

\newcommand{\C}{\mathbb{C}}
\newcommand{\eS}{\mathbb{S}}

\newcommand{\dl}{\delta}
\newcommand{\reach}{\operatorname{reach}}
\newcommand{\frk}{\mathfrak}
\newcommand{\nn}{\nonumber}

\newcommand{\1}{\mathbbm{1}}
\newcommand{\z}{\zeta}

\newcommand{\eps}{\varepsilon}	 

\DeclareMathOperator{\diam}{diam}
\DeclareMathOperator{\Unp}{Unp}

\newcommand{\fb}{\beta}
\newcommand{\fbv}{\beta^{\var}}
\newcommand{\BC}{\mathcal{M}}
\newcommand{\SC}{\mathcal{S}}
\newcommand{\FC}{\mathcal{C}}
\newcommand{\uSC}{\overline{\mathcal{S}}}
\newcommand{\lSC}{\underline{\mathcal{S}}}
\newcommand{\uBC}{\overline{\mathcal{M}}}
\newcommand{\lBC}{\underline{\mathcal{M}}}
\newcommand{\uFC}{\overline{\mathcal{C}}}
\newcommand{\lFC}{\underline{\mathcal{C}}}
\newcommand{\ubex}{\overline{\mathfrak{m}}}
\newcommand{\lbex}{\underline{\mathfrak{m}}}
\newcommand{\bex}{\mathfrak{m}}
\newcommand{\ucex}{\overline{\mathfrak{c}}}
\newcommand{\lcex}{\underline{\mathfrak{c}}}
\newcommand{\cex}{\mathfrak{c}}
\newcommand{\usex}{\overline{\mathfrak{s}}}
\newcommand{\lsex}{\underline{\mathfrak{s}}}
\newcommand{\sex}{\mathfrak{s}}
\newcommand{\udimout}{\ov{\dim}^{\mathrm{out}}}
\newcommand{\ldimout}{\underline{\dim}^{\mathrm{out}}}
\newcommand{\dimout}{\dim^{\mathrm{out}}}
\newcommand{\udim}{\ov{\dim}}
\newcommand{\ldim}{\underline{\dim}}
\newcommand{\uS}{\overline{\mathcal{S}}}

\newcommand{\Minkout}[1]{\mathcal{M}^{{\rm out},#1}}
\newcommand{\uMinkout}[1]{\ov{\mathcal{M}}^{{\rm out},#1}}
\newcommand{\lMinkout}[1]{\underline{\mathcal{M}}^{{\rm out},#1}}

\title[Support measures in fractal geometry]{Support measures in fractal geometry}

\author[G.\ Radunovi\'c]{Goran Radunovi\'c}
\address{University of Zagreb, Faculty of Science, Department of Mathematics, Bijenička cesta 30, 10000 Zagreb, Croatia.}
\email{\href{mailto:goran.radunovic@math.hr}{goran.radunovic@math.pmf.unizg.hr}}

\author[S.\ Winter]{Steffen Winter}
\address{Karlsruhe Institute of Technology, Department of Mathematics, Englerstr. 2, 76131 Karlsruhe, Germany.}
\email{\href{mailto:steffen.winter@kit.edu}{steffen.winter@kit.edu}}

\date{\today}

\begin{document}


\thanks{Goran Radunovi\' c is supported by the Croatian Science Foundation grant IP-2022-10-9820, IP-2025-02-1147 and Horizon grant  101183111-DSYREKI-HORIZON-MSCA-2023-SE-01. This paper was supported by the European Union – NextGenerationEU through the National Recovery and Resilience Plan 2021-2026. Institutional grant of University of Zagreb Faculty of Science (K IA 1.1.3. Impact4Math).
Steffen Winter acknowledges funding from DFG grant 433621248.}

\begin{abstract}
We introduce two novel families of geometric functionals---basic contents and support contents---for investigating the fractal properties of compact subsets in Euclidean space. These functionals are derived from the support measures arising in connection with the general Steiner formula due to Hug, Last, and Weil, and offer new tools for extracting geometric information beyond classical fractal dimensions. The basic contents are constructed from the support measures of the set itself, while the support contents arise from those of its parallel sets. Associated scaling exponents characterize the asymptotic behavior of these measures as the resolution parameter tends to zero. We establish a fundamental connection between the maximum of the basic scaling exponents and the outer Minkowski dimension.
Furthermore, we explore how support contents supply aggregated information on geometric features that is provided separately by the basic contents, and how they relate to fractal curvatures and complex dimensions. Our results provide a deeper understanding of which geometric aspects actually contribute to fractal invariants such as the (outer) Minkowski content. We provide some illustrative examples.

\end{abstract}

\maketitle

{\small \tableofcontents}

\section{Introduction}

In this paper we use the support measures of \cite{HugLasWeil} to study the fractal properties of general compact sets in $\R^d$. We will introduce two families of geometric quantities associated to compact sets and called \emph{basic contents} and \emph{support contents},
which we propose as tools to extract geometric information encoded in the general Steiner formula of \cite{HugLasWeil}, such as the fractal dimension (but also much more).

Fractal dimensions provide only very rough geometric information about a set. In the worst case, the assigned  value is an integer, as for the devil's staircase, i.e., the graph of the Cantor function, having all common fractal dimensions equal to 1. In such situations the dimension is not even able to distinguish fractals from `classical sets'. In any case, it is evident that a single numerical value is insufficient to capture all relevant geometric information. Accordingly, there is a need to identify additional quantities that convey essential geometric information about a given (fractal) set. Various functionals of this kind have been proposed and studied, with some proving to be particularly useful.
The  Minkowski content, proposed as a measure of `lacunarity' by Mandelbrot \cite{Ma94-lacun}, has found many applications see \cite{Ba13,DEMES13,DiIeva24}.
Beside Minkowski contents several (mostly non-rigorous) concepts have been suggested and used in the applied mathematics literature to extract further geometric information from fractal structures, see e.g.\ \cite{bruno2012shape,CaGa13,DEMES13},
showing the urgent need for further geometric descriptors.

Fractal curvature measures, based on singular curvature theory and defined via approximation by parallel sets (i.e., $\e$-neighborhoods) provide a rigorous framework for further geometric invariants \cite{Wi08,RZ19,spodarev2015estimation}.
The theory of complex dimensions \cite{FZF,lapidus2018fractal,lapidus2017distance,LF12} generalizes  Minkowski dimension and content in a different way. It encodes the geometry of a set by assigning to it a certain set of `dimensions' which are complex numbers defined as the poles of appropriate fractal zeta functions. This theory has proved useful in number theory \cite{LF12,LaPom,Lap08,LapHer21}  and e.g.\ for the detection of the formal and analytic class of dynamical systems \cite{MRR2022,klimes2021reading}, for which also some complexified generalizations of the Minkowski content play a role \cite{Res2013,Res2014}.
Another field of possible application of such geometric functionals is spectral theory. The spectral properties of operators heavily depend on the geometry of their domains. For classic domains (like smooth manifolds or domains with smooth boundary) curvature is known to play an essential role, see e.g.\ \cite{VAN_DEN_BERG-1994,MCKEAN-1967}. But what are relevant geometric quantities that capture this dependence for domains with fractal boundary? Counterexamples to the modified Weyl-Berry conjecture have made clear that it is not the Minkowski content alone, cf.~\cite{CaBr86,LaPo96}.

All the mentioned approaches have in common that they study, how a (fractal) set is embedded in the ambient space $\R^d$. The results in this paper are also along these lines and so, naturally, one should expect relations to notions such as Minkowski content, fractal curvatures and complex dimensions.
We introduce new geometric functionals to fractal geometry which are derived from the support measures of Hug, Last and Weil \cite{HugLasWeil}. Based on earlier work of Stach\'o \cite{Stacho79}, in \cite{HugLasWeil}, a general Steiner-type formula was obtained for arbitrary closed sets in $\R^d$. It expresses the (local) parallel volume of a closed set in terms of its support measures, which are signed \begin{new}reach\end{new} measures living on the generalized normal bundle of the set. They may be regarded as non-additive generalizations of curvature measures; see Section \ref{sec:pre} for details. 
 Support measures encode a lot of geometric information, but unfortunately this information is not so easily accessible in the case of fractal sets, e.g.\ it is not clear how to read off the Minkowski dimension of the set. Although support measures have various applications, e.g.\ in stochastic geometry \cite{HugLasWeil,HugLast00,Vi25} or geometric measure theory \cite{HugSan2022}, up to now they have not seriously been used in fractal geometry.

Our aim is to utilize the support measures to study fractal properties of compact sets. We will introduce two families of geometric functionals, which we call \emph{basic contents} and \emph{support contents}, along with associated scaling exponents.
The first sequence is based on the support measures of the set itself, while the second one will be based on the support measures of its parallel sets. 

The general strategy is to take only geometric features into account up to some detail level $\e>0$ and then let $\e$ tend to $0$, revealing more and more of the geometric information about a given compact set $A\subset\R^d$.
\begin{new}
For the basic contents this is achieved by considering
only those pairs in the generalized normal bundle $N(A)$
whose local reach $\delta(A,\cdot,\cdot)$ is larger than $\e$.
Writing $\mu_i(A;\cdot)$, $i=0,\ldots,d-1$, for the support measures of $A$ and letting
$$
N_{>\e}(A):=
\{(x,u)\in N(A) :\dl(A,x,u)>\e\},
$$
we study the $i$-th basic function
$\fb_i(A;\e):=\mu_i(A;N_{>\e}(A))$
and its total variation analog
$\fbv_i(A;\e):=|\mu_i|(A;N_{>\e}(A))$.
The upper basic scaling exponent $\ubex_i(A)$ describes
the growth of $\fbv_i(A;\e)$ as $\e\searrow 0$:
it is the critical value of $q$ above which
$\e^{q-i}\fbv_i(A;\e)$ tends to zero.
Corresponding lower exponents $\lbex_i(A)$ are also considered. 
Limits of the rescaled
basic functions $\e^{q-i}\fb_i(A;\e)$ as $\e\searrow 0$, when they exist, give the basic
contents $\BC_i^q(A)$; see
Definition~\ref{def:reach-measure} for details.
\end{new}

Our main results about the basic scaling exponents are summarized in Theorem~\ref{thm:basic-exp}. In particular, we determine the possible range for each $\bex_i$ (namely, $i\leq\lbex_i(A)\leq\ubex_i(A)\leq d$, provided $\mu_i(A;\cdot)\not\equiv 0$). Moreover, we establish a general and nontrivial connection to the \emph{outer Minkowski dimension} which measures the
asymptotic growth of the volume of
$A_\e\setminus A$ as $\e\searrow 0$; see
\eqref{eq:dimout-M-def} for its definition.
Namely, for any compact set $A\subset\R^d$, the (upper) outer Minkowski dimension is given by the maximum of the (upper) basic scaling exponents, that is,
$$
\udimout_M A =\max_{0\leq i\leq d-1} \ubex_i(A).
$$
The outer Minkowski dimension turns out to be just the right notion here, allowing the most general formulation.
The analysis of basic exponents is completed by a discussion of their scaling properties, see Propostion~\ref{prop:scale-basic}.

The second family of geometric functionals, the \emph{support contents}, arise from looking at the $\e$-parallel sets $A_{\e}:=\{x\in\R^d:\dist(x,A)\leq\e\}$, $\e> 0$, 
of the given $A\subset\R^d$. The support measures of $A_\e$ are also well-defined for any $\e>0$ and connected to the support measures of $A$ via a Steiner-type formula, which we review in Section~\ref{sec:steiner} (see in particular \eqref{eq:mu-parallel-local}). 
\begin{new}
The scaling behavior of the total mass 
$|\mu_i|(A_\e)$ of the total variation measure of $\mu_i(A_\e,\cdot)$ as $\e\searrow 0$ leads to the
(lower and upper) support scaling exponents
$\lsex_i(A)$ and $\usex_i(A)$.
In particular, $\usex_i(A)$ is the critical value of
$q$ above which $\e^{q-i}|\mu_i|(A_\e)$
tends to zero.
The corresponding limits of the rescaled signed total masses $\e^{q-i}\mu_i(A_\e)$,
when they exist, are the \emph{support contents} $\SC^q_i(A)$;
see Definition~\ref{def:support-content}.

Theorem~\ref{thm:support-exp} shows that the support scaling exponents
provide aggregated information compared to the basic
exponents. More precisely, for any compact set
$A\subset\R^d$ and $i=0,\ldots,d-1$,
\[
\usex_i(A)=\max_{0\leq j\leq i}\ubex_j(A).
\]
Thus the upper support
scaling exponents form a
nondecreasing sequence, with the highest-order exponent
$\usex_{d-1}(A)$ being the largest and being equal to the upper outer Minkowski dimension of $A$. We also derive some relations between the lower exponents but they are weaker in general.
\end{new}

\begin{new}
    If the basic contents of a common dimension $q<d$
exist and are finite for all orders, then they determine
all the corresponding support contents and the
outer Minkowski content  through explicit
linear relations; see Theorem~\ref{thm:SC-BC-rel}.
In particular, when $q$ is the outer Minkowski dimension of $A$ and $\Minkout{q}(A)$ the corresponding outer Minkowski content, then 
\begin{align*}
    \Minkout{q}(A)=\frac{1}{d-q}\sum_{j=0}^{d-1}\omega_{d-j}\mathcal{M}_j^q(A),
\end{align*}
where, for $j=0,\ldots,d-1$, the positive constant $\omega_{d-j}$ is defined in
\eqref{eq:kappa-and-omega} below.
Hence, these relations identify which support measures
contribute to the leading term of the outer parallel volume.
\end{new}

This can help to understand which `geometric features' of a fractal set do actually determine its Minkowski dimension or contribute to its Minkowski content. For instance, depending on the relation between the basic exponents $\bex_0(A)$ and $\bex_1(A)$ of a subset $A\subset\R^2$, either the 0-dimensional or the 1-dimensional `features' of the set (or both) may determine the outer Minkowski dimension and content; see Remark~\ref{rem:frac-ex} and the examples in Section~\ref{sec:ex} for such interpretation of our results.

\begin{new}
The support contents are closely related to two other
established approaches in fractal geometry.
For the highest order $i=d-1$, the support measure
of a parallel set is directly related to its surface area.
It turns out that the upper and lower support contents of this order
coincide, up to an explicit normalization constant, with
the corresponding upper and lower \emph{$S$-contents} introduced in \cite{RaWi2010}, see also 
\eqref{eq:S-content}. 
Thus our approach includes and recovers the classical
surface-area scaling information related with $S$-dimensions and $S$-contents, see Proposition~\ref{prop:S-content}.

For the remaining orders, support contents can be compared
with fractal curvatures \cite{RZ19, Wi08,Zahle2011}, which are defined using
curvature measures of parallel sets and appropriate
essential limits.
An advantage of support measures is that they are defined
for every positive parallel radius, whereas curvature
measures need not be available at all radii.
On the other hand, the two approaches need not capture
the same geometric information.
For instance, the Sierpi\'nski gasket has zeroth support
scaling exponent equal to zero, while its zeroth curvature
scaling exponent equals its Minkowski dimension
$\log_2 3$.
We discuss these distinctions in
Section~\ref{sec:support-funct}.    
\end{new}

There are also deep connections with the theory of complex dimensions and fractal tube formulas \cite{LF12,FZF}, which are established in a follow-up paper \cite{RaWi2}. \begin{new}Here we only give a brief outlook at the end,
see Section~\ref{page:zeta}.\end{new}
The special case of subsets of $\R$ is addressed in Section \ref{sec:R1}. Originally, the general Steiner formula of \cite{HugLasWeil} is only stated for subsets of $\R^d$ for $d\geq 2$, but it can be extended to the case $d=1$; see Proposition \ref{thm:stein1}. It allows to completely characterize the differentiability of the parallel volume, see Corollary \ref{cor:diffR1}, and to establish a direct connection to the theory of fractal strings of \cite{LF12}. In essence, the geometric counting function of a fractal string may be viewed as the basic function of the associated set.

We remark that recently support measures have been generalized for sets in non-Euclidean spaces in \cite{HugSan2022}. The scaling exponents and contents discussed here naturally generalize to such setting, but we restrict to the Euclidean case here.
Specifically, the support measures $\mu_i(A;\cdot)$ can be restricted in their second argument to suitable subsets of the normal bundle, thereby enabling the investigation of local geometric properties of the set $A$, in the spirit of \cite{WiZa13, Wi19}. 

{\bf Outline.} \begin{new}The structure of the paper is as follows. In Section~\ref{sec:pre} we recall the definition and some properties of support measures. In Section \ref{sec:basic}, basic scaling exponents and basic contents are introduced and some of their properties are stated, including one of the main results: Theorem~\ref{thm:basic-exp}. Section \ref{sec:4} is devoted to its proof. Moreover, some further properties are stated, including the scaling behavior of basic exponents and contents.
 In Section \ref{sec:steiner}, the Steiner formula for support measures is reviewed and some auxiliary results are derived from it, which prepare the discussion of support contents in Section \ref{sec:support-funct}. In the latter, support scaling exponents and support contents are introduced and their connections to other fractal functionals are explored, including basic scaling exponents and basic contents, outer Minkowski dimension and content, 
 as well as $S$-dimension and $S$-content.
 At the end of this section, the connections to fractal curvatures are addressed.
Section \ref{sec:ex} is devoted to examples, demonstrating how 
basic (and support) scaling exponents are computed and interpreted. In particular it is shown that the basic exponents are able to distinguish different geometric behavior. 
In Section~\ref{sec:R1}, subsets of $\R$ are discussed.
We conclude with an outlook on the connection between
the notions discussed here and the theory of fractal zeta
functions, see Section \ref{page:zeta}.
\end{new}

\section{Preliminaries}
\label{sec:pre}

Let $d\in\N$, $d\geq 2$ and let $I_d:=\{0,\ldots,d-1\}$ denote the standard set of indices that we will often use in this paper.
For any $k\in\N$, let
\begin{align}
  \label{eq:kappa-and-omega}
  \kappa_k&:=\frac{\pi^{k/2}}{\Gamma(k/2+1)}  &&\text{ and } & \o_k&:=k\kappa_k
\end{align}
denote the volume and the surface area, respectively, of the $k$-dimensional unit ball in $\R^k$. Here $\Gamma$ is Euler's gamma function.
We write $|\cdot|$ for the Euclidean norm and $\dist(x,A):=\inf\{|x-a|:a\in A\}$ for the (Euclidean) distance of a point $x\in\R^d$ to a set $A\subseteq\R^d$. We denote by $\eS^{d-1}:=\{y\in\R^d:|y|=1\}$ the unit sphere in $\R^d$.

For any compact set $A\stq\R^d$ and $\eps>0$, let
$$
A_\e:=\{x\in\R^d:\dist(x,A)\leq\e\}
$$
denote the (closed) $\e$-\emph{parallel set} or $\e$-\emph{neighborhood} of $A$.

We recall the general Steiner formula from Hug, Last and Weil \cite{HugLasWeil}, which allows, in particular, to express the parallel volume of any closed set $A\subset\R^d$. It is based on earlier work of Stach\'o \cite{Stacho79}. For this purpose let $A\subset\R^d$ be a closed set and let $\Unp(A)$ be the set of all points $y\in \R^d\setminus A$ which have a unique nearest point in $A$. Let $\pi_A(y)$ denote this nearest point. It is called the \emph{metric projection} of $y$ to $A$. Also, the mapping $\pi_A:\Unp(A)\to A$, $y\mapsto \pi_A(y)$ is called \emph{metric projection} onto $A$. The \emph{generalized normal bundle} of $A$ is the set
\begin{align}
  \label{eq:normal-bundle}
  N(A):=\left\{\left(\pi_A(y),\tfrac{y-\pi_A(y)}{\dist(y,A)}\right):y\in\Unp(A)\setminus A\right\}.
\end{align}
Clearly, $N(A)\subset \pa A\times \eS^{d-1}$, that is, any pair $(x,u)\in N(A)$ consists of a foot point $x\in \pa A$ and a unit vector $u\in\eS^{d-1}$ pointing `away' from $A$.
We also define the \emph{generalized metric projection} of $A$ as the mapping
$$
\Pi_A:\Unp(A)\setminus A\to N(A),\quad y\mapsto \left(\pi_A(y),\tfrac{y-\pi_A(y)}{\dist(y,A)}\right).
$$

 The \emph{reach function} (or \emph{local reach}) $\dl(A, \cdot):N(A)\to[0,\infty]$ of $A$ is defined by
 \begin{align} \label{eq:reach-def}
    \dl(A,x,u):=\sup\{t\geq 0: \pi_A(x+tu)=x\}.
 \end{align}
It may be extended to all of $\R^d \times \eS^{d-1} $ by setting $\dl(A, x, u) := 0$ for $(x, u) \notin N(A)$. Note that $\dl(A, \cdot) > 0$ on $N(A)$ and that $\dl(A,\cdot)$ is measurable, see \cite[Lemma 6.2]{HugLasWeil}. If $A$ is convex, then $\dl(A,\cdot)\equiv \infty$.  The \emph{reach} of $A$ is the number $\reach(A):=\inf\{\dl(A,x,u):(x,u)\in N(A)\}$ and $A$ is called a set of \emph{positive reach}, if $\reach(A)>0$.

For any signed measure $\mu$, we denote by $\mu^+$, $\mu^-$ and $|\mu|:=\mu^+ +\mu^-$, respectively, the positive, negative and total variation measure of $\mu$.

\begin{theorem}[General Steiner formula {\cite[Thm.~2.1]{HugLasWeil}}]
	\label{thm:generalSteiner}
    \begin{new}
        For any nonempty closed set $A\subset\R^d$, there exist
reach measures $\mu_0(A;\cdot),\ldots,\mu_{d-1}(A;\cdot)$
of $A$, called the \emph{support measures} of $A$, satisfying
    \end{new}
  \begin{align}
   \int_{N(A)}\1\{x\in B\}\min\{r,\dl(A,x,u)\}^{d-i}|\mu_i|(A;\d(x,u))<\infty,
  \end{align}
  for $i=0,\ldots,d-1$, any compact set $B\subset\R^d$ and any $r>0$, such  that,
  such that, for any bounded measurable function $f:\R^d\to\R$
with compact support,
   \begin{align}
	\label{eq:generalSteiner}
\int_{\R^d\setminus A} f(x) \d x &
=\sum_{i=0}^{d-1}\o_{d-i}\int_0^{\infty}
\int_{N(A)}t^{d-i-1}\1{\{t<\dl(A,x,u)\}}
f(x+tu)\mu_i(A;\d(x,u))\d t.
\end{align}
\end{theorem}

\begin{new}
Here support measures are understood as reach measures
in the sense of \cite[\S2.1, p.~240]{HugLasWeil}.
In particular, for every compact $B\subset\R^d$ and every
$t>0$, their restrictions to
\[
\{(x,u)\in N(A):x\in B,\ \dl(A,x,u)\geq t\}
\]
are finite signed measures.
Their positive and negative parts, and hence their total
variation measures, are defined on all of $N(A)$.
Note that signed integrals and identities below are understood
where the corresponding expressions are defined;
in particular, the ambiguous difference $\infty-\infty$ is never used.
\end{new}
For compact $A$, choosing $f=\1_{A_\e}$ in
\eqref{eq:generalSteiner},
one gets in particular an expression for the $\e$-parallel volume of $A$ for any $\e>0$, see also \cite[\S4.3]{HugLasWeil},
\begin{equation}
	\label{eq:v-par}
	V(A_\e\setminus A)=\sum_{i=0}^{d-1}\o_{d-i}\int_0^{\e}t^{d-i-1}\int_{N(A)}\1{\{t<\dl(A,x,u)\}}\mu_i(A;\d(x,u))\d t.
\end{equation}
Here and throughout $V(C):=\mathscr{L}^d(C)$ denotes the $d$-dimensional volume or Lebesgue measure of a set $C\subset\R^d$.

Recall from \cite[Corollary 2.5]{HugLasWeil} that the \emph{$i$-th support measure} of a closed set $A\subset\R^d$ is given explicitly by
\begin{equation}\label{eq:support-formula}
	\mu_i(A;\cdot)=\frac{1}{\omega_{d-i}}\int_{N(A)}\1{\{(x,u)\in\cdot\}}H_{d-1-i}(A,x,u)\sH^{d-1}(\d (x,u)),
\end{equation}
where $\sH^{d-1}$ is the $(d-1)$-dimensional Hausdorff measure on the (generalized) normal bundle $N(A)$.\footnote{We use the convention that the Hausdorff measure is appropriately normalized so that $\mathscr{H}^k([0,1]^k)=1$ for any integer $k$.}
Furthermore, for any $j\in I_d$, the elementary symmetric function $H_{j}$ of $A$ is defined (for $\sH^{d-1}$-almost all $(x,u)\in N(A)$) by
\begin{align} \label{eq:Hj}
    H_j(A,x,u):=\frac{\sum_{|I|=j}\prod_{l\in I}k_l(A,x,u)}{\prod_{i=1}^{d-1}\sqrt{1+k_i(A,x,u)^2}},
\end{align}
see e.g.\ \cite[Eq.\ (2.13)]{HugLasWeil}, where $k_1(A,x,u),\ldots,k_{d-1}(A,x,u)$ denote the {\em generalized principal curvatures} of $A$ at $(x,u)\in N(A)$.\footnote{At an infinite principal curvature, the normalized
factors in \eqref{eq:Hj} are interpreted by continuity:
$k/\sqrt{1+k^2}\to1$ and $1/\sqrt{1+k^2}\to0$
as $k\to+\infty$.}
They were initially defined for sets of positive reach in \cite{zahle86} (see also \cite[\S 4.4]{RZ19}) and then the definition was extended in \cite{HugLasWeil} to arbitrary nonempty closed subsets of $A\subset \R^d$ (for which they are well-defined for $\sH^{d-1}$-a.a.\ $(x,u)\in N(A)$).

We recall some properties of the support measures, see also \cite[\S 4.4]{HugLasWeil}. These measures are \emph{locally defined}, meaning that, whenever $A_1\cap U=A_2\cap U$ for some closed sets $A_1,A_2\subset\R^d$ and some open set $U\subset\R^d$, then
\begin{align} \label{eq:loc-defined}
  \mu_i(A_1;B)=\mu_i(A_2;B)
\end{align}
holds for all Borel sets $B\subset U\times \eS^{d-1}$ and any $i\in I_d$.
Moreover, support measures are \emph{motion covariant}: for any $i\in I_d$, any closed set $A\subset\R^d$ and any rigid motion $g\in\R^d$  with rotational part $\hat g\in SO_d$ and translational part  $b\in\R^d$ (such that $g(x)=\hat g(x)+b$, $x\in\R^d$),
\begin{align} \label{eq:motion-cov}
  \mu_i(g(A); g(B))=\mu_i(A;B)
\end{align}
holds for any Borel set $B\subset \R^d\times \eS^{d-1}$, where $g(B):=\{(g(x),\hat g(u)): (x,u)\in B\}$. Moreover, for any $j\in I_d$ the $j$-th support measure is \emph{homogeneous of degree $j$}: for any closed set $A\subset\R^d$ and any scaling factor $\lambda>0$,
 \begin{align} \label{eq:homogen}
  \mu_j(\lambda A; \lambda B)=\lambda^j \mu_j(A;B),
\end{align}
holds for any Borel set $B\subset \R^d\times \eS^{d-1}$, where $\lambda B:=\{(\lambda x,u): (x,u)\in B\}$.

\begin{new}
Let us briefly recall some connections with 
curvature measures. Let $C_0(A,\cdot),$ $ \ldots,C_{d-1}(A,\cdot)$ denote the
\emph{curvature-direction measures} of a compact set $A\subset\R^d$ on
$\R^d\times\eS^{d-1}$, which we will sometimes call curvature measures
for brevity. Their projections onto $\R^d$ are the usual
curvature measures; see \cite[\S4.5 and Definition~5.14]{RZ19}.
For any compact set $A\subset\R^d$ with positive reach
(including any compact convex set), the measures
$C_0(A,\cdot),\ldots ,C_{d-1}(A,\cdot)$ are the directional
extensions of Federer's curvature measures \cite{Fed59},
and they coincide with the corresponding support measures,
i.e., $C_i(A,\cdot)=\mu_i(A;\cdot)$ for $i\in I_d$. \label{page:curv}
\end{new}

 Thus, support measures are a generalization of Federer's curvature measures to arbitrary closed sets. Note that they are a \emph{non-additive} extension of these curvature measures. Other (additive) extensions exist, in particular to the class of UPR-sets. It consists of sets which can be represented as locally finite unions of sets with positive reach such that their finite intersections also have positive reach. A subclass is the extended convex ring ECR, which consists of all subsets of $\R^d$ that can locally be represented by finite unions of compact convex sets. Curvature measures can be extended additively to UPR-sets, meaning that these measures are well-defined for any such set and are additive:
\begin{align} \label{eq:curv:additive}
  C_i(A_1\cup A_2, \cdot)=C_i(A_1,\cdot)+C_i(A_2,\cdot)-C_i(A_1\cap A_2,\cdot)
\end{align}
holds, whenever $A_1,A_2,A_1\cup A_2,A_1\cap A_2$ are UPR-sets, see \cite[\S 5.2 and in particular Cor.~5.15]{RZ19}.

Also curvature-direction measures live on a subset of $\partial A\times \eS^{d-1}$ but in general their support may be larger than $N(A)$. If $A\subset\R^d$ has positive reach, then $\mu_i(A;\cdot)=C_i(A,\cdot)$ for any $i\in I_d$. The general relation between these measures is as follows. For any UPR-set $A\subset\R^d$ and any $i\in I_d$,
\begin{align} \label{eq:mu-C-rel}
  \mu_i(A;\cdot)=C_i(A;N(A)\cap\cdot).
\end{align}
For $A$ in the extended convex ring this is stated in ~\cite[eq.~(3.1)]{HugLasWeil}, see also \cite[Thm.~3.3]{HugLast00}. In general, it follows by comparing the integral representations of $\mu_i(A;\cdot)$ (see \eqref{eq:support-formula}) and $C_i(A,\cdot)$ (see e.g.~\cite[Thm.~5.18]{RZ19}) and noting that the index function appearing in the latter representation equals 1 for almost all $(x,n)\in N(A)$, cf.~\cite[Thm.~5.6]{RZ19}.
Hence, in general, curvature-direction measures may contain strictly more geometric information about a set $A$ than support measures. This is partially compensated by the information encoded in the reach function appearing in formulas involving support measures.

Note that curvature-direction measures (whenever defined) share with the corresponding support measures the property of being locally defined, the motion covariance and the homogeneity, cf.~\eqref{eq:loc-defined}, \eqref{eq:motion-cov} and \eqref{eq:homogen}.
For any compact set $A\subset\R^d$ admitting curvature measures, the total masses
\begin{align}
  \label{eq:total-curv}
  C_i(A):=C_i(A,\R^d\times\eS^{d-1})
\end{align} of the curvature measures $C_i(A,\cdot)$, $i\in I_d$, are known as the \emph{total (Lipschitz-Killing) curvatures} of $A$ or, in convex geometry, as \emph{intrinsic volumes}. For any compact set $A\subset\R^d$ with positive reach, one has $C_i(A)=\mu_i(A;\R^d\times\eS^{d-1})$.

\section{Basic content and basic scaling exponents}
\label{sec:basic}

In this section, we introduce the notions of \emph{basic contents} and the associated \emph{basic scaling exponents} for an arbitrary compact set in $\R^d$. They are defined in terms of the \emph{basic functions}, which we discuss first. Then our focus will be on the properties of the basic exponents, which are summarized in Theorem~\ref{thm:basic-exp} and Proposition~\ref{prop:scale-basic}.
Throughout this section---and the remainder of the paper---$A$ will denote a nonempty compact subset of $\R^d$.

\begin{proposition}[Basic functions of a compact set] \label{def:basic}
   For any $i\in I_d$, we define the {\em $i$-th  basic function} of a nonempty compact set $A\subset\R^d$ by
\begin{equation}
	\label{eq:m_i}
	{\fb}_i(t):=\fb_i(A;t):=\int_{N(A)}\1{\{t<\dl(A,x,u)\}}\mu_i(A;\d(x,u)), \quad t>0,
\end{equation}
and we denote by $\fbv_i$ its total variation analog, i.e.,
\begin{equation}
	\label{eq:m_i-tot}
	\fbv_i(t):=\fbv_i(A;t):=\int_{N(A)}\1{\{t<\dl(A,x,u)\}}|\mu_i|(A;\d(x,u)), \quad t>0.
\end{equation}
	\label{lem:parallel-supp-finite}
	For each $i\in I_d$, 
$\fb_i$ and $\fbv_i$ are real-valued functions, i.e., $|\fbv_i(t)|<\infty$ and $|\fb_i(t)|<\infty$ for any  $t>0$.
	Furthermore, $\fbv_i$ is nonnegative, monotonically decreasing and right continuous.
\end{proposition}
 \begin{proof}
   Let $i\in I_d$. The nonnegativity and monotonicity of $\fbv_i$ are immediate from its definition.
	The finiteness of $\fbv_i(t)$ for $t>0$ can be deduced from  \cite[Thm.\ 2.1, eq.\ (2.2)]{HugLasWeil}: Setting $r=t$ and $B=A$ in that theorem (which is possible since $A$ is assumed to be compact), one gets 	
    \begin{equation}
	    \label{eq:integrability}
		\int_{N(A)}\min\{t,\dl(A,x,u)\}^{d-i}|\mu_i|(A;\d (x,u))<\iy.
	\end{equation}
	This integral remains finite if the integrand is multiplied by $\1\{t<\dl(A,x,u)\}$.
 Observe that $\1{\{t<\dl(A,x,u)\}}\neq0$ implies $\min\{t,\dl(A,x,u)\}=t$ so that we obtain
	\begin{equation}
		t^{d-i}\int_{N(A)}\1{\{t<\dl(A,x,u)\}}|\mu_i|(A;\d (x,u))<\iy,
	\end{equation}
	from which the finiteness of $\fbv_i(t)$, and hence of $\fb_i(t)$ follows.

	To show continuity from the right we fix $t_0>0$ and choose any decreasing sequence $(t_n)_{n\geq 1}$ in $(t_0,+\infty)$ such that $t_n\searrow t_0$. For any $n\in\N_0$ let $B_n:=\{(x,u)\in N(A):\dl(A,x,u)>t_n\}$.
	Then $B_n\stq B_{n+1}$ and $B_0=\bigcup_{n\geq 1}B_n$. By the monotonicity of the positive measure $|\mu_i|(A;\cdot)$, it follows that $\fbv_i(t_n)=|\mu_i|(A;B_n)\to|\mu_i|(A;B_0)=\fbv_i(t_0)$ as $n\to\infty$.
\end{proof}
\begin{remark}\label{rem:pm-basic}
    Instead of $\fbv_i$ one can also consider the corresponding positive and negative variation analogs of $\fb_i$ and denote them by $\fb_i^{\pm}$, respectively.
    We will refer to them later when needed. Note that all of the definitions and most of the results involving $\fbv_i$ can be appropriately reformulated in terms of $\fb_i^{\pm}$. In particular, the properties of $\fbv_i$ in Proposition~\ref{def:basic} hold equally for $\fb_i^\pm$.
\end{remark}

\begin{remark}
	In general, the function $\fbv_i$ fails to be continuous from the left, since the set $B_t:=\{(x,u)\in N(A):\dl(A,x,u)=t\}$ might be of positive $|\mu_i|$-measure for some $t>0$.
	\begin{new}
	    As an example, let $A\subset\R^2$ be the union of the two
parallel segments $[0,\ell]\times\{0\}$ and
$[0,\ell]\times\{2\}$, where $\ell>0$.
Then $B_1$ consists of the upward normal pairs on the
lower segment and the downward normal pairs on the
upper segment.
Each oriented side contributes half its length to $\mu_1$,
while the endpoint normals have zero $\mu_1$-mass.
It is easy to see that
$\fbv_1(1)=\fb_1(1)=\ell$, while
$\fbv_1(t)=\fb_1(t)=2\ell$ for any $0<t<1$.
Hence the left limit $\fbv_1(1-)$ exists and equals $2\ell$.
	\end{new}
  
\end{remark}

\begin{remark}
	By \cite[Prop.\ 4.1]{HugLasWeil}, the support measure $\mu_{d-1}$ is a non-negative $\sigma$-finite measure on $\R^d\times\eS^{d-1}$. This implies that the corresponding basic  function $\fb_{d-1}$ is always nonnegative and $\fbv_{d-1}=\fb_{d-1}$.
\end{remark}

\begin{remark}
    The integrals that we call here the basic functions of a compact set already appear implicitly in \cite[\S 4.3]{HugLasWeil}, where also their continuity from the right is shown. Moreover, it is shown there that  their continuity at a point $t_0$ implies the existence of the derivative of the function $t\mapsto V(A_t)$ at $t_0$; see \cite[Cor.\ 4.5]{HugLasWeil}.
    For random compact sets $A$, also the corresponding expectations of such functions appear in \cite{HugLasWeil} and in \cite{Last06}, where they are called \emph{mean curvature functions}.
\end{remark}

In this paper, we study, in particular, the limiting behaviour of the basic functions of a given set $A$ as $t\searrow 0$ in order to extract refined information on the fractal properties of $A$; similarly as is done in the classical definition of the Minkowski dimension via the parallel volume.
In general, one cannot expect the limit of $\fb_i(t)$ to exist when $t\searrow 0$. Therefore, one needs to introduce proper rescaling, and also possibly look at upper and lower limits.
Recall that, for $q\geq 0$, the \emph{$q$-dimensional upper outer Minkowski content} of a compact set $A\subset\R^d$ is defined by
\begin{equation}
        \label{eq:outer-Mink}
		\uMinkout{q}(A):=\limsup_{\eps\searrow 0}\eps^{q-d} V(A_\eps\setminus A),
	\end{equation}
and the \emph{lower outer Minkowski content} $\lMinkout{q}(A)$ similarly, by replacing $\limsup$ by $\liminf$. Moreover, the \emph{upper outer Minkowski dimension} is given by
\begin{equation} \label{eq:dimout-M-def}
    \udimout_M(A):=\inf\{q\geq 0: \uMinkout{q}(A)=0\}
\end{equation}
and a lower counterpart $\ldimout_M(A)$ similarly using $\lMinkout{q}(A)$.
Note that for compact sets $A\subset \R^d$ with $V(A)=0$, the outer Minkowski content and dimension coincide with the ordinary ones. However, for the questions studied here, the outer versions are just the right notion in general. The outer Minkowski content of dimension $d-1$ plays an important role in geometric measure theory as a notion of surface area, see \cite{AmCoVi08,Vi09}, and has recently been generalized to non-Euclidean settings \cite{LuVi16}.

The next definition mimics the one of the outer Minkowski content and dimension replacing the parallel volume by the just introduced basic functions.
\begin{definition}[Basic contents and exponents]
	\label{def:reach-measure}
	Let $i\in I_d$ and $q\in\R$.
	We define the {\em ($q$-dimensional) upper $i$-th basic content} of a compact set $A\subseteq\R^d$ by
	\begin{equation}
        \label{eq:upper-basic}
		\uBC_i^q(A):=\limsup_{t\searrow 0}t^{q-i}\fb_i(t),
	\end{equation}
	along with its total variation analog
	\begin{equation}
        \label{eq:upper-basic-var}
		\uBC_i^{\var,q}(A):=\limsup_{t\searrow 0}t^{q-i}\fbv_i(t),
	\end{equation}
	and also their {\em lower} counterparts, $\lBC_i^q(A)$ and $\lBC_i^{\var,q}(A)$, by replacing the upper limits in \eqref{eq:upper-basic} and \eqref{eq:upper-basic-var}, respectively, by lower limits.

	We also introduce the {\em upper $i$-th basic scaling exponent} $\ubex_{i}(A)$ of $A$ as
	\begin{equation}
		\label{eq:reach-exp}
		\ubex_i:=\ubex_{i}(A):=\inf\{q\in\R:\uBC_i^{\var,q}(A)=0\}=\sup\{q\in\R:\uBC_i^{\var,q}(A)=+\iy\},
	\end{equation}
	as well as its {\em lower} counterpart $\lbex_{i}=\lbex_i(A)$, using $\lBC_i^{\var,q}(A)$ in \eqref{eq:reach-exp} instead of $\uBC_i^{\var,q}(A)$.
    Furthermore, if $\ubex_i=\lbex_i$, then we denote the common
value by $\bex_i$ and call it the {\em $i$-th basic scaling
exponent} of $A$.
For any $q\in\R$, if $\uBC_i^q(A)=\lBC_i^q(A)$, we denote
the common value by $\BC_i^q(A)$ and call it the
{\em ($q$-dimensional) $i$-th basic content of $A$}.
Analogously, if $\uBC_i^{\var,q}(A)=\lBC_i^{\var,q}(A)$,
we denote the common value by $\BC_i^{\var,q}(A)$.
\end{definition}

If a set $A\subset\R^d$ is sufficiently regular and non-fractal, then one can expect that, for all $i\in I_d$, the $i$-th basic scaling exponent $\bex_i(A)$ exists and equals $i$, except when the support measure $\mu_i(A;\cdot)$ vanishes. In the latter situation, one has $\fb_i=\fb_i^{\var}\equiv 0$ and so it is consistent to define the basic exponent $\bex_i(A)$ to be $-\infty$. Furthermore, when $\bex_i(A)=i$, it is reasonable to expect that the corresponding $i$-th basic content $\BC^{i}_i(A)$ exists and is equal to the total mass of the measure $\mu_i(A;\cdot)$. In contrast, for fractals we expect that at least one of the basic exponents will exhibit a more intricate behavior. Before addressing such cases, we begin with a more precise statement concerning a class of well-behaved non-fractal sets, namely convex sets.

\begin{proposition}[Basic exponents of convex sets] \label{prop:convex}
    Let $K$ be a nonempty compact convex set in $\R^d$ and let $k\in\N$ be the affine dimension of $K$.
    Then, for each $i\in I_d$ with $i\leq k$, $\bex_i(K)=i$. Moreover, $\BC^i_i(K)$ exists and equals the intrinsic volume $C_i(K)$, cf.~\eqref{eq:total-curv}. Furthermore, for each $i\in I_d$ with $i>k$, $\mu_i(K,\cdot)\equiv 0$ and so $\bex_i(K)=-\infty$. Hence $\BC^q_i(K)=0$ for any $q\in\R$. (In particular, $\BC^i_i(K)=0$, in correspondence with $C_i(K)=0$.)
\end{proposition}
 \begin{proof}
 Fix $i\in I_d$.
 Recall that for the convex compact set $K\subset\R^d$, the reach is infinite, i.e., $\dl(K,\cdot)\equiv\infty$. Moreover, for any $i\in I_d$, $\mu_i(K;\cdot)$ is a nonnegative measure with total mass equal to $C_i(K)$, i.e., $\mu_i(K;N(K))=C_i(K)$. It is also well known that $C_i(K)=0$ for $i>k$. For the $i$-th basic function $\fb_i(K;\cdot)$ this implies
\begin{align*}
    \fb_i(K;t)=\int_{N(K)}\1{\{t<\dl(K,x,u)\}}\mu_i(K;\d(x,u))=\mu_i(K;N(K))=C_i(K).
\end{align*}
 Hence $\fb_i^{\var}(K;\cdot)=\fb_i(K;\cdot)\equiv C_i(K)>0$ for $i\leq k$ and $\fb_i(K;\cdot)\equiv 0$ for $i>k$, from which it is easy to deduce $\bex_i(K)=i$ and $\BC^i_i(K)=C_i(K)$ for $i\leq k$, and $\bex_i(K)=-\infty$ for $i>k$ as claimed.
\end{proof}

To help the reader to get familiar with these notions, we compute basic exponents and basic contents for a number of simple examples in Section~\ref{sec:ex} below.
Our main results about the basic scaling exponents are summarized in the following statement, which relates them in particular to the outer Minkowski dimension.
\begin{theorem}[Properties of Basic Exponents] \label{thm:basic-exp}
    Let $A\subset\R^d$ be a nonempty compact set.
 For each $i\in I_d$ one of the following is true:
    \begin{enumerate}
        \item[(a)]  $\mu_i(A;\cdot)\equiv 0$ and so by convention we let $\un{\frk{m}}_i(A)=\ov{\frk{m}}_i(A)=-\infty$;
        \item[(b)] $i\leq \lbex_i(A)\leq \ubex_i(A)$.
    \end{enumerate}
       It is always true that  $\mu_0(A;\cdot)\not\equiv 0$. Furthermore,
     \begin{equation}\label{eq:dim_max_m}
     \max_{i\in I_d} \ubex_i(A) =\udimout_M A \quad \text{ and } \quad  \max_{i\in I_d}\lbex_i(A)\leq \ldimout_M(A).
 \end{equation}
\end{theorem}

\begin{remark}
An immediate consequence of Theorem~\ref{thm:basic-exp} is that, for all indices $i$ such that $i>\ldimout_M A$, assertion (a) holds (meaning that $\mu_i(A;\cdot)\equiv 0$ and so $\lbex_i(A)=\ubex_i(A)=-\infty$).
  In \cite[Proposition 2.4]{Last06}, it is shown that if $\sH^{k}(\partial A)=0$ for some nonempty closed set $A$, then $\mu_k(A;\cdot)\equiv 0$, which is slightly stronger. Indeed, the Hausdorff dimension of $\partial A$ may be strictly smaller than the (lower) outer Minkowski dimension of $A$ and so $\mu_i(A;\cdot)$ may vanish for further indices $i$ below $\ldimout_M A$, namely for all $i>\lfloor{\dim}_H A\rfloor$. 
   Note, however, that a general upper bound for the basic scaling exponents is only provided by the (upper/lower) outer Minkowski dimension and not by the Hausdorff dimension, as equation \eqref{eq:dim_max_m} and Remark~\ref{rem:bex-sharpness} clarify.
\end{remark}

\begin{remark} \label{rem:bex-sharpness}
  For all indices $i\in I_d$ except $i=0$, the measures $\mu_i(A;\cdot)$ do indeed vanish for some sets $A$, that is, assertion (a) is possible for each $i\neq 0$. Moreover, the first inequality in (b) is optimal for all indices $i\in I_d$ in the sense that there exist sets $A$ for which equality holds. Examples for both phenomena are provided by compact convex sets $K\subset\R^d$ as discussed in Proposition~\ref{prop:convex}. Indeed, if the affine dimension of $K$ is at least $i$, then $\bex_i(K)=i$, and if it is below $i$, then $\bex_i(K)=-\infty$.

  Furthermore, if $K$ has affine dimension $k\leq d-1$,
  we have $\dimout_M K=\dim_M K=k=\max\{\bex_i(K):i\in I_d\}$. Hence such sets $K$ are trivial examples for the fact that $\ldimout_M K$ is a \emph{sharp} upper bound for all the lower basic scaling exponents $\lbex_i$ as stated in \eqref{eq:dim_max_m} (and, similarly, that $\udimout_M K$ is a sharp upper bound for all $\ubex_i$). In view of the first equality in \eqref{eq:dim_max_m}, one may ask, whether a corresponding equality holds for the lower basic scaling exponents. We believe that this is not the case. In general, $\ldimout_M A$ is only an upper bound for the the maximum of the $\lbex_i(A)$, but there is no equality in general. 
 \end{remark}

The introduction of the basic scaling exponents also enables us to improve the integrability result of \cite[eq.\ (2.2)]{HugLasWeil} which will be crucial in our follow-up paper \cite{RaWi2}. More precisely, in Lemma \ref{lem:integrability-better} below we show that the exponent $d$ in equation \eqref{eq:integrability} (and similarly in \cite[eq.~(2.2)]{HugLasWeil}) can be replaced by any real number strictly larger than $\ubex_i(A)$.

Before turning to the proof of Theorem \ref{thm:basic-exp} let us comment on the possible behavior of the basic exponents.
\begin{remark} \label{rem:frac-ex}
  Theorem \ref{thm:basic-exp} leaves the possibility that either of the basic exponents could be the largest and could thus determine the outer Minkowski dimension of the set. In $\R^2$, one has two exponents, $\bex_0$ and $\bex_1$. In Section~\ref{sec:ex} below we discuss in detail three (classes of) examples which illustrate that either of the three cases $\bex_0<\bex_1$, $\bex_0=\bex_1$ and $\bex_0>\bex_1$ is possible, see also Figure~\ref{fig:ex-joint}.
More precisely, for the Sierpi\'nski gasket we observe $\bex_0=0$ and $\bex_1=\dimout_M>1$, see Example~\ref{ex:SG}. It demonstrates that the measure $\mu_1$ can be responsible for the `fractality'.
But one can also have $\bex_0=\bex_1(=\dimout_M>1)$  (see the \emph{fractal windows} in Example~\ref{ex:FW}) and $1<\bex_1<\bex_0(=\dimout_M)$ (see the \emph{enclosed fractal dust} in Example \ref{ex:FD}). Hence the outer Minkowski dimension may also be determined by the support measure $\mu_0$ or by both measures $\mu_0$ and $\mu_1$.

In $\R^d$, $d>2$, we expect the same variability. Any of the basic exponents (or any subset of them) can be the largest and will thus determine $\udimout_M$.
\end{remark}

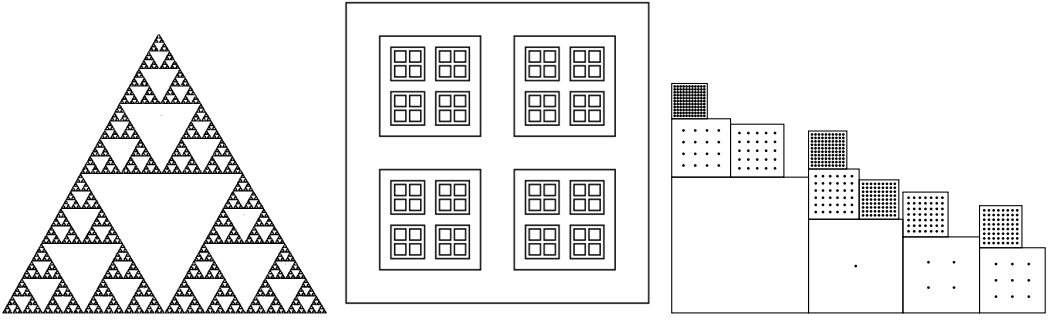
\begin{figure}[ht]
    \includegraphics[width=0.3\textwidth]{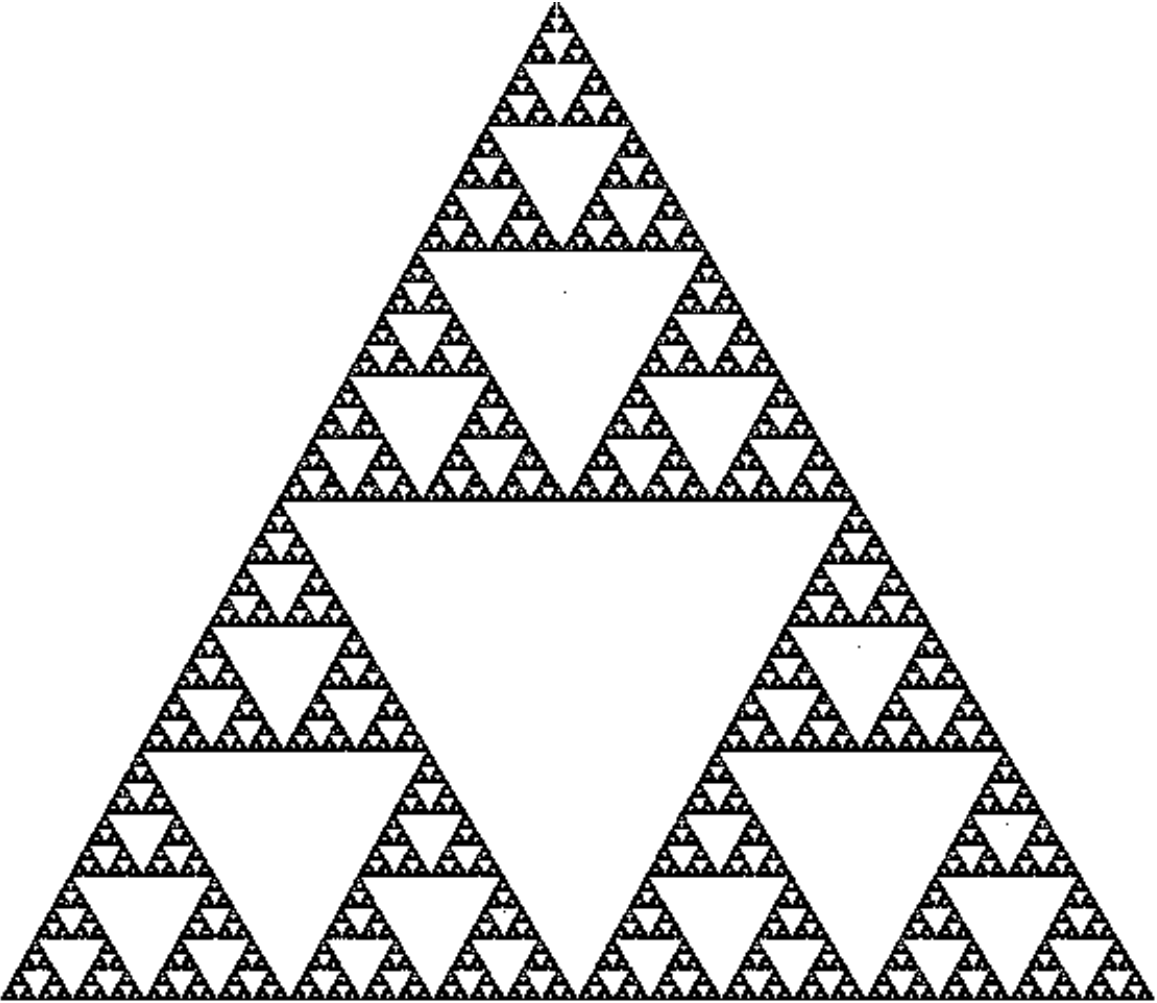}
    \includegraphics[width=0.3\textwidth]{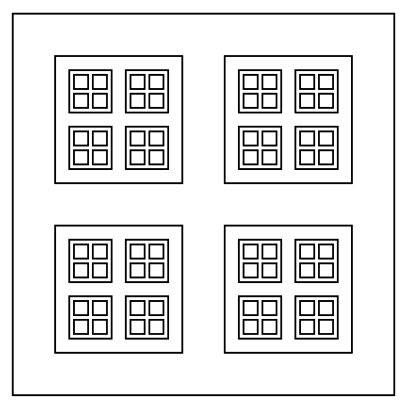}
    \begin{tikzpicture}[scale=1.8]
	\draw (0, 0) rectangle (1,1);
		\equidistantPoints{xshift=1cm}{0.69}{1}{3};
		\equidistantPoints{xshift=1cm+0.69cm}{0.56}{2}{3}
		\equidistantPoints{xshift=1cm+0.69cm+0.56cm}{0.48}{3}{3}
		\equidistantPoints{yshift=1cm}{0.43}{4}{3}
		\equidistantPoints{yshift=1cm,xshift=0.43cm}{0.39}{5}{3}
		\equidistantPoints{yshift=0.69cm,xshift=1cm}{0.37}{6}{3}
		\equidistantPoints{yshift=0.56cm,xshift=1cm+0.69cm}{0.33}{7}{3}
		\equidistantPoints{yshift=0.48cm,xshift=1cm+0.69cm+0.56cm}{0.31}{8}{3}
		\equidistantPoints{yshift=0.69cm,xshift=1cm+0.37cm}{0.29}{9}{3}
		\equidistantPoints{yshift=0.69cm+0.37cm,xshift=1cm}{0.28}{10}{3}
		\equidistantPoints{yshift=1cm+0.43cm}{0.26}{11}{3}
	\end{tikzpicture}
   \caption{\label{fig:ex-joint} Examples of sets illustrating the possible relations between the basic exponents: $\bex_0<\bex_1$ (left), $\bex_0=\bex_1$ (middle), $\bex_0>\bex_1$ (right), meaning that the support measure $\mu_1$ (left), $\mu_0$ (right) or both together (middle) determine the outer Minkowski dimension, cf.\ Remark~\ref{rem:frac-ex}. }
\end{figure}

\begin{new} 
\section{Proof of Theorem~\ref{thm:basic-exp} 
} \label{sec:4}

In this section we provide a proof of Theorem~\ref{thm:basic-exp},
for which we proceed as follows. First we prove the dichotomy of (a) vs.\ (b) in Theorem~\ref{thm:basic-exp}, see Lemma~\ref{lem:lower-bound} below.
Then we recall a result from \cite{Last06}, which implies that $|\mu_0|(A;\cdot)\not\equiv 0$, cf.~
Proposition~\ref{prop:pos-curv}.
The next step is to establish the `$\geq$'-relation in \eqref{eq:dim_max_m},
i.e., that $\udimout_M$ is bounded from above by the maximum of the upper basic scaling exponents, see Proposition~\ref{prop:mink-basic-ineq}.
The remaining inequalities in \eqref{eq:dim_max_m} are derived from some general upper bounds for the variational basic functions $\fbv$ stated in Propositon~\ref{prop:basic-est}, see Proposition~\ref{prop:basic-exp-upper-bd}.

After completion of the proof of Theorem~\ref{thm:basic-exp} we end the section with a discussion of some further properties. 
Scaling properties of basic exponents and basic contents are addressed in Proposition~\ref{prop:scale-basic}, and finally an improved integrability property for support measures is derived from Theorem~\ref{thm:basic-exp}, see Lemma~\ref{lem:integrability-better}.
\end{new}

As a first step towards the proof of Theorem~\ref{thm:basic-exp}, we verify, as announced, the dichotomy between cases (a) and (b).

\begin{lemma}
\label{lem:lower-bound}
	Let $A\subset\R^d$ be a nonempty compact set and $i\in I_d$. Then either $\mu_i(A;\cdot)\equiv 0$, or 
    the following inequalities hold:
	$i\leq \un{\frk{m}}_i\leq \ov{\frk{m}}_i$.
\end{lemma}

\begin{proof}
	The inequality between the upper and the lower basic scaling exponent is clear from their respective definitions.
	By Proposition~\ref{lem:parallel-supp-finite}, $\fbv_i$ is monotone decreasing. Hence $\fbv_i(t)$ either tends to $+\iy$, or to some positive number as $t\searrow 0$, or is equal to $0$ for all $t>0$.
	In the latter case, one obtains $|\mu_i|(A;N(A))=\lim_{t\searrow 0}\fbv_i(t)=0$ and hence $\mu_i$ is trivial.
	In the other two cases, clearly, if $q<i$ then $\lBC_i^{\var,q}(A)=+\infty$ (since $t^{q-i}\to +\iy$ as $t\searrow 0$) and hence $\ubex_i\geq\lbex_i\geq i$ as desired.
\end{proof}

The assertion $|\mu_0|(A;\cdot)\not\equiv 0$ is immediate from the following statement. 

\begin{new}
\begin{proposition}\emph{\cite[Propositions 2.2 and 2.3]{Last06}}\label{prop:pos-curv}
   For any nonempty and compact set $A\subset\R^d$, the set $B_\infty:=\{(x,u)\in N(A): \delta(A,x,u)=\infty\}$ satisfies $\mu_i(A;B_\infty\cap\cdot)\geq 0$ for any $i\in I_d$ and $\mu_0(A;B_\infty)=1$. 
\end{proposition}
\end{new}

As a next step towards a proof of Theorem~\ref{thm:basic-exp} we show that $\udimout_M$ is upper bounded by the maximal upper basic scaling exponent.

\begin{proposition}
	\label{prop:mink-basic-ineq}
	For any compact subset $A\subset\R^d$, 
	\begin{equation}
	\label{eq:mink-basic-ineq}
		\udimout_M A\leq\max_{i\in I_d}\ubex_i(A). 
	\end{equation}
\end{proposition}

\begin{proof}
    Assume that $A$ is such that \eqref{eq:mink-basic-ineq} is not true. Then we can choose some $q$ such that  $\max\{\ubex_i:i\in I_d\}<q<\udimout_M A$. For each $i\in I_d$, this implies
    $q>\ubex_i$ and thus
    $\uBC_i^{\var,q}(A)=0$.
    Hence there is some positive constant $C_i$ such that $t^{q-i}\fbv_{i}(t)\leq C_i$ holds for all $t\in(0,1]$, by the monotonicity of $\fbv_i$ established in Proposition~\ref{lem:parallel-supp-finite}. 
	Now we infer from the Steiner formula \eqref{eq:v-par} that, for any $\e\in(0,1]$,
	\begin{align}
		V(A_{\e}\setminus A)&\leq \sum_{i=0}^{d-1}\o_{d-i}\int_0^{\e}t^{d-i-1}\fbv_i(t)\d t\leq\sum_{i=0}^{d-1}\o_{d-i}C_i\int_0^{\e}t^{d-i-1}t^{i-q}\d t\nn\\
		&=\frac{\e^{d-q}}{d-q}\sum_{i=0}^{d-1}\o_{d-i}C_i\nn.
	\end{align}
	This shows that $\e^{q-d}V(A_{\e}\setminus A)$ is bounded as $\e\searrow 0$, which contradicts the assumption $q<\udimout_MA$.	
\end{proof}

We point out that the corresponding argument fails for $\ldimout_M$ and the lower basic scaling exponents $\lbex_i$.\\

For establishing that $\udimout_MA$
is an upper bound for all basic exponents $\ubex_i$ (and, similarly, that $\ldimout_M$ an upper bound for $\lbex_i$), 
the following estimates for the variational basic functions $\fbv_i$ turn out to be helpful.  They are derived from some relations in \cite{HugLasWeil} and they will also be used later on in the proof of Proposition~\ref{prop:support-est}.

\begin{proposition}
    \label{prop:basic-est}
  \begin{new}  
  (i) For $i\in I_d$ there exists a constant $c_i>0$ such that
    \begin{equation*}
        \fbv_i(A;t)\leq c_it^{i-d} V(A_t\setminus A)
    \end{equation*}
    holds for all nonempty compact sets $A\subseteq\R^d$ and all
$t>0$. 

   (ii) For $i\in I_d$ and any $R>0$, there exists a constant $c'_i=c_i'(R)>0$ such that
    \begin{equation*}
        \fbv_i(A;t)\leq c'_it^i
    \end{equation*}
    for any nonempty compact set $A\subseteq\R^d$ and all
$t>0$ satisfying $t\geq R\diam(A)$.
\end{new}
\end{proposition}

\begin{new}
\begin{proof}\label{lem:basic-est-proof}
Fix $i\in I_d$, and let
$A\subseteq\R^d$ be nonempty and compact.
Put $l=d-1-i$.

By \cite[Eq.~(2.21)]{HugLasWeil}, there exists a constant
$c_3(l,d)>0$, depending only on $l$ and $d$ (but not on $A$), such that
$$
c_3(l,d)\,\min\{\dl(A,x,u),t\}^{l+1}
|H_l(A,x,u)|
\leq
\sum_{j=0}^{d-1}
\frac{\min\{\dl(A,x,u),t\}^{j+1}}{j+1}
H_j(A,x,u)
$$
for $\sH^{d-1}$-almost all $(x,u)\in N(A)$.
Integrating over $N(A)$ and applying
\cite[Eq.~(2.16)]{HugLasWeil} with the choice $B:=A$, gives
$$
c_3(l,d)
\int_{N(A)}
\min\{\dl(A,x,u),t\}^{d-i}
|H_l(A,x,u)|
\,\sH^{d-1}(\d(x,u))
\leq V(A_t\setminus A).
$$

Furthermore, \eqref{eq:support-formula} and the
definition \eqref{eq:m_i-tot} of the function $\fb^{\var}_i$ imply\footnote{Indeed, by \eqref{eq:support-formula}, the restriction of $\mu_i(A;\cdot)$ to $\{\delta(A,\cdot)>t\}$ has density $\omega_{d-i}^{-1}H_{d-1-i}(A,\cdot)$ with respect to $\sH^{d-1}$. Hence its total variation has density $\omega_{d-i}^{-1}|H_{d-1-i}(A,\cdot)|$.}
\begin{equation}\label{eq:var-beta-integral}
\fbv_i(A;t)
=
\frac{1}{\omega_{d-i}}
\int_{N(A)}
\1{\{t<\dl(A,x,u)\}}
|H_l(A,x,u)|
\,\sH^{d-1}(\d(x,u)).
\end{equation}
Since
$
\min\{\dl(A,x,u),t\}=t$ on
$\{\dl(A,x,u)>t\},
$
we conclude 
\begin{align*}
t^{d-i}\fbv_i(A;t)
&\leq
\frac{1}{\omega_{d-i}}
\int_{N(A)}
\min\{\dl(A,x,u),t\}^{d-i}
|H_l(A,x,u)|
\,\sH^{d-1}(\d(x,u))\\
&\leq
\frac{V(A_t\setminus A)}
     {c_3(l,d)\,\omega_{d-i}},
\end{align*}
which proves the first assertion (with $c_i:=c_3(l,d)^{-1}\omega_{d-i}^{-1}$).

For a proof of assertion (ii), fix $R>0$. Clearly, $A_t$ is contained in a ball of radius
$\diam(A)+t$ centered at any point of $A$. Thus, if
$t\geq R\diam(A)$, then
$$
V(A_t\setminus A)
\leq \kappa_d\bigl(\diam(A)+t\bigr)^d
\leq \kappa_d(1+R^{-1})^d t^d.
$$
Combining the last two estimates yields
$$
\fbv_i(A;t)\leq c'_i t^i,
\qquad
c'_i:=
\frac{\kappa_d(1+R^{-1})^d}
     {c_3(d-1-i,d)\,\omega_{d-i}}.
$$
The constant $c'_i$ depends only on $d$, $i$, and $R$,
but not on $A$ or $t$.
\end{proof}

Now we are ready to establish the claimed upper bounds for the basic exponents.
\end{new}
\begin{proposition}
    \label{prop:basic-exp-upper-bd}
  For any nonempty compact set $A\subseteq\R^d$ and any $i\in I_d$,
  \begin{align*}
      \ubex_i(A)\leq \udimout_M A \quad \text{ and }\quad \lbex_i(A)\leq \ldimout_M A.
  \end{align*}
\end{proposition}

\begin{new}
\begin{proof} 
    Let $i\in I_d$ and  $A\subset\R^d$ be nonempty and compact. Let $q>\udimout_M A$. Then $\uMinkout{q}(A)=0$, i.e., $\lim_{\e\searrow 0} \e^{q-d} V(A_\e\setminus A)=0$. By Proposition~\ref{prop:basic-est} (i),
    $$
    \lim_{\e\searrow 0} \e^{q-i} \fb^{\var}_i(\e)\leq \lim_{\e\searrow 0} \e^{q-i}\cdot c_i \e^{i-d} V(A_\e\setminus A)=
    c_i \lim_{\e\searrow 0} \e^{q-d} V(A_\e\setminus A)=0,
    $$
    which implies $q\geq \ubex_i(A)$. Since $q>\udimout_M A$ was arbitrary, it follows $\udimout_M A\geq \ubex_i$ showing the first assertion.  

For a proof of the second assertion, let $q>\ldimout_M A$. Then $\lMinkout{q}(A)=0$, which implies there is some decreasing null sequence $(t_j)_{j\in\N}$ such that $\lim_{j\to\infty} t_j^{q-d} V(A_{t_j}\setminus A)=0$. Again, by Proposition~\ref{prop:basic-est} (i), we infer that
    $$
     t_j^{q-i} \fb^{\var}_i(t_j)\leq  c_i t_j^{q-d} V(A_{t_j}\setminus A)\to 0, \text{ as } j\to\infty,
    $$
    which implies $q\geq \lbex_i(A)$. Since $q>\ldimout_M A$ was arbitrary, we conclude $\ldimout_M A\geq \lbex_i(A)$, i.e., the second assertion. 
\end{proof}
\end{new}

\begin{proof}[{\bf Proof of Theorem~\ref{thm:basic-exp}:}]
For the dichotomy of (a) and (b) see Lemma~\ref{lem:lower-bound}. The assertion $\mu_0(A;\cdot)\not\equiv 0$ is a direct consequence of $\mu_0(A;B_\infty)>0$ in Proposition~\ref{prop:pos-curv}. 
 Finally, the relations in \eqref{eq:dim_max_m} follow by combining Propositions~\ref{prop:mink-basic-ineq} and \ref{prop:basic-exp-upper-bd}.
\end{proof}

Another important property is the following invariance of the basic scaling exponents and the corresponding contents with respect to similarity mappings, which is a consequence of the scaling properties and motion covariance of support measures, cf.~\eqref{eq:motion-cov} and \eqref{eq:homogen}.

\begin{proposition}[Scaling property for basic exponents and contents]\label{prop:scale-basic}
	Let $A$ be a compact subset of $\R^d$ and $r>0$.
	Let $rA:=\{rx:x\in A\}$.
	Then, for each $i\in I_d$, 
	\begin{equation}\label{eq:basic-scaling-inv}
		\ov{\frk{m}}_{i}(rA)=\ov{\frk{m}}_{i}(A) \quad\mathrm{and}\quad \un{\frk{m}}_{i}(rA)=\un{\frk{m}}_{i}(A).
	\end{equation}
		Furthermore, for any $q\in\R$,
	\begin{equation}
		\uBC_i^q(rA)=r^q\, \uBC_i^q(A)  \quad\mathrm{and}\quad \lBC_i^q(rA)=r^q \lBC_i^q(A).
	\end{equation}
	Moreover, the same relations hold for the total variation counterparts, that is,
	\begin{equation}
		\uBC_i^{\var,q}(rA)=r^q\,\uBC_i^{\var,q}(A)  \quad\mathrm{and}\quad \lBC_i^{\var,q}(rA)=r^q\lBC_i^{\var,q}(A).
	\end{equation}
\end{proposition}

\begin{proof}
	Observe that, after performing the change of variables $y:=x/r$,
	\begin{equation}
		\begin{aligned}		\fb_i(rA,t)&=\int_{N(rA)}\1{\{t<\dl(rA,x,u)\}}\mu_i(rA;\d(x,u))\\
		&=\int_{N(rA)}\1{\{t<\dl(rA,ry,u)\}}\mu_i(rA;\d(ry,u)).
	\end{aligned}
	\end{equation}	
	Since $\dl(rA,ry,u)=r\dl(A,y,u)$ and $\mu_i(rA;\d(ry,u))=r^i \mu_i(A;\d(y,u))$, 
    by the scaling property \eqref{eq:homogen} of the support measures, we infer that
	\begin{equation} \label{eq:scaling-beta}
		\fb_i(rA,t)=r^i\int_{N(A)}\1{\{tr^{-1}<\dl(A,y,u)\}}\mu_i(A;\d(y,u))
		=r^i\fb_i(A,tr^{-1}).
	\end{equation}
	Thus, we obtain for the (upper) $i$-th basic content of $rA$, 
	\begin{equation}
		\ov{\Cal M}_i^q(rA)=\limsup_{t\searrow 0}t^{q-i}\fb_{i}(rA,t)=\limsup_{t\searrow 0}(tr^{-1})^{q-i}r^q\fb_{i}(A,tr^{-1})=r^q\ov{\Cal M}_i^q(A).
	\end{equation}
	The corresponding relations for the lower basic  contents and for the total variation counterparts are derived analogously.
	From this one immediately concludes the equality of the upper basic scaling exponents of $rA$ and $A$ as asserted in \eqref{eq:basic-scaling-inv} (and similarly for the lower basic exponents).
\end{proof}

 As a last step we state the announced improvement of the integrability result for support measures from \cite[eq.~(2.2)]{HugLasWeil}  (see also equation \eqref{eq:integrability} above). \begin{new}While the exponent $d$ in \eqref{eq:integrability} is a universal exponent ensuring integrability for all closed sets $A\subset\R^d$, the result below allows to interpret $\ubex_i(A)$ as the optimal individual exponent above which this integrability is ensured for a given $A$.\end{new} 

\begin{lemma}
	\label{lem:integrability-better}
	Let $A$ be a compact subset of $\R^d$, $i\in I_d$, and $t>0$.  
	Then the following integral is finite for any $\gamma>\ubex_i(A)$:
	\begin{equation}
	\label{eq:integrability-better}
		\int_{N(A)}\min\{t,\dl(A,x,u)\}^{\gamma-i}|\mu_i|(A;\d (x,u))<\iy.
	\end{equation}
\end{lemma}

\begin{proof}
If $\mu_i(A;\cdot)\equiv0$, then the integral in
\eqref{eq:integrability-better} is zero and there is nothing
to prove. We may therefore assume that
$\mu_i(A;\cdot)\not\equiv0$.
By Theorem~\ref{thm:basic-exp}, we then have
$\ubex_i(A)\geq i$, and hence $\gamma-i>0$.
	Let $q\in(\ubex_i,\gamma)$. Then we have $ \uBC_i^{\var,q}(A)=0$, implying that $\tau^{q-i}\fbv_i(\tau)$ is bounded from above by some positive constant $C_t$ for all $\tau\in(0,t]$, by the monotonicity of $\fbv_i$ established in Proposition \ref{lem:parallel-supp-finite}, i.e.,
	\begin{equation}
		\label{eq:estimate}
		\fbv_{i}(\tau)\leq C_t\tau^{i-q}\quad\textrm{for all}\quad \tau\in(0,t].
	\end{equation}
	We subdivide the normal bundle $N(A)$ of $A$ in a dyadic fashion according to the local reach at its points, i.e., we let
	\begin{equation}
		B_j:=\left\{(x,u)\in N(A):2^{-j-1}t<\dl(A,x,u)\leq 2^{-j}t\right\}, \quad\textrm{for all}\quad j\in\N_0,
	\end{equation}
	so that we have the following (pairwise) disjoint decomposition of the normal bundle:
	\begin{equation}
		N(A)=\left\{(x,u)\in N(A):\dl(A,x,u)>t\right\}\cup\bigcup_{j=0}^{\infty} B_j.
	\end{equation}
	Consequently, the integral in \eqref{eq:integrability-better}, which we denote here by $I(\gamma)$, can be decomposed as follows:
	\begin{align}
	\label{eq:integral-sum-est}
		I(\gamma)=&\sum_{j=0}^{\iy}\int_{B_j}\min\{t,\dl(A,x,u)\}^{\gamma-i}|\mu_i|(A;\d (x,u))\\
		&+\int_{\left\{(x,u)\in N(A):\dl(A,x,u)>t\right\}}t^{\gamma-i}|\mu_i|(A;\d (x,u)).\nn
	\end{align}
	Observe that the last integral above is exactly equal to $t^{\gamma-i}\fbv_i(t)$ and hence finite by Proposition \ref{lem:parallel-supp-finite}.
	
	Let us further proceed by estimating the sum of integrals appearing in \eqref{eq:integral-sum-est}. Observe that on the set $B_j$ the integrand is bounded from above by $(2^{-j}t)^{\gamma-i}$, since the exponent $\gamma-i$ is positive due to the relation $\gamma>\ubex_i\geq i$. Hence,
	\begin{align*}
		I(\gamma)&\leq\sum_{j=0}^{\iy}2^{-j(\gamma-i)}t^{\gamma-i}\int_{B_j}|\mu_i|(A;\d (x,u))+t^{\gamma-i}\fbv_i(t)\\
		&\leq\sum_{j=0}^{\iy}2^{-j(\gamma-i)}t^{\gamma-i}\int_{\left\{(x,u)\in N(A):\dl(A,x,u)>2^{-j-1}t\right\}}|\mu_i|(A;\d (x,u))+t^{\gamma-i}\fbv_i(t)\nn\\
		&\leq C_t\cdot\sum_{j=0}^{\iy}2^{-j(\gamma-i)}t^{\gamma-i}(2^{-j-1}t)^{i-q}+t^{\gamma-i}\fbv_i(t)\nn\\
		  &=2^{q-i}t^{\gamma-q}C_t\sum_{j=0}^{\infty}2^{-j(\gamma-q)}+t^{\gamma-i}\fbv_i(t)\\
        &=\frac{2^{q-i}t^{\gamma-q}C_t}{1-2^{q-\gamma}}+t^{\gamma-i}\fbv_i(t)<\iy,
	\end{align*}
	where the second inequality is due to the inclusion $B_j\stq\left\{(x,u):\dl(A,x,u)>2^{-j-1}t\right\}$. In the third inequality we used the uniform estimate \eqref{eq:estimate},
    and for the last equality that we have chosen $q<\gamma$ so that the geometric series converges.
	This completes the proof.
\end{proof}

\section{The Steiner formula for support measures}\label{sec:steiner}

 In this section we recall a relation between the support measures of a compact set $A$ and its parallel sets $A_\eps$, $\eps>0$, known as the \emph{Steiner formula for support measures}. We observe that it may be reformulated in terms of the basic functions of $A$, see Proposition~\ref{prop:Steiner-for-mu_i}, and use this observation to establish some bounds on the growth behavior of the support measures $\mu_i(A_\eps,\cdot)$ of the parallel sets $A_\eps$ as a function of $\eps$ in Proposition~\ref{prop:support-est}. 
 \begin{new} These bounds will be useful for the results of the next section. We also address the relation of the support measures of $A_\e$ and of its closed complement $\widetilde{A_\e}$, see Lemma~\ref{lem:normal-reflection}, allowing to draw some connection to curvature measures. Finally, we discuss in Proposition~\ref{prop:invert-rela} a certain inversion of the formulas in Propostion~\ref{prop:Steiner-for-mu_i}, expressing the basic functions in terms of support functions.
 \end{new}

According to \cite[Corollary 4.4]{HugLasWeil}, the support measures of the $\e$-parallel sets of $A$ satisfy the following Steiner-type formula in terms of the support measures of the set $A$. For any $\e>0$, 
 \begin{equation}\label{eq:mu-parallel-local}
     \mu_i(A_{\e};\hat{B})=\sum_{j=0}^ic_{i,j}\,\e^{i-j}\,\int_{N(A)}\hspace{-3mm}\1{\{\e<\dl(A,x,u)\}}\,\1{\{T_{\eps}(x,u)\in \hat{B}\}}\,\mu_j(A;\d (x,u)),
 \end{equation}
 for any $\sH^{d-1}$-measurable set $\hat{B}\subset N(A_\e)$. Here $T_{\e}\colon N(A)\to \R^d\times\eS^{d-1}$ is the mapping defined by
 \begin{align} \label{eq:T_eps-def}
     T_{\e}(x,u):=(x+\e u,u)
 \end{align}
 and the constants $c_{i,j}$ are given by
\begin{equation}\label{eq:cij}
	c_{i,j}:=\binom{d-j}{d-i}\frac{\kappa_{d-j}}{\kappa_{d-i}}
\end{equation}
for $i,j\in\{0,\ldots,d-1\}$. Note that the $c_{i,j}$ do not depend on $A$, that $c_{i,j}=0$ if $j>i$ and that $c_{i,i}=1$. \begin{new}Hence these coefficients form a lower triangular matrix with units on the diagonal.\end{new} Note also that $T_\eps(x,u)\in N(A_\eps)$ if and only if $(x,u)\in N(A)$ and $\eps<\dl(A,x,u)$.
In addition, by \cite[Corollary 4.4]{HugLasWeil}, the total variation of $\mu_i(A_{\e};\cdot)$ is finite.

\begin{new}
The next statement is a direct consequence of equation \eqref{eq:mu-parallel-local} and the definition of the basic functions. The total mass $\mu_i(A_{\eps}):=\mu_i(A_{\e};\R^d\times\eS^{d-1})$ of the $i$-th support measure of any parallel set $A_\e$ of a compact set $A$ can be expressed in terms of the basic functions of $A$.
\begin{proposition}\label{prop:Steiner-for-mu_i}
 Let $A\subseteq\R^d$ be a nonempty compact set and $i\in I_d$. Then, for any $\eps>0$, and with $c_{i,j}$ as defined in \eqref{eq:cij},
\begin{equation}\label{eq:mu-parallel}
	\mu_i(A_{\eps})=\sum_{j=0}^ic_{i,j}\,\e^{i-j}\,\fb_j(\e)
\end{equation}
and 
\begin{equation}\label{eq:mu-parallel-var}
	|\mu_i|(A_{\eps})\leq\sum_{j=0}^ic_{i,j}\,\e^{i-j}\,\fb_j^{\var}(\e).
\end{equation}
\end{proposition}
\begin{proof} 
Choose $\hat B=N(A_\eps)$ in \eqref{eq:mu-parallel-local}. Since $T_{\eps}(x,u)\in N(A_\eps)$ whenever $\eps<\dl(A,x,u)$, it is clear that in this case the integrals on the right hand side of \eqref{eq:mu-parallel-local} simplify to those in the definition \eqref{eq:m_i} of the basic functions of $A$. 

For a proof of \eqref{eq:mu-parallel-var}, fix $\e>0$ and observe that $\mu_i(A_\eps;\cdot)$ is a finite signed measure. 
Let $N^+\cup N^-=N(A_\e)$ be a Hahn decomposition
for $\mu_i(A_\e;\cdot)$.
Choosing $\hat B=N^{\pm}$ in \eqref{eq:mu-parallel-local}, one obtains
\begin{align*}
    |\mu_i(A_\e;N^{\pm})|
    &=\left|\sum_{j=0}^ic_{i,j}\,\e^{i-j}\,\int_{N(A)}\hspace{-3mm}\1{\{\e<\dl(A,x,u)\}}\,\1{\{T_{\eps}(x,u)\in N^{\pm}\}}\,\mu_j(A;\d (x,u))\right|\\
    &\leq\sum_{j=0}^ic_{i,j}\,\e^{i-j}\,\int_{N(A)}\hspace{-3mm}\1{\{\e<\dl(A,x,u)\}}\,\1{\{T_{\eps}(x,u)\in N^{\pm}\}}\,|\mu_j|(A;\d (x,u)).
\end{align*}
Combining both estimates, using
$N(A_\e)=N^+\cup N^-$ and recalling that
$T_\e(x,u)\in N(A_\e)$ if and only if
$\e<\dl(A,x,u)$, we get
\begin{align*}
    |\mu_i|(A_{\e})&=
     |\mu_i(A_\e;N^{+})|+|\mu_i(A_\e;N^{-})|\\
    &\leq\sum_{j=0}^ic_{i,j}\,\e^{i-j}\,\int_{N(A)}\hspace{-3mm}\1{\{\e<\dl(A,x,u)\}}\,\1{\{T_{\eps}(x,u)\in N(A_\e)\}}\,|\mu_j|(A;\d (x,u)),
\end{align*}
from which \eqref{eq:mu-parallel-var} is transparent.
\end{proof} 
\end{new}

Our next aim is to establish upper bounds on the total mass $|\mu_i|(A_\e)$ of the total variation of the support measure for parallel sets. \begin{new}They may be viewed as analogues of the bounds for $\fbv_i$ derived in Proposition~\ref{prop:basic-est}. They are easily deduced from those and the previous statement. Part (ii)  should also be compared with a corresponding estimate for curvature measures in \cite[see Theorem~4.1]{Zahle2011}, cf.~also Remark~\ref{rem:normal-refl} below.\end{new}

\begin{new}
\begin{proposition}
    \label{prop:support-est}
    (i) For $i\in I_d$  there exists a constant $\check c_i>0$ such that
    \begin{equation}
        \label{eq:basic-est-mui-vol}
        |\mu_i|(A_{\eps})\leq \check c_i \eps^{i-d}V(A_\e\setminus A)
    \end{equation}
    for all compact sets $A\subseteq\R^d$ and all $\eps>0$.
    
   (ii) For $i\in I_d$ and any $R>0$ there exists a constant $\hat c_i=\hat c_i(R)>0$ such that
    \begin{equation}
        \label{eq:basic-est-mui}
        |\mu_i|(A_{\eps})\leq \hat c_i\eps^i
    \end{equation}
    for all compact sets $A\subseteq\R^d$ and all $\eps$ such that  $\eps\geq R\diam(A)$.
\end{proposition}
\begin{proof}
    Let $i\in I_d$ and let $A\subset\R^d$ be nonempty and compact. 
    (i) Combining the estimate in \eqref{eq:mu-parallel-var} 
    with Proposition~\ref{prop:basic-est} (i), we infer the existence of positive constants $c_j$, $j=0,\ldots, i$, independent of $A$, such that, for any $\e>0$,
  \begin{align*}
	|\mu_i|(A_{\eps})\leq\sum_{j=0}^ic_{i,j}\,\e^{i-j}\,\fb_j^{\var}(\e)\leq \sum_{j=0}^i c_{i,j}\,\e^{i-j}\,c_j \e^{j-d}V(A_{\eps}\setminus A)=\check c_i \e^{i-d} V(A_{\eps}\setminus A),
\end{align*}  
where the constant $\check c_j:=\sum_{j=0}^i c_{i,j}c_j$ is positive and independent of $A$, proving the asserted estimate in (i).

For a proof of (ii), fix $R>0$. Combining the estimate in 
\eqref{eq:mu-parallel-var} with Proposition~\ref{prop:basic-est} (ii), we infer the existence of positive constants $c_j'(R)$, $j=0,\ldots, i$, independent of $A$, such that, for any $\e\geq R\diam(A)$,
  \begin{align*}
	|\mu_i|(A_{\eps})\leq\sum_{j=0}^ic_{i,j}\,\e^{i-j}\,\fb_j^{\var}(\e)\leq \sum_{j=0}^i c_{i,j}\,\e^{i-j}\,c_j'(R) \e^j=\hat c_i \e^i,
\end{align*}  
where $\hat c_j:=\hat c_j(R):=\sum_{j=0}^i c_{i,j}c_j'(R)$ is independent of $A$, proving the bound in (ii).
\end{proof}
\end{new}

\begin{new}
\begin{remark} \label{rem:normal-refl}
An alternative way to derive the bound in \eqref{eq:basic-est-mui} (though only for $R>\sqrt{2}$) is via normal reflection (see Lemma~\ref{lem:normal-reflection} below) and the relation of $\mu_i(A;\cdot)$ with the corresponding curvature-direction measure $C_i(A,\cdot)$ (cf.\ p.\,\pageref{page:curv}):
Let $R>\sqrt{2}$.  By \cite[Lemma 2.2]{HugLasWeil}, for any $\varepsilon>R\diam(A)$, the set $\widetilde{A_{\eps}}$ has positive reach and therefore, by \cite[\S 3]{HugLasWeil} and \cite[Theorem 3.3]{HugLast00}, $\mu_i(\widetilde{A_{\eps}})=C_i(\widetilde{A_{\eps}};N(\widetilde{A_{\eps}}))$.
By \cite[Theorem 4.1]{Zahle2011}, there exists a positive constant $c_i$ (depending only on $R$) such that
$$|\mu_i|(\widetilde{A_{\eps}})=C_i^{\var}(\widetilde{A_{\eps}};N(\widetilde{A_{\eps}}))\leq C_i^{\var}(\widetilde{A_{\eps}};\R^d\times\eS^{d-1})\leq c_i\eps^i$$
for all $\eps>R\diam(A)$, where $C_i^{\var}$ denotes the total variation measure of $C_i(A,\cdot)$.
Combining this with Lemma~\ref{lem:normal-reflection} shows \eqref{eq:basic-est-mui} for $R>\sqrt{2}$.
\end{remark}
\end{new}

The following relation between the total mass of the total variation $|\mu_i|(A_\e)$ of the $\e$-parallel set of a compact set $A\subseteq \R^d$ and of its closed complement $\widetilde{A_{\e}}:=\overline{\R^d\setminus A_{\e}}$ is of interest in its own and gives an alternative way to derive \eqref{eq:basic-est-mui} in case when $R>\sqrt{2}$.

\begin{lemma}\label{lem:normal-reflection}
Let $A\subseteq\R^d$ be compact and nonempty and $i\in I_d$ arbitrary.
Then,
\begin{equation}\label{eq:reflection-support}
    |\mu_i|(A_{\eps})\leq|\mu_i|(\widetilde{A_{\eps}}),
\end{equation}
for all $\eps>\diam(A)$.
\end{lemma}

\begin{proof}
    Let $\eps>\diam(A)$. 
    First, we claim that if $(x,u)\in N(A_{\eps})$, then the reflection $T(x,u):=(x,-u)$ is an element of $ N(\widetilde{A_{\eps}})$, i.e.,
    \begin{align}
        \label{eq:reflection-prop}
        N(A_{\eps})\subseteq T(N(\widetilde{A_{\eps}})).
    \end{align}
    To establish this, let $(x,u)\in N(A_{\eps})$ and observe that, by \cite[Lem.\ 2.2]{HugLasWeil}, $\widetilde{A_{\eps}}$ is a set with positive reach,  since $\eps>\diam(A)$.
    Fix some $r<\mathrm{reach}(\widetilde{A_{\eps}})$ and let $y:=x-ru$.
    We are going to prove that $\pi_{\widetilde{A_{\eps}}}(y)=x$, which will imply that $(x,-u)\in N(\widetilde{A_{\eps}})$ and thus \eqref{eq:reflection-prop}.
    Recall that $\pi_{\widetilde{A_{\eps}}}(y)=x$ holds if and only if $B(z,r)\cap\widetilde{A_{\eps}}(y)=\{x\}$, where $B(y,r)$ is the closed ball centered at $z$ of radius $r$.

    Observe that  $(x,u)\in N(A_{\eps})$ implies $(x-\pi_A(x))/\dist(x,A)=u$.
    Indeed, let $z$ be such that $\pi_{A_{\eps}}(z)=x$ and assume for a contradiction that the segment $[\pi_A(x),z]$ intersects $\partial A_{\eps}$ in a point $x'\neq x$.
    Since $|\pi_A(x)-x'|=\eps$, the triangle inequality implies
    $|\pi_A(x)-x'|+\eps=|\pi_A(x)-z|<|\pi_A(x)-x|+|x-z|=\eps+|x-z|$, i.e., $|x'-z|<|x-z|$, in contradiction to the fact that $\pi_{A_{\eps}}(z)=x$.

    Hence, the open balls satisfy $B^{\circ}(y,r)\subset B^{\circ}(\pi_A(x),\eps)$ since $x$, $y$ and $\pi_A(x)$ are co-linear.
    Furthermore, since $B^{\circ}(\pi_A(x),\eps)\cap \widetilde{A_{\eps}}=\emptyset$, we infer that $\dist(\widetilde{z},\widetilde{A_{\eps}})=r$. By the definition of $y$, we have $|y-x|=r$ and, since $\pi_{\widetilde{A_{\eps}}}(y)$ is unique, we conclude that it has to be equal to $x$. This proves \eqref{eq:reflection-prop}.

    To continue the proof, fix some $i\in I_d$. Observe that, for all $(x,u)\in N(A_{\eps})$, the following {\em reflection relation} for the generalized principal curvatures holds: $k_j(A_{\eps},x,u)=-k_j(\widetilde{A_{\eps}},x,-u)$ for $j=1,\ldots, d-1$ and $\sH^{d-1}$-a.e.\ $(x,u)\in N(A_{\eps})$, see e.g.\ \cite[Prop.\ 5.1.]{HugLasWeil}. It follows now directly from the definition of $H_{j}$ (see \eqref{eq:Hj})
    that, for all such $(x,u)$,
    \begin{equation}\label{eq:reflection}
        H_{d-i-1}(\widetilde{A_{\eps}},x,-u)=(-1)^{d-i-1}H_{d-i-1}(A_{\eps},x,u).
    \end{equation}

    Hence, one has
    $$
    \begin{aligned}
    \mu_i(A_{\eps})&=\frac{1}{\o_{d-i}}\int_{N(A_{\eps})}H_{d-i-1}(A_{\eps},x,u)\sH^{d-1}(\d (x,u))\\
    &=\frac{(-1)^{d-i-1}}{\o_{d-i}}\int_{N(\widetilde{A_{\eps}})}\1{\{(x,-u)\in N(A_{\eps})\}}H_{d-i-1}(\widetilde{A_{\eps}},x,u)\sH^{d-1}(\d (x,u)),
    \end{aligned}
    $$
    where the second equality is due to the change of variables $T(x,u)=(x,-u)$ (with Jacobian $JT(x,u)=1$) and the reflection formula \eqref{eq:reflection}.
    For the total variation this implies
    $$
    \begin{aligned}
    |\mu_i|(A_{\eps})&=\frac{1}{\o_{d-i}}\int_{N(\widetilde{A_{\eps}})}\1{\{(x,u)\in N(A_{\eps})\}}|H_{d-i-1}(\widetilde{A_{\eps}},x,u)|\sH^{d-1}(\d (x,u))\\
    &\leq \frac{1}{\o_{d-i}}\int_{N(\widetilde{A_{\eps}})}|H_{d-i-1}(\widetilde{A_{\eps}},x,u)|\sH^{d-1}(\d (x,u))=|\mu_i|(\widetilde{A_{\eps}}),
    \end{aligned}
    $$
    as claimed in \eqref{eq:reflection-support}, completing the proof of Lemma~\ref{lem:normal-reflection}.
\end{proof}

The following statement is a certain converse to
Proposition~\ref{prop:Steiner-for-mu_i},
which is interesting in its own right.
It will in particular be used in the proof of
Corollary~\ref{cor:lower-comparison-regular}.

\begin{proposition}\label{prop:invert-rela}
    Let $A\subseteq\R^d$ be nonempty and compact.
    Then, for each $i\in I_d$ and all $\eps>0$,
    \begin{equation}\label{eq:basic_from_support}
	    \fb_i(\eps)=\sum_{j=0}^ib_{i,j}\,\e^{i-j}\,\mu_j(A_{\e}),
    \end{equation}
    where the coefficients $b_{i,j}$ are independent of $\e$ and $A$, and given by
	\begin{equation}\label{eq:b_ij}
		b_{i,j}:=(-1)^{i+j}\det(C_{j,i}),
	\end{equation}
with $C_{j,i}$ being the submatrix of $C:=(c_{i,j})_{i,j=0,\ldots,d-1}$ with deleted $j$-th row and $i$-th column and $c_{i,j}$ as in \eqref{eq:cij}.

Furthermore, for all $\eps>0$, \begin{new}
\begin{equation}\label{eq:bound-basic-up}
\fb^{\var}_i(\eps)\leq \sum_{j=0}^i|b_{i,j}| \e^{i-j}\,|\mu_j|(A_{\e}).
\end{equation}
\end{new}
\end{proposition}

\begin{proof}
Fix $\eps>0$. Multiplying, for any $i\in I_d$, equation \eqref{eq:mu-parallel} by $\e^{-i}$ we obtain
\begin{equation*}
	\e^{-i}\mu_i(A_{\eps})=\sum_{j=0}^ic_{i,j}\,\e^{-j}\,\fb_j(\e), \quad i=0,\ldots,d-1.
\end{equation*}
These $d$ equations form a linear system $C\mathrm{x}=\mathrm{y}$ connecting the vectors $$
\mathrm{y}:=(\e^{-i}\mu_i(A_{\eps}))_{i=0,\ldots,d-1} \text{ and } \mathrm{x}:=(\e^{-j}\,\fb_j(\e))_{i=0,\ldots,d-1}
$$
via the lower triangular matrix $C:=(c_{i,j})_{i,j=0,,\ldots,d-1}$ with diagonal coefficients equal to 1. \begin{new}(Recall that $c_{i,j}=0$ for $i<j$ and $c_{i,i}=1$.)\end{new}
Hence, the system is invertible and by Cramer's rule one obtains formula \eqref{eq:basic_from_support}.

It remains to prove \eqref{eq:bound-basic-up}.
For this purpose we consider local versions of the basic functions.
For any $j\in I_d$ and any $\sH^{d-1}$-measurable set $B\subseteq N(A)$, define
\begin{equation}   \label{eq:fb-local} \fb_j(A;B,t):=\int_{N(A)}\1{\{t<\dl(A,x,u)\}}\times\1{\{(x,u)\in B\}}\mu_j(A;\d (x,u)),\quad t>0.
\end{equation}
In view of the local nature of \eqref{eq:mu-parallel-local}, it is easy to see that equation \eqref{eq:basic_from_support} has a local counterpart: for any $\sH^{d-1}$-measurable $B\subseteq N(A)$,\begin{equation}\label{eq:basic_from_support_local}
	\fb_i(A;B,\e)=\sum_{j=0}^ib_{i,j}\e^{i-j}\,\mu_j(A_{\e};T_{\e}(B)),
\end{equation}
where the mapping $T_\e$ is as in \eqref{eq:T_eps-def}.

\begin{new}
For fixed $\e>0$, the set function
$B\mapsto\nu(B):=\fb_i(A;B,\e)$ is a finite signed measure on $N(A)$.
Choose a Hahn decomposition $N^+\cup N^-=N(A)$ for $\nu$; hence
$\nu$ is nonnegative on every measurable subset of $N^+$
and nonpositive on every measurable subset of $N^-$.
Applying \eqref{eq:basic_from_support_local} to the sets $B=N^{\pm}$, we infer that
\begin{align*}
|\fb_i(A;N^{\pm},\eps)|
&=\left|\sum_{j=0}^i b_{i,j}\e^{i-j}\,
  \mu_j(A_{\e};T_{\e}(N^{\pm}))\right|
\le\sum_{j=0}^i|b_{i,j}|\e^{i-j}\,|\mu_j|\,(A_{\e};T_{\e}(N^{\pm})).
\end{align*}
   Combining both estimates and observing that $T_\e(N^+)+T_\e(N^-)=N(A_\e)$, which implies $|\mu_j|\,(A_{\e};T_{\e}(N^{+}))+|\mu_j|\,(A_{\e};T_{\e}(N^{-}))=|\mu_j|\,(A_{\e})$, we conclude that  
\begin{align*}
    \fb_i^{\var}(A;\eps)=|\fb_i(A;N^{+},\eps)|+|\fb_i(A;N^{-},\eps)|\leq \sum_{j=0}^i|b_{i,j}|\e^{i-j}\,|\mu_j|(A_{\e}).
\end{align*}
\end{new}
This completes the proof of \eqref{eq:bound-basic-up}.
\end{proof}

\section{Support contents and support scaling exponents}\label{sec:support-funct}

	In this section, we introduce a second family of functionals associated to any compact set in $\R^d$: the \emph{support contents}. Just as the basic contents they are naturally derived from the support measures. We study their properties and explore their relation with the basic contents 
as well as with (outer) Minkowski contents and fractal curvatures.
    For any nonempty compact set $A\subseteq\R^d$ and $i\in I_d$, we call
	the function
	\begin{equation}
		\eps\mapsto\mu_i(A_{\e})
	\end{equation}
	the {\em $i$-th tube-support function} of $A$.
    Tube-support functions take the role of the basic functions in the subsequent definitions and considerations.

\begin{definition}[Support content of a compact set]\label{def:support-content}
Let $A\subseteq\R^d$ be a nonempty compact set. For any $i\in I_d$ and $q\in\R$, we define the {\em ($q$-dimensional)  $i$-th upper support content of $A$} by
\begin{equation}
	\uSC_i^q(A):=\limsup_{\eps\searrow 0}\e^{q-i}\mu_i(A_{\e}),
	\end{equation}
	and  its total variation analog by
	\begin{equation}
	\uSC_i^{\var,q}(A):=\limsup_{\eps\searrow 0}\e^{q-i}|\mu_i|(A_{\e}).
	\end{equation}
	We also introduce lower variants by replacing the upper limits by lower limits, and denote them by $\lSC_i^q(A)$ and $\lSC_i^{\var,q}(A)$, respectively.
	Moreover, we define the {\em upper $i$-th support scaling exponent} $\usex_i(A)$ of $A$ by
	\begin{equation}
		\label{eq:support-exp}
		\usex_i:=\usex_i(A):=\inf\{q\in\R:\uSC_i^{\var,q}(A)=0\}=\sup\{q\in\R:\uSC_i^{\var,q}(A)=+\iy\},
	\end{equation}
	as well as its {\em lower} counterpart $\lsex_i=\lsex_i(A)$, by replacing $\uSC_i^{\var,q}(A)$ by $\lSC_i^{\var,q}(A)$.
	
    If $\usex_i=\lsex_i$, then we denote the common value by
$\sex_i$ and call it the {\em $i$-th support scaling exponent}
of $A$.
For any $q\in\R$, if $\uSC_i^q(A)=\lSC_i^q(A)$, we denote
the common value by $\SC_i^q(A)$ and call it the
{\em ($q$-dimensional) $i$-th support content of $A$}.
Analogously, if $\uSC_i^{\var,q}(A)=\lSC_i^{\var,q}(A)$,
we denote the common value by $\SC_i^{\var,q}(A)$.
\end{definition}

Before turning to the general relations between support contents and basic contents (and between the corresponding exponents), let us first discuss the two special cases $i=0$ and $i=d-1$.
\begin{remark} \label{rem:b0-s0-rel}
    For any nonempty compact set $A\subset\R^d$ and any $q\in\R$,
    \begin{align}\label{eq:SC0=BC0}
    \uSC_0^q(A)=\uBC_0^q(A) \quad \text{ and } \quad \lSC_0^q(A)=\lBC_0^q(A).
    \end{align}
    As a consequence,
    \begin{align}\label{eq:m0=s0}
    \usex_0(A)=\ubex_0(A) \quad \text{ and } \quad \lsex_0(A)=\lbex_0(A).
    \end{align}
    \begin{new}
    The content identities in \eqref{eq:SC0=BC0} follow from \eqref{eq:mu-parallel},
which in case $i=0$ reads $\mu_0(A_\e)=\fb_0(A;\e)$.
Moreover, \eqref{eq:mu-parallel-local} and the injectivity
of $T_\e$ give
$
|\mu_0|(A_\e)=\fbv_0(A;\e).
$
The identities for the exponents therefore follow from their
definitions.
    \end{new}
\end{remark}

\begin{proposition}\label{prop:s_d-1}
   Let $A\subseteq\R^d$ be a nonempty compact set. For any $q\in[0,d)$,
   \begin{equation} \label{eq:SCd-1vsMinkout}
        \frac{2}{d}\uSC_{d-1}^{q}(A)
        \leq \uMinkout{q}(A)\leq \frac{2}{d-q}\uSC_{d-1}^{q}(A) \text{ and } \frac{2}{d-q}\lSC_{d-1}^{q}(A)\leq \lMinkout{q}(A).
    \end{equation}
    Moreover, $\uSC^{d}_{d-1}(A)=\uMinkout{d}(A)=0$.     As a consequence, $\usex_{d-1}(A)\leq d$,
    \begin{equation}\label{eq:s_d-1_dimout}
        \usex_{d-1}(A)=\udimout_M(A) \quad \text{ and } \quad \lsex_{d-1}(A)\leq\ldimout_M(A).
    \end{equation}
    In particular, if $\sex_{d-1}(A)$ exists (i.e., if $\lsex_{d-1}(A)=\usex_{d-1}(A)$), then $\dimout_M A$ exists and $\dimout_M A=\sex_{d-1}(A)$.
    Moreover, if $\SC_{d-1}^{s}(A)$ exists (i.e., if $\lSC_{d-1}^s(A)=\uSC_{d-1}^s(A)$) for some $s\in[0,d)$, then also the corresponding outer Minkowski content $\Minkout{s}(A)$ exists and $\Minkout{s}(A)=\frac 2{d-s}\SC_{d-1}^{s}(A)$.
\end{proposition}
\begin{proof}
    Note that $\mu_{d-1}(A_{\eps},\cdot)$ is a positive measure. 
    For its total mass, the relation
    \begin{equation} \label{eq:parsurfarea}
		2\mu_{d-1}(A_{\eps})=V_A^{(+)}(\eps)=\sH^{d-1}(\pa^+A_{\eps})\leq\sH^{d-1}(\pa A_{\eps})
	\end{equation}
	holds, with equality in the last relation for all $\e>0$ up to at most countably many exceptions. Here $V^{(+)}_A(\e)$ denotes the right derivative at $\e$ of the volume function $r\mapsto V_A(r):=V(A_r)$ and $\partial^+$ the positive boundary; see \cite[\S 4.3]{HugLasWeil}.
    \begin{new} 
    Recall from \cite{RaWi2010} that, for any compact set $A\subset\R^d$ and any $s\geq 0$, the $s$-dimensional (upper) $S$-content of $A$ is defined by
     \begin{align}\label{eq:S-content}
     \uS^s(A):=\limsup_{r\searrow 0} \frac{\sH^{d-1}(\partial A_r)}{(d-s)\kappa_{d-s} r^{d-1-s}},
     \end{align}
     and that it is closely related to the (upper) outer Minkowski content (defined in \eqref{eq:outer-Mink}). It follows from \eqref{eq:parsurfarea} and \cite[Cor.~3.6]{RaWi2010}, that
    \begin{equation*}
        \uSC_{d-1}^{s}(A)\leq\frac{(d-s)\kappa_{d-s}}{2}\overline{\mathcal{S}}^{s}(A)\leq \frac{d}{2}\overline{\mathcal{M}}^{\mathrm{out},s}(A),
    \end{equation*}
    which proves the first inequality in \eqref{eq:SCd-1vsMinkout}.
    \end{new}
    Moreover, it follows from \cite[eq.~(3.2)]{RaWi2010} that for each $r>0$ there are $t_1,t_2\in(0,r]$ such that $V_A'(t_1)$, $V_A'(t_2)$ exist and
    \begin{equation} \label{eq:RaWi10_3.2}
        \frac{V_A'(t_1)}{(d-s)t_1^{d-s-1}}\leq \frac{V_A(r)-V_A(0)}{r^{d-s}}\leq \frac{V_A'(t_2)}{(d-s)t_2^{d-s-1}}.
    \end{equation}
    Recalling that at differentiability points $\eps$ of $V_A$ we have $2\mu_{d-1}(A_\eps)=V_A'(\eps)$, we conclude from the right hand inequality that
    \begin{align*}
        \overline{\mathcal{M}}^{\mathrm{out},s}(A)=\limsup_{r\searrow 0}\frac{V_A(r)-V_A(0)}{r^{d-s}}
        &\leq\limsup_{r\searrow 0}\frac{V_A'(t_2(r))}{(d-s)(t_2(r))^{d-s-1}}\\&\leq \limsup_{r\searrow 0}\frac{2\mu_{d-1}(A_r)}{(d-s)r^{d-1-s}}= \frac{2}{d-s}\uSC_{d-1}^{s}(A).
    \end{align*}
    This shows the second inequality in \eqref{eq:SCd-1vsMinkout} and the third one is obtained similarly from the left inequality in \eqref{eq:RaWi10_3.2}.

    It is deduced in \cite[p.1665]{RaWi2010} (from \cite[Satz 4]{Kneser51}), that, for any $r>0$,
    \begin{equation*}
        \sH^{d-1}(\partial A_r)\leq \frac dr (V_A(r)-V_A(0)).
    \end{equation*}
    Combining this with \eqref{eq:parsurfarea} and noting that $V_A(r)-V_A(0)\to 0$, as $r\searrow 0$, we conclude that $\uSC^{d}_{d-1}(A)=\limsup_{r\searrow 0} r \mu_{d-1}(A_r)=0$. The relation $\usex_{d-1}(A)\leq d$ is a direct consequence, and the `$\leq$'-relations in \eqref{eq:s_d-1_dimout} follow from those in \eqref{eq:SCd-1vsMinkout}.

    Furthermore, if $\sex_{d-1}(A)$ exists, then \eqref{eq:s_d-1_dimout} implies $\sex_{d-1}(A)\leq \ldimout_M A\leq \udimout_M A\leq \sex_{d-1}(A)$, showing the equality $\sex_{d-1}(A)=\dimout_M A$.
    Finally, if $\SC_{d-1}^{s}(A)$ exists for some $s\in [0,d)$, then \eqref{eq:SCd-1vsMinkout} implies
    $\frac{2}{d-s}\SC_{d-1}^{s}(A)\leq\lMinkout{s}(A)\leq \uMinkout{s}(A)\leq \frac{2}{d-s}\SC_{d-1}^{s}(A)$ (which holds even in case $\SC_{d-1}^{s}(A)$ equals $0$ or $\infty$). Hence $\Minkout{s}(A)$ exists and satisfies the last equation stated in Proposition~\ref{prop:s_d-1}, completing its proof.
    \end{proof}

\begin{new}The upper $S$-content $\uSC^s(A)$ defined in \eqref{eq:S-content} has a lower counterpart $\lSC^s(A)$ defined by replacing $\limsup$ with $\liminf$.\end{new} With a bit of extra work it is possible to show that, even though the total mass of the support measure $\mu_{d-1}(A_\eps,\cdot)$ may differ from the surface area $\sH^{d-1}(\partial A_\e)$ for certain $\e>0$ (up to countably many), the asymptotic behavior as $\e\searrow 0$ of both quantities coincides. Hence, the (upper and lower) support contents of order $d-1$ coincide with the corresponding $S$-contents up to a universal normalization constant. \begin{new}Recall also the upper and lower \emph{$S$-dimensions} defined by $\udim_S A:=\inf\{s:\uS^s(A)=0\}$ and $\ldim_S A:=\inf\{s:\lSC^s(A)=0\}$, respectively, which in turn coincide with the support scaling exponents of order $d-1$.\end{new}
\begin{proposition}[Relation with $S$-content] \label{prop:S-content}
    For any nonempty compact set $A\subset\R^d$ and any $s\in[0,d)$,
    \begin{align} \label{eq:S-content-rel}
        \uSC_{d-1}^s(A)&=\tfrac{(d-s)\kappa_{d-s}}{2}\uSC^s(A) &\text{and}&& \lSC_{d-1}^s(A)&=\tfrac{(d-s)\kappa_{d-s}}{2}\lSC^s(A).
    \end{align}
   In particular, the support scaling exponent of order $d-1$ of $A$ equals its S-dimension, i.e., 
   \begin{align}
      \label{eq:S-dim-rel}
      \usex_{d-1}(A)&=\udim_S A &\text{and} &&\lsex_{d-1}(A)=\ldim_S A.
    \end{align}
\end{proposition}
\begin{proof} 
    The `$\leq$'-relations in \eqref{eq:S-content-rel} follow at once from \eqref{eq:parsurfarea}. For the reverse relations, we recall that, at any $r>0$,  the left and right derivatives $V_A^{(-)}$ and $V_A^{(+)}$ of $V_A$ exist. Moreover, $V_A^{(-)}$ is continuous from the left and $V_A^{(+)}$ of $V_A$ is continuous from the right, see \cite{Stacho76}.
    Let $D\subset(0,\infty)$ be the set of values $r$ such that $V_A'(r)$ exists. Write $S_A(r):=\sH^{d-1}(\partial A_r)$ and recall from \cite[Cor.~2.6]{RaWi2010} that the left continuity of $V_A^{(-)}$ implies that, for any $r>0$,
    \begin{equation}\label{eq:S-content-upperbd}
        S_A(r)\leq\lim_{t\nearrow r, t\in D} V_A'(t).
    \end{equation}
    Fix $\delta>0$ and let $(\e_n)_{n\in\N}$ be a decreasing null sequence such that
    \begin{equation*}
        \uSC^s(A)=\lim_{n\to\infty} \frac{S_A(\e_n)}{(d-s)\kappa_{d-s}\e_n^{d-s-1}}.
    \end{equation*}
    By \eqref{eq:S-content-upperbd}, we can find for each $n\in\N$ some $r_n<\e_n$ with $r_n\in D$ such that
    $$
    \frac{S_A(\e_n)}{(d-s)\kappa_{d-s}\e_n^{d-s-1}}\leq \frac{V_A'({r_n})}{(d-s)\kappa_{d-s} r_n^{d-s-1}}+\delta=\frac{2\mu_{d-1}(A_{r_n})}{(d-s)\kappa_{d-s} r_n^{d-s-1}}+\delta.
    $$
    It follows that
    $$
    \uSC^s(A)\leq \limsup_{n\to\infty} \frac{2\mu_{d-1}(A_{r_n})}{(d-s)\kappa_{d-s} r_n^{d-s-1}}+\delta
    \leq\frac{2}{(d-s)\kappa_{d-s}}\uSC^s_{d-1}(A)+\delta.
    $$
    Since $\delta>0$ was arbitrary, this shows the `$\geq$'-relation in the first equation in \eqref{eq:S-content-rel}.

    Similarly, by \eqref{eq:parsurfarea} and the right continuity of $V_A^{(+)}$, we have
    \begin{equation} \label{eq:S-content-lowerbd}
    2\mu_{d-1}(A_r)=V_A^{(+)}(r)=\lim_{t\searrow r} V_A^{(+)}(t)=\lim_{t\searrow r,t\in D} V_A'(t)=\lim_{t\searrow r,t\in D} S_A(t).
    \end{equation}
    Now fix again $\delta>0$. Let $(\e_n)_{n\in\N}$ be a decreasing null sequence such that
    \begin{equation*}
        \lSC^s_{d-1}(A)=\lim_{n\to\infty} \frac{\mu_{d-1}(A_{\e_n})}{\e_n^{d-s-1}}.
    \end{equation*}
    By \eqref{eq:S-content-lowerbd}, we can find for each
$n\in\N$ some $r_n\in(\e_n,2\e_n)$ with $r_n\in D$
such that
    $$
    \frac{2\mu_{d-1}(A_{\e_n})}{\e_n^{d-s-1}}\geq \frac{S_A({r_n})}{r_n^{d-s-1}}-\delta.
    $$
    In particular, $r_n\to0$ as $n\to\infty$.
    It follows that
    $$
    2\lSC^s_{d-1}(A)\geq \liminf_{n\to\infty} \frac{S_A(r_n)}{r_n^{d-s-1}}-\delta\geq \liminf_{r\searrow 0} \frac{S_A(r)}{r^{d-s-1}}-\delta=(d-s)\kappa_{d-s}\lSC^s(A)-\delta.
    $$
    Since $\delta>0$ was arbitrary, this shows the `$\geq$'-relation in the second equation in \eqref{eq:S-content-rel} and completes the proof.
\end{proof}

From equation~\eqref{eq:m0=s0} one might get the idea that $\bex_i(A)=\sex_i(A)$ may also hold for indices $i\neq 0$. Although this does happen for certain sets $A$ and indices $i$, in general, this is far from the truth. For instance, according to Theorem~\ref{thm:basic-exp}, the support measure $\mu_i(A,\cdot)$ vanishes if the outer Minkowski dimension of $A$ is smaller than $i$  (resulting in $\bex_i(A)=-\infty$), while the support scaling exponents $\lsex_i$ and $\usex_i$
will always be nonnegative.
Some general relations between the scaling exponents $\usex_i$ and $\ubex_i$, and $\udimout_M$ are summarized in the following statement, which is the main result of this section.

\begin{new}
\begin{theorem}\label{thm:support-exp}
    Let $A\subset\R^d$ be a nonempty compact set, and let
    $i\in I_d$. Then, for every $\e>0$,
    $\mu_i(A_\eps;\cdot)\not\equiv 0$.
    Moreover, the following inequalities hold:
    \begin{equation}\label{eq:s_i-ineq}
        0\leq \lsex_i(A)\leq \usex_i(A)\leq \udimout_M A.
    \end{equation}
    Furthermore,
    \begin{equation}\label{eq:s<max}
        \usex_i(A)=\max_{j\in\{0,\ldots,i\}}\ubex_j(A).
    \end{equation}
    In particular, the upper support scaling exponents form
    a non-decreasing sequence, that is, they satisfy
    \begin{equation}\label{eq:increasing-exp}
        \usex_0(A)\leq\usex_1(A)\leq\cdots\leq\usex_{d-1}(A).
    \end{equation}
    Finally,
    \begin{equation}\label{eq:sd-1-dim}
        \udimout_M(A)=\max_{i\in I_d}\usex_i(A)
        \quad\text{and}\quad
        \ldimout_M(A)\geq\max_{i\in I_d}\lsex_i(A)
        \geq\ldim_S(A).
    \end{equation}
\end{theorem}
\end{new}

\begin{remark} \label{rem:increasing-exp}
\begin{new} 
Theorem~\ref{thm:support-exp} implies in particular that, for any compact set in $\R^d$, the (upper) support scaling exponent $\usex_{k}$ of some order $k\in I_d$ is determined by the (upper) basic scaling exponents $\ubex_0,\ldots,\ubex_{k}$ up to that order. Knowledge of just $\ubex_k$ is not enough for this in general, but knowledge of all basic exponents up to order $k$ suffices. Conversely, if for some compact set and $k\in I_d$ the exponent $\usex_k$ is known (or even all exponents $\usex_0,\ldots,\usex_k$), then this does only yield an upper bound for $\ubex_k$ in general, namely 
$$
\ubex_k\leq \usex_k=\max_{j\in\{0,\ldots,k\}} \usex_j.
$$     
It is clear that this inequality
can be strict.
(Recall that even $\ubex_i(A)=-\infty$ is possible, while
$\usex_i(A)\geq0$.)
The corresponding comparisons between lower basic and
support scaling exponents require additional assumptions;
see Corollary~\ref{cor:lower-comparison-regular} below.
\end{new}
\end{remark}

\begin{proof}[{Proof of Theorem~\ref{thm:support-exp}:}]
   Let $B_\infty$ be the set from Proposition~\ref{prop:pos-curv} and, for any $\e>0$, let $T_\e$ be the mapping defined in equation \eqref{eq:T_eps-def}.  Then $T_\e(B_\infty)\subset N(A_\e)$,  and so,
    by \eqref{eq:mu-parallel-local},
    for any $i\in I_d$,
    \begin{equation}
        \mu_i(A_{\eps};T_\e(B_\infty))=\sum_{j=0}^ic_{i,j}\,\e^{i-j}\,\mu_j(A;B_\infty).
    \end{equation}
    Here we have used that on the set $B_\infty$ the local reach is infinite.
    Recall that $c_{i,j}>0$ for $i\geq j$. By Proposition \ref{prop:pos-curv}, we have $\mu_j(A;B_\infty)\geq 0$, with strict inequality for $j=0$.
    This implies that $\mu_i(A_{\eps};T_\e(B_\infty))\geq c\cdot\e^{i}$, where $c:=c_{i,0}\cdot\mu_0(A;B_\infty)>0$.
    We conclude that
    \begin{equation}
        |\mu_i|(A_{\eps})\geq|\mu_i|(A_{\eps};T_\e(B_\infty))\geq\mu_i(A_{\eps};T_\e(B_\infty))\geq c\cdot\e^{i}
    \end{equation}
    for any $\e>0$, which implies $\mu_i(A_\eps;\cdot)\not\equiv 0$ and, moreover, that $\lsex_i(A)\geq 0$.
    
\begin{new}
    For the proof of the `$\leq$'-relation in \eqref{eq:s<max}, we 
employ Proposition~\ref{prop:Steiner-for-mu_i}.
Fix $i\in I_d$ and let $q>\max_{j\leq i}\ubex_j(A)$.
Then, by definition of $\ubex_j$, one has, for each
$j=0,\ldots,i$,
$\e^{q-j}\fb_j^{\var}(A;\e)\to0$, as $\eps\searrow 0$.
By \eqref{eq:mu-parallel-var}, this implies that
\[
\e^{q-i}|\mu_i|(A_\eps)
\leq
\sum_{j=0}^i c_{i,j}\e^{q-j}\fb_j^{\var}(\e)
\rightarrow0
\quad\text{as }\eps\searrow 0,
\]
and hence $q\geq\usex_i(A)$.
Since $q>\max_{j\leq i}\ubex_j(A)$ was arbitrary,
we conclude that
$\max_{j\leq i}\ubex_j(A)\geq\usex_i(A)$,
proving the `$\leq$'-relation in \eqref{eq:s<max}.

For a proof of the reverse relation, let, for $\e>0$,
\begin{equation}
    N_{>\e}(A):=
\{(x,u)\in N(A):\dl(A,x,u)>\e\}
\end{equation}
denote the set of normal pairs whose local reach exceeds $\e$.\footnote{In particular, one has $\fb_i(A;\e)=\mu_i(A;N_{>\e}(A))$ for any nonempty compact set $A\subset\R^d$ and $i\in I_d$.}
Recall that the restrictions of the support measures to $N_{>\e}(A)$
are finite signed measures, by
Proposition~\ref{lem:parallel-supp-finite}.
For any measurable set $B\subseteq N_{>\e}(A)$, consider the
polynomial
\begin{equation}\label{eq:PBS1}
    P_B(s):=\sum_{j=0}^i c_{i,j}(\e s)^{i-j}\mu_j(A;B).
\end{equation}
The set $N_{>\e}(A)$ is held fixed while $s$ varies.
Since $T_{\e s}$ is injective and $\dl(A,x,u)>\e$
on $B$, equation \eqref{eq:mu-parallel-local} gives
\[
P_B(s)=\mu_i(A_{\e s};T_{\e s}(B)),
\qquad 0<s\leq1.
\]

Fix now any $(i+1)$ distinct numbers in $(0,1)$, for instance, equidistantly, 
\[
\theta_k:=\frac{k+1}{i+2},\qquad k=0,\ldots,i,
\]
and let
\[
L_k(s):=
\prod_{\substack{0\leq h\leq i\\h\neq k}}
\frac{s-\theta_h}{\theta_k-\theta_h}
\]
be the corresponding Lagrange polynomials.
Since $P_B$ is a polynomial of degree at most $i$, the
Lagrange interpolation formula
(see, e.g., \cite[\S2.1.1]{StoerBulirsch})
gives the representation
\begin{equation}\label{eq:PBS2}
    P_B(s)=\sum_{k=0}^i L_k(s)P_B(\theta_k).
\end{equation}
For $j=0,\ldots,i$, denote by $\lambda_{j,k}$ the coefficient
of $s^{i-j}$ in $L_k(s)$. Comparing coefficients between \eqref{eq:PBS1} and \eqref{eq:PBS2}, we obtain
\[
c_{i,j}\e^{i-j}\mu_j(A;B)
=
\sum_{k=0}^i\lambda_{j,k}
\mu_i(A_{\theta_k\e};T_{\theta_k\e}(B)).
\]
All coefficients $\lambda_{j,k}$ depend only on $i,j,k$:
they are coefficients of the polynomials $L_k$, which
depend only on the fixed nodes $\theta_k=(k+1)/(i+2)$.
In particular, they are independent of $A$, $\e$ and $B$.
We next apply this identity to the positive and negative sets of
a Hahn decomposition of $\mu_j(A;\cdot)|_{N_{>\e}(A)}$.
Taking absolute values and adding the two estimates yields
\[
c_{i,j}\e^{i-j}\fb_j^{\var}(A;\e)
\leq
\sum_{k=0}^i|\lambda_{j,k}|\,|\mu_i|(A_{\theta_k\e}).
\]
Here we have used that, for each $k$, the images of the
two Hahn sets under $T_{\theta_k\e}$ are disjoint subsets
of $N(A_{\theta_k\e})$.

Now let $q>\usex_i(A)$. Multiplying the last inequality
by $\e^{q-i}/c_{i,j}$, we infer that
\begin{align*}
\e^{q-j}\fb_j^{\var}(A;\e)
&\leq
\frac1{c_{i,j}}\sum_{k=0}^i
|\lambda_{j,k}|\e^{q-i}|\mu_i|(A_{\theta_k\e})\\
&=
\frac1{c_{i,j}}\sum_{k=0}^i
|\lambda_{j,k}|\theta_k^{i-q}
(\theta_k\e)^{q-i}|\mu_i|(A_{\theta_k\e})
\to0
\end{align*}
as $\eps\searrow 0$.
Hence $q\geq\ubex_j(A)$ for every $j\leq i$.
Since $q>\usex_i(A)$ was arbitrary, we conclude that
\[
\max_{0\leq j\leq i}\ubex_j(A)\leq\usex_i(A),
\]
which completes the proof of \eqref{eq:s<max}.

The monotonicity stated in \eqref{eq:increasing-exp} follows
immediately, since the maximum in \eqref{eq:s<max}
is taken over an increasing family of indices.
\end{new}
\begin{new}
The first equality in \eqref{eq:sd-1-dim} follows from \eqref{eq:increasing-exp}, \eqref{eq:s<max}, and Theorem~\ref{thm:basic-exp}. Indeed,
\[
\max_{i\in I_d}\usex_i(A)
=\usex_{d-1}(A)
=\max_{j\in I_d}\ubex_j(A)
=\udimout_M(A).
\]
  For the lower exponents $\lsex_i$, we 
use Proposition~\ref{prop:support-est}\,(i) instead, providing a constant $\check c_i>0$, independent of $\e$, such that
\[
|\mu_i|(A_\e)
\leq \check c_i\e^{i-d}V(A_\e\setminus A),
\qquad \e>0.
\]
Let $q>\ldimout_M A$. Then $\lMinkout{q}(A)=0$, and
there is a decreasing null sequence $(\e_n)_{n\in\N}$ such that 
$\e_n^{q-d}V(A_{\e_n}\setminus A)\to0$.
Consequently,
\[
\e_n^{q-i}|\mu_i|(A_{\e_n})
\leq C_i\e_n^{q-d}V(A_{\e_n}\setminus A)
\to0,
\]
which implies $q\geq\lsex_i(A)$.
Since $q>\ldimout_M A$ was arbitrary, we conclude that
$\lsex_i(A)\leq\ldimout_M A$ for every $i\in I_d$.
This proves the first inequality in the second relation
of \eqref{eq:sd-1-dim}.
The last relation in \eqref{eq:sd-1-dim} is immediate from
\eqref{eq:S-dim-rel}.
Finally, from \eqref{eq:sd-1-dim} it is now also transparent
that $\usex_i(A)\leq\udimout_M A$ holds for any $i\in I_d$,
completing the proof of \eqref{eq:s_i-ineq} and thus of
Theorem~\ref{thm:support-exp}.
\end{new}
\end{proof}

\begin{new}
The following statement gives the corresponding lower
comparisons under the additional assumption that the
relevant basic scaling exponents exist.

\begin{corollary}\label{cor:lower-comparison-regular}
    Let $A\subset\R^d$ be nonempty and compact, and let
    $i\in I_d$. Assume that the basic scaling exponents
    $\bex_j(A)$ exist for all $j=0,\ldots,i$.
    Here the value $-\infty$ is allowed and corresponds to $\mu_j(A;\cdot)\equiv0$.
    Then
    \begin{equation}\label{eq:lower-regular-max}
        \max_{0\leq j\leq i}\lsex_j(A)
        =
        \max_{0\leq j\leq i}\bex_j(A)
        =
        \usex_i(A).
    \end{equation} 
\end{corollary}

\begin{proof} 
Put $M:=\max_{0\leq j\leq i}\bex_j(A)$ and let $k$
be the smallest index attaining this maximum.
Since $\mu_0(A;\cdot)\not\equiv0$, Theorem~\ref{thm:basic-exp}
implies that $M$ is finite and nonnegative.

Equation \eqref{eq:s<max} gives
\[
\usex_k(A)=M,\qquad
\usex_j(A)<M\quad\text{for }j<k,
\]
and $\lsex_j(A)\leq\usex_j(A)\leq M$ for every $j\leq i$.
It therefore suffices to prove that $\lsex_k(A)=M$.

If $k=0$, this follows from \eqref{eq:m0=s0}.
Suppose $k>0$ and assume, for a contradiction, that
$\lsex_k(A)<M$.
Choose $q$ such that
\[
\max\left\{\lsex_k(A),
           \max_{0\leq j<k}\usex_j(A)\right\}<q<M.
\]
By the definition of $\lsex_k(A)$, there is a decreasing
null sequence $(\e_n)_{n\in\N}$ such that
\[
\e_n^{q-k}|\mu_k|(A_{\e_n})\rightarrow0.
\]
For every $j<k$, we also have
$\e_n^{q-j}|\mu_j|(A_{\e_n})\to0$,
since $q>\usex_j(A)$.
By \eqref{eq:bound-basic-up} in
Proposition~\ref{prop:invert-rela}, this implies that
\[
\e_n^{q-k}\fb_k^{\var}(A;\e_n)
\leq
\sum_{j=0}^k|b_{k,j}|\,
\e_n^{q-j}|\mu_j|(A_{\e_n})
\rightarrow0.
\]
Hence $\lbex_k(A)\leq q$. This contradicts the fact, that by definition of $M$ and $k$, $\lbex_k(A)=\ubex_k(A)=M>q$. 
Therefore $\lsex_k(A)=M$, proving \eqref{eq:lower-regular-max}.
\end{proof}
\end{new}

Basic contents (in case they exist) may be used to determine support contents as well as the outer Minkowski content.

\begin{theorem} \label{thm:SC-BC-rel}
	Let $A\subseteq\R^d$ be a compact set and $i\in I_d$.
	Assume that for some $q\in\R$ the $q$-dimensional basic  contents $\BC^q_j(A)$ of order $j$ exist (in $\R$) for all $j\in\{0,\ldots,i\}$.
	Then, the ($q$-dim.) support content of order $i$ exists and is given by
	\begin{equation}
		\SC_i^q(A)=\sum_{j=0}^ic_{i,j}\mathcal{M}_j^q(A),
	\end{equation}
where the constants $c_{i,j}$ are as in \eqref{eq:cij}.
    Moreover, if \begin{new}for some $q<d$\end{new} the basic contents $\BC^q_j(A)$ exist for all $j\in I_d$, then the outer Minkowski content exists 
    and is given by
    \begin{equation} \label{eq:Mink-BC-rel}
		\Minkout{q}(A)=\frac{1}{d-q}\sum_{j=0}^{d-1}\omega_{d-j}\mathcal{M}_j^q(A).
	\end{equation}
    \begin{new}In the case $q=d$, always $\Minkout{d}(A)=0$ (regardless of any assumptions on $\BC_j^d(A)$).\end{new}
\end{theorem}

\begin{proof}
	The statement follows by multiplying $\eps^{q-i}$ in equation \eqref{eq:mu-parallel} and taking the limit as $\eps\searrow 0$, recalling the basic and support contents from Definitions \ref{def:reach-measure} and \ref{def:support-content}, respectively. The assumed existence of the limits on the right of the equation implies the existence of the limit on the left.
    The second assertion follows from the first one (choose $i=d-1$) and the last assertion in Proposition~\ref{prop:s_d-1}, recalling the definition of the constants $c_{i,j}$.
\end{proof}

Equation \eqref{eq:Mink-BC-rel} provides a decomposition of the outer Minkowski content into a sum, in which each summand describes the individual contribution of a single support measure to the Minkowski content. The summand related to the $i$-th support measure may be (vaguely) interpreted as the contribution of the `\emph{$i$-dimensional features}' of the set $A$ to its Minkowski content. In some examples this can be made precise, see Section~\ref{sec:ex}. 
Of course, the interesting choice for the exponent $q$ in
\eqref{eq:Mink-BC-rel} is $q=\dimout_M A<d$, provided
this dimension exists and the corresponding basic contents
exist in $\R$.
In the sum on the right of \eqref{eq:Mink-BC-rel}, only
those basic contents can be nonzero for which the
corresponding upper basic exponents $\ubex_i(A)$ equal
$\dimout_M A$ (and are thus maximal by \eqref{eq:dim_max_m}).
Indeed, if $\ubex_i(A)<q$, then $\BC_i^q(A)=0$.
For sets $A$ in $\R^2$, for instance, equation
\eqref{eq:Mink-BC-rel} simplifies, for $0\leq q<2$, to
 \begin{equation} \label{eq:Mink-BC-rel2}
		\Minkout{q}(A)=\frac{2}{2-q}\left(\pi\BC^q_0(A)+\BC^q_1(A)\right).
	\end{equation}
Depending on whether the outer Minkowski dimension $q:=\dimout_M(A)$ satisfies $q=\bex_1(A)>\bex_0(A)$, $q=\bex_0(A)>\bex_1(A)$ or $q=\bex_1(A)=\bex_0(A)$, provided the corresponding $q$-dimensional basic contents
exist in $\R$, the outer Minkowski content $\Minkout{q}(A)$ will be determined by either $\BC^q_1(A)$, $\BC^q_0(A)$, or a linear combination of these two basic contents. 
In Section~\ref{sec:ex}, we discuss examples for all three
situations concerning the basic exponents, see
Examples~\ref{ex:SG}, \ref{ex:FD} and \ref{ex:FW},
respectively.

While the support content $\SC^q_{d-1}(A)$ of order $d-1$ is directly related to the outer Minkowski content, other support contents (of order $i\neq d-1$) should be compared to \emph{fractal curvatures}. We briefly recall fractal curvatures and associated curvature scaling exponents, which are
defined very analogously to support contents just with the support measures replaced by the curvature measures; see below. A difficulty is that often curvature measures are  not defined for \emph{all} parallel sets $A_\e$ of a given compact set $A\subset\R^d$. Note that for sets in $\R^d$ with $d\leq 3$, curvature measures $C_0(A_\e,\cdot),\ldots, C_{d-1}(A_\e,\cdot)$ are well-defined for Lebesgue almost all radii $\e>0$, see e.g.\ \cite{Zahle2011} or \cite[Ch.~10]{RZ19}. In dimensions $d>3$, this is put as an assumption.
Then one considers (upper or lower) essential limits instead of limits, meaning that the Lebesgue null set of radii $\e>0$ for which $A_\e$ does not admit curvature measures, is taken out. As before, we write $C_i(A_\e):=C_i(A_\e,\R^d\times\eS^{d-1})$ for the $i$-th total curvature and $C_i^{\var}(A_\e):=C_i^{\var}(A_\e,\R^d\times\eS^{d-1})$ for the total mass of the total variation measure of $C_i(A_\e,\cdot)$.
 A better name for fractal curvatures in the present context would probably be `\emph{curvature contents}'.

\begin{new}
In the following definition the curvature-direction
version of fractal curvatures is recalled; see
\cite[Remark~10.9]{RZ19}.
Projection onto $\R^d$ preserves total 
mass, but
does, in general, not preserve total mass of the variation measure.
Thus the resulting contents (fractal curvatures) for a fixed exponent $q$
 agree with those obtained from the projected measures,
whereas the variation-based contents and the curvature scaling exponents
below may differ. Note that they can only be larger than their projected versions.
\end{new}

\begin{definition}[Fractal curvatures of a compact set]\label{def:fr-curv}
Let $A\subseteq\R^d$ be a nonempty compact set such that, for Lebesgue a.a.\ $\e>0$, $A_\e$ admits curvature measures. For any $i\in I_d$ and $q\in\R$, the {\em ($q$-dimensional)  $i$-th upper fractal curvature of $A$} is defined by
\begin{equation}
	\uFC_i^q(A):=\esslimsup_{\eps\searrow 0}\e^{q-i}C_i(A_{\e}),
	\end{equation}
	and  its total variation analog by
	\begin{equation}
	\uFC_i^{\var,q}(A):=\esslimsup_{\eps\searrow 0}\e^{q-i}C_i^{\var}(A_{\e}).
	\end{equation}
	Lower variants are introduced by replacing the essential upper limits by essential lower limits, and denoted by $\lFC_i^q(A)$ and $\lFC_i^{\var,q}(A)$, respectively.
	Moreover, the {\em upper $i$-th curvature scaling exponent} $\ucex_i(A)$ of $A$ is defined by
	\begin{equation}
		\label{eq:curv-exp}
		\ucex_i:=\ucex_i(A):=\inf\{q\in\R:\uFC_i^{\var,q}(A)=0\}=\sup\{q\in\R:\uFC_i^{\var,q}(A)=+\iy\},
	\end{equation}
	and its {\em lower} counterpart $\lcex_i=\lcex_i(A)$ by replacing $\uFC_i^{\var,q}(A)$ with $\lFC_i^{\var,q}(A)$.

    If $\ucex_i=\lcex_i$, then we denote the common value by
$\cex_i$ and call it the {\em $i$-th curvature scaling
exponent} of $A$.
For any $q\in\R$, if $\uFC_i^q(A)=\lFC_i^q(A)$, we denote
the common value by $\FC_i^q(A)$ and call it the
{\em ($q$-dimensional) $i$-th fractal curvature of $A$}.
Analogously, if $\uFC_i^{\var,q}(A)=\lFC_i^{\var,q}(A)$,
we denote the common value by $\FC_i^{\var,q}(A)$.
    
\end{definition}
The existence of fractal curvatures $\FC^q_i(F)$ has been studied extensively in particular for self-similar sets, see e.g.~\cite{Wi08,Zahle2011} \begin{new} for the projected versions and \cite{BZ13,RZ19} for those based on curvature-direction measures. The question has been raised, whether there exist some connections with fractal tube formulas as studied in \cite{LF12,FZF}. The support contents open a new road (explored further in \cite{RaWi2}) to approach these connections. \end{new}Let us clarify the connection between fractal curvatures and support contents, for which the relation \eqref{eq:mu-C-rel} is crucial. To avoid any technicalities, let us assume here that $A$ is such that for almost all $\e>0$, the parallel set $A_\e$ is an UPR-set. Then for such $\e$ one has, by \eqref{eq:mu-C-rel},
\begin{align}
  \label{eq:mu-C-rel2}
  \mu_i(A_\e;\cdot)=C_i(A_{\e};N(A_{\e})\cap\cdot),
\end{align}
which implies in particular
$$
|\mu_i|(A_\e)\leq C_i^{\var}(A_{\e}).
$$
\begin{new}
To pass from this almost-everywhere inequality to the
scaling exponents, we use the right continuity of
$\e\mapsto|\mu_i|(A_\e)$,
which follows from the right continuity of the variational basic
functions and the local Steiner formula.

Indeed, for each $t>0$, a finite signed measure
$\nu_i^t$ on $N(A)$ is defined by
\[
\nu_i^t(B):=
\sum_{j=0}^i c_{i,j}t^{i-j}
\mu_j(A;N_{>t}(A)\cap B),
\qquad B\subseteq N(A)\text{ measurable}.
\]
Although $\nu_i^t$ vanishes outside $N_{>t}(A)$,
it may be regarded as a finite measure on $N(A)$ for any $t>0$.
Denote by $|\nu_i^t|$ the total variation of the measure $\nu_i^t$.
By \eqref{eq:mu-parallel-local} and the injectivity
of $T_t$, we have $\mu_i(A_t;T_t(\cdot))=\nu_i^t(\cdot)$ and so in particular
\begin{equation}\label{eq:tot-var-eq}
   |\mu_i|(A_t)=|\nu_i^t|(N(A)). 
\end{equation}
Now write, for any $t>r>0$, $|\nu_i^t-\nu_i^r|$ for the total variation
measure of the signed difference measure
$\nu_i^t-\nu_i^r$. 
Since $N_{>t}(A)\subseteq N_{>r}(A)$,
one has
\[
|\nu_i^t-\nu_i^r|(N(A))
=
|\nu_i^t-\nu_i^r|(N_{>r}(A)),
\]
and, by the definition of $\nu_i^t$ and $\nu_i^r$,
\begin{align*}
\nu_i^t-\nu_i^r
=\sum_{j=0}^i c_{i,j}\Big[
&(t^{i-j}-r^{i-j})
\mu_j(A;N_{>r}(A)\cap \cdot)
-t^{i-j}
\mu_j(A;N_{>r}(A)\setminus N_{>t}(A)\cap\cdot)
\Big].
\end{align*}
Hence, by the triangle inequality for total variation, we infer 
\begin{align*}
|\nu_i^t&-\nu_i^r|(N(A))\\
&\leq \sum_{j=0}^i c_{i,j}
\Big[
(t^{i-j}-r^{i-j})
|\mu_j|(A;N_{>r}(A))
+t^{i-j}
|\mu_j|(A;N_{>r}(A)\setminus N_{>t}(A))
\Big]\\
&=\sum_{j=0}^i c_{i,j}
\Big[
(t^{i-j}-r^{i-j})\fbv_j(A;r)
+t^{i-j}
\big(\fbv_j(A;r)-\fbv_j(A;t)\big)
\Big].
\end{align*}
The last equality follows from the definitions of
the basic variation functions and the inclusion
$N_{>t}(A)\subseteq N_{>r}(A)$.
The right-hand side tends to zero as $t\searrow r$,
by the finiteness and right continuity of
$\fbv_j(A;\cdot)$.
Finally, recalling \eqref{eq:tot-var-eq} and applying the reverse triangle inequality, 
we conclude that
\[
\big||\mu_i|(A_t)-|\mu_i|(A_r)\big|=
\big||\nu_i^t|(N(A))-|\nu_i^r|(N(A))\big|
\leq|\nu_i^t-\nu_i^r|(N(A))
\rightarrow 0
\qquad(t\searrow r).
\]
This shows that $\e\mapsto|\mu_i|(A_\e)$ is right-continuous
on $(0,\infty)$.

Consequently, for every $q\in\R$, the ordinary upper and
lower limits of $\e^{q-i}|\mu_i|(A_\e)$ at zero agree
with their essential counterparts: by right continuity,
each value is approximated on an interval to its right
of positive Lebesgue measure.\end{new}
From this it is easy to deduce that, for any $i\in I_d$,
\begin{align*}
  \usex_i(A)&\leq \ucex_i(A) &\text{and} &&  \lsex_i(A)&\leq \lcex_i(A).
\end{align*}
Since support contents and fractal curvatures are signed quantities, it is not possible to conclude corresponding inequalities for them (only the variation analogues satisfy $\uSC^{\var,q}_i(A)\leq \uFC^{\var,q}_i(A)$ for any $q$), but it is very clear from \eqref{eq:mu-C-rel2} that fractal curvatures can capture more geometric information not contained in the support measures.
To illustrate how much more information may be captured in fractal curvatures, we briefly discuss the Sierpi\'nski gasket $SG$. It is shown below in Example~\ref{ex:SG} that $\bex_0(SG)=0$, which implies $\sex_0(SG)=0$, see Remark~\ref{rem:b0-s0-rel}. (In fact, the measure $\mu_0(SG;\cdot)$ is only supported by the three vertices of the initial triangle of the construction and does not see the inner structure of $SG$, resulting in $|\mu_0|(SG_{\e})=\mu_0(SG_\e)=1$ for any $\e>0$.)
In contrast, the curvature measures $C_0(SG_\eps,\cdot)$ have a rich structure. They are defined for all $\e>0$ (as all the parallel sets $SG_\e$ are polyconvex).  It can be deduced from \cite[Ex.~2.4.1]{Wi08} that, for $q:=\dimout_M(SG)=\log_2 3$, $\e^{q}C_0^{\var}(SG_\e)$ is bounded from above and below by positive constants as $\e\searrow 0$ and that therefore $\cex_0(SG)=q>0$. Hence the curvature measures $C_0(SG_{\e},\cdot)$ are so much richer in structure that even the scaling exponent $\cex_0$ is strictly larger than the corresponding support scaling exponent $\sex_0$. One can show that the limit $\FC^q_0(SG)$ does not exist due to oscillations ($SG$ is a lattice self-similar set) but that upper and lower fractal curvatures exist.

Proposition~\ref{prop:invert-rela} allows to obtain a certain converse to Theorem~\ref{thm:SC-BC-rel}. The basic contents can also be expressed in terms of the support contents (in case the latter exist).

\begin{theorem} \label{thm:BC-SC-reverse}
	Let $A\subseteq\R^d$ be a compact set and $i\in I_d$.
	Assume that for some $q\in\R$ the $q$-dimensional support  contents $\SC^q_j(A)$ exist (in $\R$) for all $j\in\{0,\ldots,i\}$.
	Then, the ($q$-dim.) basic content of order $i$ exists (in $\R$) and
	\begin{equation}
		\BC_i^q(A)=\sum_{j=0}^ib_{i,j}\SC_j^q(A),
	\end{equation}
    where the coefficients $b_{i,j}$ are given by \eqref{eq:b_ij}.
\end{theorem}

\begin{proof}
	The statement follows by multiplying $\eps^{q-i}$ in equation \eqref{eq:basic_from_support} in Proposition~\ref{prop:invert-rela} and taking the limit as $\eps\searrow 0$. The assumed existence of the limits on the right of the equation implies the existence of the limit on the left.
\end{proof}

\section{Examples} \label{sec:ex}

 In the section, the possible behavior of basic exponents and contents is examined through several examples in $\R^2$, which demonstrate how these functionals are computed for specific sets. The first examples are nonfractal. 
 \begin{new}
     The circle considered in Example~\ref{ex:circle}
illustrates that the signed basic function $\fb_0$ can
vanish for all sufficiently small radii, even though
the measure $\mu_0$ is nonzero.
 \end{new}
 The frames in Example~\ref{ex:bdsq} demonstrate the difference between curvature measures and support measures.
These examples exhibit typical nonfractal behavior:  each $\bex_i$ exists and equals either $i$ or $-\infty$.  In contrast,
for fractals, at least one of the basic exponents $\bex_i$ shows a different behavior.  In $\R^2$, it is e.g.\ possible that $\bex_0=0$ and $\bex_1=\dimout_M>1$, as observed below for the Sierpi\'nski gasket (see Example~\ref{ex:SG}). It demonstrates that the measure $\mu_1$ can be responsible for the `fractality'.
But one can also have $\bex_0=\bex_1(=\dimout_M)$  (see the \emph{fractal windows} in Example~\ref{ex:FW}) and $\bex_1<\bex_0(=\dimout_M)$ (see the \emph{enclosed fractal dusts} in Example \ref{ex:FD}). Hence the Minkowski dimension can also be determined by the support measure $\mu_0$ alone or jointly by the measures $\mu_0$ and $\mu_1$, see also Remark~\ref{rem:frac-ex}.

We also address the support scaling exponents. 
Recall that, for sets in $\R^2$, the upper support scaling
exponents are easily deduced from the upper basic exponents;
see Remark~\ref{rem:increasing-exp}.

Recall that for any compact set $A\subset\R^2$ and $\sH^{1}$-a.a.\ $(x,u)\in N(A)$,
\begin{equation}\label{eq:H01}
	H_0(A,x,u)=\frac{1}{\sqrt{1+k_1(A,x,u)^2}} \quad\textrm{and}\quad H_1(A,x,u)=\frac{k_1(A,x,u)}{\sqrt{1+k_1(A,x,u)^2}},
\end{equation}
where $k_1(A,x,u)$ denotes the generalized principal curvature of $A$ at $(x,u)$.
\begin{new}
    For the closure of a domain with smooth boundary, the
outward unit normal $u$ is unique at each boundary point $x$,
and the generalized principal curvatures coincide with
the classical ones.
If the boundary itself is regarded as the set, both normal
orientations occur, as in the circle example below.
\end{new}
We start by discussing the basic contents and exponents for a disc and a circle in $\R^2$ of radius $R>0$.

\begin{examp}[Disc of radius $R$]\label{ex:disc}
Let $B:=B(0,R)$ be the disc in $\R^2$ with radius $R$ centered at the origin. 
Then $N(B)=\{(x,u): x\in\partial B, u=x/|x|\}$ and the generalized principal curvature at each $(x,u)\in N(B)$ is given by $k_1(B,x,u)=1/R$. Hence, $$
H_0(B,x,u)=\frac{R}{\sqrt{1+R^2}}=:H_0(R) \quad\text{ and }\quad H_1(B,x,u)=\frac{1}{\sqrt{1+R^2}}=:H_1(R)
$$
are both constant.  Moreover, the reach of $B$ is infinite.
Combining equations \eqref{eq:m_i} and \eqref{eq:support-formula}, one easily derives the basic functions of $B$ (note also that $\omega_1=2$ and $\omega_2=2\pi$)\begin{new}:
\begin{equation}\label{eq:M0B}
	\fb_0(t)=
    \frac{H_1(B)}{\o_2}\sH^1(N(B))=1
\quad \text{ and }\quad
	\fb_1(t)=
    \frac{H_0(B)}{\o_1}\sH^1(N(B))=R\pi.
\end{equation}
Here we used 
that $\sH^1(N(B))=2\pi\sqrt{1+R^2}$. 
(Indeed, 
$N(B)=\{(x,u)\in\R^4:\|x\|=R,\ u=x/R\}$, which is a circle in $\R^4$ with radius $\sqrt{1+R^2}$.) 
Since $\fb_0$ and $\fb_1$ are both constant functions, no rescaling is necessary as $t\searrow 0$ and hence the basic contents are
 \begin{equation}
	\BC_0^0(B)=1\quad\text{ and }\quad \BC_1^1(B)=R\pi.
\end{equation}
Hence, the corresponding basic scaling exponents are $\bex_0(B)=0$ and $\bex_1(B)=1$ (in correspondence with Proposition~\ref{prop:convex}).
Now Remark~\ref{rem:b0-s0-rel} and the identity
$|\mu_1|(B_\e)=\pi(R+\e)$ imply for the support scaling
exponents of $B$ that
$\sex_0(B)=\bex_0(B)=0$ and $\sex_1(B)=\dimout_M B=1=\max_{i\in\{0,1\}}\bex_i(B)$, and Theorem~\ref{thm:SC-BC-rel} implies for the support contents of $B$ that $\SC_0^0(B)=1$ and $\SC^1_1(B)=\BC^1_1(B)=R\pi$, since $\BC^1_0(B)=0$. This can also be verified by direct computations. $\lozenge$\end{new}
\end{examp}

\begin{examp}[Circle of radius $R$]\label{ex:circle}
Let now $S:=\partial B$ be the circle in $\R^2$ of radius $R>0$ centered at $0$.
The calculation of the basic functions $\fb_0$ and $\fb_1$ is analogous to that for the disc $B$ but we have to take into account that each foot point $x\in  S$ has now two unit normals, one outward pointing $u:=u(x):={x}/{|x|}$ and one inward pointing $-u$. Therefore, the normal bundle may be split into the disjoint union $N( S)=N^+\cup N^-$
with $N^+:=\{(x,u(x)): x\in S\}$ and $N^-:=\{(x,-u(x)): x\in S\}$. Observe that the reach function satisfies $\delta( S; x,u)=\infty$ for $(x,u)\in N^+$ and  $\delta( S; x,-u)=R$ for $(x,-u)\in N^-$. 
\begin{new}The principal curvatures satisfy the relation $k_1( S,x,u)=-k_1( S,x,-u)$, and hence we get
$$
H_1( S,x,u)=-H_1( S,x,-u)=\frac{1}{\sqrt{1+R^2}}=H_1(R).
$$
Observing that, by symmetry, $\sH^1(N^+)=\sH^1(N^-)=\sH^1(N(B))$, it is easily derived (compare with Example~\ref{ex:disc}), that
\begin{equation}
	\fb_0(t)=
    \1{\{t\geq R\}} \quad \text{ and } \quad \fbv_0(t)=1+\1{\{t<R\}}.
\end{equation}
In particular, $\fb_0$ vanishes for small $t>0$, while $\fbv_0$ is constant, implying
	$\mathcal{M}_0^0( S)=0$
    and 
    $\mathcal{M}_0^{\var,0}( S)=2$.
We conclude that $\bex_0( S)=0$. In fact, $\mathcal{M}_0^q( S)=0$ for any $q\in\R$. This particular example illustrates, why one should use the contents $\BC^{\var,q}_i( S)$ in definition \eqref{eq:reach-exp} of the $i$-th basic scaling exponent and not the basic contents  $\BC^{q}_i( S)$ themselves.
Clearly, $0$ is the correct $0$-th basic exponent here.

An analogous computation gives 
\begin{equation}\label{eq:M1paB}
		\fb_1(t)=\fbv_1(t)=(1+\1{\{t<R\}})R\pi.
    \end{equation}
Hence, $\mathcal{M}_1^1( S)=\mathcal{M}_1^{\var,1}( S)=2R\pi$ and so $\bex_1( S)=1$.
Now Theorem~\ref{thm:SC-BC-rel} can be used to deduce the (outer) Minkowski content of $ S$. Note that $\dimout_M S=1$ and so, by equation \eqref{eq:Mink-BC-rel2}, we recover
$
\BC^1(S)=\Minkout{1}(S)=2\left(\pi\BC^1_0(S)+\BC^1_1(S)\right)=4\pi R.
$ $\lozenge$\end{new}
\end{examp}

\begin{examp}[Boundary of a square in $\R^2$]
\label{ex:bdsq}
Let $Q$ be the boundary of a square in $\R^2$ with side length $a>0$ and denote by $V$ the set of vertices of $Q$.
The generalized normal bundle of $Q$ can be decomposed into three parts, $N(Q)=N^{\mathrm{v}}\cup N^+\cup N^-$, where we let $(x,u)\in N^{\mathrm{v}}$ iff $x\in V$, $(x,u)\in N^+$ iff $x\notin V$ and $u$ points outwards, and  $(x,u)\in N^-$ iff $u$ points inwards.
For $(x,u)\in N^{\mathrm{v}}$, we have $\delta(Q,x,u)=\infty$ and $k_1(Q,x,u)=\infty$.
For $(x,u)\in N^+$, $\delta(Q,x,u)=\infty$ and $k_1(Q,x,u)=0$, and for $(x,u)\in N^-$, $\delta(Q,x,u)=\min\{|x-v|: v\in V\}$ and $k_1(Q,x,u)=0$.
(For the local reach in the last case note that the exoskeleton of $Q$ is the union of its two diagonals and hence $\delta(Q,x,u)$  is the distance of $x$ in direction $u$ to the closest diagonal, which equals the above minimum.)

It follows that $H_1(Q,\cdot,\cdot)$ vanishes on $N^+\cup N^-$ and so the support measure $\mu_0(Q,\cdot)$ is concentrated on $N^{\mathrm{v}}$. For $(x,u)\in N^{\mathrm{v}}$, we have $H_1(Q,x,u)=1$ and hence
\begin{equation*}
    \fb_0(t)=\frac{1}{\o_2}\int_{N^{\mathrm{v}}}\mathscr{H}^{1}(\d (x,u))=1,
\end{equation*}
since $N^{\mathrm{v}}$ consists of four unit quarter-circles in the $u$-plane (one around each vertex).

In contrast, the support measure $\mu_1(Q,\cdot)$ is concentrated on $N^+\cup N^-$ and on this set $H_0(Q,\cdot,\cdot) \equiv 1$. (Indeed, $H_0(Q,x,u)=0$ for $(x,u)\in N^{\mathrm{v}}$.)
Therefore,
\begin{equation*}
\begin{aligned}
     \fb_1(t)&=\frac{1}{\o_1}\int_{N^+\cup N^-}\1\{t<\delta(Q,x,u)\}\mathscr{H}^{1}(\d (x,u))\\
     &=2a+ (2a-4t)\,\1\{t<a/2\}.
\end{aligned}
\end{equation*}

We conclude that $\mathcal{M}_0^0(Q)=1$ and $\mathcal{M}_1^1(Q)=4a$.
In particular, this implies  $\bex_0(Q)=0$ and $\bex_1(Q)=1$. While $\mathcal{M}_1^1(Q)$ equals the length of $Q$, $\mathcal{M}_0^0(Q)$ is not equal to the Euler characteristic here, reflecting the fact that $\mu_0(Q,\cdot)$ 'sees' the outer angles but not the inner ones. In contrast, the fractal curvature $\FC_0^0(Q)$ of $Q$ (see Def.~\ref{def:fr-curv} for the definition) equals 0, and thus coincides with the Euler characteristic of $Q$. (Indeed the parallel sets $Q_\e$ are polyconvex for all $\e>0$ and thus $C_0(Q_\e)$ is well-defined and equals 0 for $\e>0$ sufficiently small). 

In a similar way, one can show  that, for the boundary $Q$ of any $d$-dimensional hypercube in $\R^d$ and each $i\in I_d$, $\bex_i(Q)=i$ holds.
$\lozenge$
\end{examp}

Next we discuss three examples of fractal sets in $\R^2$ that demonstrate how to  apply basic functions and basic scaling exponents in fractal geometry. We start with the classical Sierpi\'nski gasket.

\begin{examp}[The Sierpi\'nski gasket] \label{ex:SG}
    Let $SG$ be the classical Sierpi\'nski gasket constructed by repeated removal of open equilateral triangles from a closed equilateral triangle $A$ of side length $1$, see Figure~\ref{fig:SG}~(left). 
In each step $n\in\N$, we remove $3^{n-1}$ open triangles of side length $2^{-n}$ and $SG$ is the limit set of this procedure.

First observe that the generalized normal bundle of $SG$ satisfies $N(SG)=N(A)\cup N^-$, where $N^-$ is the set of all pairs $(x,u)$ such that $x$ is an interior point of a side of one of the removed triangles and $u$ is the unique inner unit normal of this triangle at $x$. For any $(x,u)\in N^-$, near $x$ the boundary of $SG$ is locally flat and so $k_1(SG,x,u)=0$.

It is rather straightforward to conclude that $\fb_0(SG;t)=1$ for all $t>0$, because $\mu_0(SG,\cdot)$ is concentrated on those pairs $(x,u)\in N(A)$ for which $x$ is a vertex of the initial triangle $A$. (In fact, one has $\mu_0(SG,\cdot)=\mu_0(A,\cdot)$; cf.\ also  \cite[Example 4.8]{HugLasWeil} and the calculations in Example~\ref{ex:bdsq} just above.)
Hence, one obtains immediately $\BC_0^0(SG)=1$ and, since $\mu_0(A,\cdot)$ is nonnegative, $\bex_0(SG)=0$.

The determination of $\bex_1(SG)$ is more involved.
 The integral over the subset $N(A)\subset N(SG)$ contributes a constant term to the function $\fb_1(SG;t)$, namely the value $\fb_1(A;t)=3/2$ for each $t>0$ (corresponding to half the boundary length of $A$).

So let us turn our attention to the set $N^-$.
Let  $g=1/(4\sqrt{3})$ be the {\em inradius} of the triangle removed in the first step. For $k\in\N$, denote by $\Delta_k$ the set of triangles removed at step $k$. Clearly, the cardinality of $\Delta_k$ is $3^{k-1}$ and each $T\in \Delta_k$ has side length $2^{-k}$ and inradius $g\cdot2^{-k+1}$. Hence the local reach at pairs $(x,u)\in N^{-}$ with foot point $x\in\pa T$ is at most $g\cdot2^{-k+1}$ and so the triangle $T$ will not contribute to $\fb_1(SG;t)$ for $t\geq g\cdot2^{-k+1}$. The precise contribution $\ell(T,t)$ of $T$ to $\fb_1(SG;t)$ is given by
$$
\ell(T,t)=3\cdot 2^{-k-1}-3\cot(\pi/6) t=3(2^{-k-1}-\sqrt{3} t), \text{ for } t<g\cdot2^{-k+1},
$$
which is easily seen from Figure~\ref{fig:SG}. (It is half the total length of the red segements.)  
\begin{figure}[ht]
    \includegraphics[width=0.75\textwidth]{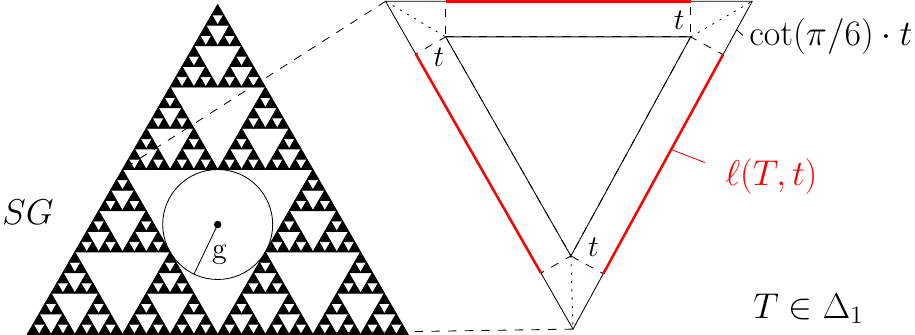}
    \caption{
    \label{fig:SG} Sierpi\'nski gasket $SG$ as discussed in Example~\ref{ex:SG}, generated from an initial equilateral triangle of sidelength 1. The largest removed triangle $T$ has inradius $g=1/(4\sqrt{3})$. The set of foot points in $\pa T$ of pairs in the generalized normal bundle of $SG$ contributing to $\beta_1(SG;t)$ is marked in red (and the contribution is given by $\ell(T,t)$).
    }
\end{figure}

Fix $n\in\N$ and let $t$ be such that $g\cdot 2^{-n}\leq t< g\cdot 2^{-n+1}$. Then triangles up to step $n$ contribute to $\fb_1(SG;t)$. More precisely, we get for such $t$
\begin{equation} \label{eq:bf-SG}
\begin{aligned}
     \fb_1(SG;t)&=\fb_1(A;t)+\sum_{k=1}^n \sum_{T\in \Delta_k} \ell(T,t)
    =\tfrac32+\sum_{k=1}^n 3^{k-1} \cdot 3(2^{-k-1}-\sqrt{3} t)\\
    &=\left(\tfrac32\right)^{n+1} - \tfrac 32 \sqrt{3} (3^n-1)\cdot t.
\end{aligned}
\end{equation}
\begin{new}
It is now obvious that $\fb_1(SG;t)$ is bounded from above
by the first term $\left(3/2\right)^{n+1}$.
Put $D:=\log_2 3$. Since $D>1$, our assumption
$t<g\cdot2^{-n+1}$ implies
\begin{align*}
t^{D-1}\fb_1(SG;t)
&\leq t^{D-1}\left(\tfrac32\right)^{n+1}<g^{D-1}2^{(-n+1)(D-1)}
  \left(\tfrac32\right)^{n+1}
=\tfrac94g^{D-1},
\end{align*}
where we have used $2^{D-1}=3/2$.
Since this estimate holds for any sufficiently small $t>0$,
and $\fb_1(SG;t)=\fbv_1(SG;t)$, we conclude that
$\uBC_1^D(SG)<\infty$ and so $\ubex_1(SG)\leq D$.

To see that $D$ is also a lower bound for $\lbex_1(SG)$
(and hence equal to $\bex_1(SG)$), observe from
\eqref{eq:bf-SG} that, for all $t$ such that
$g\cdot2^{-n}\leq t<g\cdot2^{-n+1}$,
\begin{align*}
\fb_1(SG;t)
&>\left(\tfrac32\right)^{n+1}
  -\tfrac32\sqrt{3}\cdot3^n\cdot g2^{-n+1}=\left(\tfrac32\right)^{n+1}
  -3\sqrt{3}\,g\left(\tfrac32\right)^n
=\tfrac34\left(\tfrac32\right)^n,
\end{align*}
where we have used $\sqrt{3}\,g=1/4$.
Since $t\geq g\cdot2^{-n}$ implies
\[
t^{D-1}\geq
g^{D-1}2^{-n(D-1)}
=g^{D-1}\left(\tfrac32\right)^{-n},
\]
we infer that
\[
t^{D-1}\fb_1(SG;t)>\tfrac34g^{D-1}.
\]
Hence, for all sufficiently small $t>0$,
$t^{D-1}\fb_1(SG;t)$ is bounded from below by a positive
constant. This shows $\lBC_1^D(SG)>0$ and thus, in particular,
$\lbex_1(SG)\geq D$.
We conclude that $\bex_1(SG)=D=\log_2 3$.

The basic content $\BC_1^D(SG)$ does not exist as a limit.
Indeed, fix $\lambda\in[1,2)$ and put
$t_n:=g\lambda2^{-n}$. Equation \eqref{eq:bf-SG} gives
\begin{align*}
t_n^{D-1}\fb_1(SG;t_n)
&=\tfrac38g^{D-1}
  \left(\lambda^{D-1}(4-\lambda)+\lambda^D3^{-n}\right)\rightarrow
  \tfrac38g^{D-1}\lambda^{D-1}(4-\lambda)
  \quad(n\to\infty).
\end{align*}
Since the expression on the right is nonconstant in $\lambda$,
different choices of $\lambda\in[1,2)$ give different
subsequential limits, proving the asserted nonexistence
of $\BC_1^D(SG)$.
$\lozenge$
\end{new}
\end{examp}

\begin{examp}[The fractal window] \label{ex:FW}
We discuss a family of inhomogeneous self-similar compact sets
$A=A(r)\subset\R^2$, depending on a parameter $r\in(0,1/2)$.
We first present exact formulas for their basic functions
and discuss the resulting basic exponents and contents.
The formulas will then be derived from renewal equations.
For $r\in[1/4,1/2)$, both basic exponents coincide and
determine the outer Minkowski dimension, i.e.,
$\bex_0(A)=\bex_1(A)=\dimout_M A$,
whereas for $r\in(0,1/4)$ the zeroth basic exponent is
strictly smaller than the first.

\begin{figure}[ht]
\includegraphics[width=5cm]{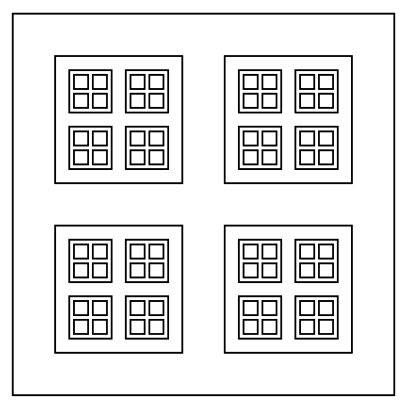}
\caption{\label{fig:window} The fourth prefractal approximation of the `fractal window' with scaling ratios equal to $r\in(0,1/2)$.
It is an example of an inhomogeneous self-similar set defined as a unique solution of the inhomogeneous equation $A=\bigcup_{i=1}^4\Phi_i(A)\cup Q$, where $Q$ is the boundary of the unit square and $\Phi_i$ are the 4 similarity contraction mappings.
The (outer) Minkowski dimension of the fractal window is $\max\{1,\log_{1/r}4\}$ and can attain any number in $[1,2)$. Here the set is depicted for $r=1/3$.}
\end{figure}

The set $A$ arises as an inhomogeneous self-similar set (or self-similar set with condensation) \cite{BaDe85,Barnsley88,MR3024316}, 
generated by four similarities with mutually equal contraction ratio $r\in(0,1/2)$.
Let $Q:=\partial [0,1]^2$ be the `initial frame' (the \emph{condensation set}), and let $\Phi_i:\R^2\to\R^2$, $i=1,\ldots,4$, be the contracting similarities, which scale by $r$ and translate by $\mathbf{x}_i$, where $\mathbf{x}_1=(p,p)$, $\mathbf{x}_2=(p,2p+r)$, $\mathbf{x}_3=(2p+r,p)$ and $\mathbf{x}_4=(2p+r,2p+r)$.
Then $p:=(1-2r)/3$ is the distance from any child square
$\Phi_i([0,1]^2)$ to the initial frame $Q$.
Also the distance between any two distinct child squares is at least $p$.
Let $A\subset\R^2$ be the unique (nonempty and compact) solution of the inhomogeneous equation $A=\bigcup_{i=1}^4\Phi_i(A)\cup Q$, see Figure~\ref{fig:window}.
Note that the corresponding homogeneous self-similar set is a Cantor dust (satisfying the strong separation condition) and, by well-known results, its Minkowski, Hausdorff and similarity dimensions coincide and are equal to the unique real solution of the Moran equation $1-4r^s=0$, i.e., to the value $s=\log_{1/r}4$; see \cite{Hut81}.
On the other hand, for any $r\in(0,1/2)$, the inhomogeneous self-similar set $A$, has (outer) Minkowski dimension equal to $\max\{1,\log_{1/r}4\}$; see \cite{MR3024316}.

\begin{new}
Observe that any pair in the normal bundle $N(A)$ of $A$ has its foot point in one of the frames. (In particular, points in $A$ which
are not on frames give no additional normal pairs.) Therefore, both support measures are nonnegative (compare with Example~\ref{ex:bdsq}), and so
\[
\fb_i(A;t)=\fbv_i(A;t),\qquad i=0,1.
\]
In particular, the basic functions are finite for any
positive radius $t$ and nonincreasing. 

Let 
\[
K=[0,1]^2,\qquad
D:=\log_{1/r}4,\qquad
g:=\frac{p}{\sqrt{2}}.
\]

Note for completemness that, for any $t\geq g$, the basic functions of $A$ coincide with those of $K$. Hence they are constant:
\begin{equation}\label{eq:FW-exterior}
\fb_0(A;t)=1,\qquad \fb_1(A;t)=2,
\qquad t\geq g.
\end{equation}
For $t\in(0,g)$ the situation is more involved. We claim that the basic functions of $A$ are given (for any parameter $r\in(0,\frac 12)$) by
\begin{equation}\label{eq:fb0Ar}
\fb_0(A;t)
=t^{-D}G_{0,r}(-\log_rt)-\frac13,
\qquad 0<t<g,
\end{equation}
and, for $r\neq1/4$,
\begin{equation}\label{eq:fb1Ar}
\fb_1(A;t)
=t^{1-D}G_{1,r}(-\log_rt)
+\frac4{1-4r}+\frac43t,
\qquad 0<t<g.
\end{equation}
The functions $G_{0,r}$ and $G_{1,r}$ are explicitly given in the proof below. For the derivation of the basic exponents and contents it is only important to know that they are bounded, nonconstant,
$1$-periodic functions which are bounded away from zero. (The function $G_{0,r}$ is positive. 
The same is true for $G_{1,r}$ when $r>1/4$,
whereas for $r<1/4$ the function $G_{1,r}$ is negative.)
In the critical case $r=1/4$, the corresponding formula
for $\fb_1(A;\cdot)$ is
\begin{equation}\label{eq:FW-critical}
\fb_1(A;t)
=\frac4{\log4}\log\frac1t
+G_{1,1/4}(-\log_{1/4}t)+\frac43t,
\qquad 0<t<g,
\end{equation}
where $G_{1,1/4}$ is again a bounded, nonconstant,
$1$-periodic function.

Before proving these formulas, let us discuss their consequences for the basic exponents
and contents.
The boundedness of the periodic factors in \eqref{eq:fb0Ar} and \eqref{eq:fb1Ar}, and of the logarithmic
term in \eqref{eq:FW-critical} imply that
\[
\bex_0(A)=D,\qquad
\bex_1(A)=\max\{1,D\}.
\]
Since both basic exponents exist, Theorem~\ref{thm:basic-exp}
implies that the outer Minkowski dimension exists
and equals $\max\{1,D\}$.
Furthermore, since $V(A)=0$, it also coincides
with the ordinary Minkowski dimension. Hence
\[
\dim_M A=\dimout_M A=\max\{1,D\}.
\]

Moreover, the positivity and nonconstancy of $G_{0,r}$
implies, for any $r\in(0,1/2)$,
\[
0<\lBC_0^D(A)<\uBC_0^D(A)<\infty.
\]
For $r>1/4$, the same argument applied to $G_{1,r}$
gives
\[
0<\lBC_1^D(A)<\uBC_1^D(A)<\infty.
\]
Indeed, every fixed value of a periodic factor recurs
along a sequence of radii $r^nt_0\to0$, while the
remaining terms in the corresponding rescaled formulas
tend to zero.
In particular, neither of the corresponding 
basic contents exists as a limit. 

For $r<1/4$, equation \eqref{eq:fb1Ar} gives
\[
\BC_1^1(A)
=\lim_{t\searrow 0}\fb_1(A;t)
=\frac4{1-4r}.
\]
Moreover, since $\bex_0=D<1$ in this case, $\BC_0^1(A)=0$ and so formula \eqref{eq:Mink-BC-rel} in Theorem~\ref{thm:SC-BC-rel} (see also \eqref{eq:Mink-BC-rel2}) directly yields the (outer) Minkowski content: 
$$
\Minkout{1}(A)=2\BC^1_1(A)+2\pi\BC^1_0(A)=\frac{8}{1-4r}.
$$
In particular, $A$ is $1$-dimensional Minkowski measurable.

In the critical case $r=1/4$, both basic exponents equal $1$, but
\eqref{eq:FW-critical} yields
\[
\fb_1(A;t)
=\frac4{\log4}\log\frac1t+O(1)
\qquad(t\searrow 0).
\]
In particular, the first basic content
diverges to $+\infty$.
Moreover, the general Steiner formula \eqref{eq:v-par}
gives
\[
V(A_t\setminus A)
=\frac8{\log4}\,t\log\frac1t+O(t)
\qquad(t\searrow 0),
\]
since $t\fb_0(A;t)=O(1)$.
Thus also $\Minkout{1}(A)=+\infty$.

These formulas exhibit both, the periodic scaling
behaviour and the logarithmic transition at $r=1/4$.
For $r<1/4$, the basic function $\fb_0(A;\cdot)$ still records
a nonconstant periodic scaling pattern but not of leading order.
$A$ has finite total length and $\bex_1(A)=1$.
A more detailed analysis of the periodic factors and
their connections with fractal zeta functions is done
in \cite{RaWi2}.

It remains to verify the expressions for the basic functions in \eqref{eq:fb0Ar} -- \eqref{eq:FW-critical}.
Set
\[
E_0:=Q,\qquad E_j:=\Phi_j(A),\quad j=1,\ldots,4.
\]
The sets $E_0,\ldots,E_4$ are pairwise separated by distance
at least $p$.

For $0<t<p/2$, a normal pair survives the reach cutoff at $t$
for one of the sets $E_j$ if and only if it survives it for
their union $A$.
Indeed, suppose $(x,u)\in N(E_j)$ and
$\dl(E_j,x,u)>t$. Choose
\[
t<s<\min\{\dl(E_j,x,u),p/2\}.
\]
The closed tangent disk $\overline B(x+su,s)$ meets $E_j$
only at $x$. If $y\in E_k$ with $k\neq j$, then
\[
|y-(x+su)|
\ge |y-x|-s
\ge p-s>s,
\]
so this disk does not meet any other component.
Hence $\dl(A,x,u)>t$.
The converse is immediate, since a tangent disk avoiding
the whole union also avoids the component containing its
contact point.

Locality of the support measures, \eqref{eq:loc-defined},
therefore gives
\[
\fb_i(A;t)=\fb_i(Q;t)
+\sum_{k=1}^4\fb_i(\Phi_k(A);t),
\qquad i=0,1,\quad 0<t<p/2.
\]

Under $\Phi_k(x)=rx+\mathbf{x}_k$, local reach scales by $r$,
while the $i$-th support measure scales by $r^i$.
Thus translation covariance and \eqref{eq:homogen}  (cf.~also \eqref{eq:scaling-beta}) imply
\begin{align*}
\fb_i(\Phi_k(A);t)
&=
r^i\int_{N(A)}
\1{\{t<r\dl(A,x,u)\}}\,
\mu_i(A;\d(x,u))=r^i\fb_i(A;t/r).
\end{align*}
Notice that the cutoff becomes $t/r$; this scaling identity
does not require $t/r<p/2$.

For the frame $Q$ (with sidelength $a=1$), the basic functions have been derived in Example~\ref{ex:bdsq}: 
for any $0<t<1/2$,
\[
\fb_0(Q;t)=1 \quad \text{ and }\quad \fb_1(Q;t)=4-4t.
\]
Since $t<p/2<1/2$, we obtain the following {\em renewal equations}
\begin{equation}\label{eq:FW-renewal}
\begin{aligned}
\fb_0(A;t)&=1+4\fb_0(A;t/r),\\
\fb_1(A;t)&=4-4t+4r\fb_1(A;t/r),
\end{aligned}
\qquad 0<t<p/2.
\end{equation}
The restriction $t<p/2$ is strict: at the separation
threshold, normal pairs can be lost.

Let
\[
\Omega:=\operatorname{int}K
\setminus\bigcup_{k=1}^4\Phi_k(K).
\]
The bounded complement of $A$ is the disjoint union
of the regions $\Phi_w(\Omega)$, where $\Phi_w$ ranges
over all finite compositions of the four similarities,
including the identity.
One identifies $w$ with all finite words over 4 symbols, including the empty word for the identity map, and then $|w|$ denotes the length of the word $w$.
Indeed, a point not on a frame either enters one of these
regions at a finite stage or remains inside nested child
squares at every stage. In the latter case it is a limit
of frame points and belongs to $A$.
Since the inradius of $\Omega$ is $g=p/\sqrt{2}$,
the region $\Phi_w(\Omega)$ has inradius $r^{|w|}g$.

It follows that all normal pairs with normals pointing into bounded components of the complement of $A$
have reach at most $g$. The exterior normals of $K$ have
infinite reach and contribute mass $1$ and $2$ to the two basic
functions, respectively (compare with \eqref{eq:FW-exterior}).

Since $rg<p/2$, we may apply \eqref{eq:FW-renewal}
at the radius $rt$ for every $0<t<g$.
The first equation gives
\[
\fb_0(A;rt)+\frac13
=4\left(\fb_0(A;t)+\frac13\right).
\]
Set
\[
h_0(t):=t^D\left(\fb_0(A;t)+\frac13\right).
\]
Since $r^D=1/4$, the preceding relation is equivalent to
\[
h_0(rt)=h_0(t),\qquad 0<t<g.
\]
Thus $h_0$ is multiplicatively $r$-periodic. More precisely,
its restriction to the interval $[rg,g)$ determines $h_0$
on all of $(0,g)$: for every $t\in(0,g)$ there is a unique
$n\geq0$ such that $r^{-n}t\in[rg,g)$, and then
\[
h_0(t)=h_0(r^{-n}t).
\]
Passing to the logarithmic variable $x=-\log_r t$,
multiplication of $t$ by $r$ corresponds to replacing $x$
by $x-1$. Hence $h_0(t)$ depends only on $-\log_r t$
modulo integers. Equivalently, there exists a $1$-periodic
function $G_{0,r}$ such that
\[
h_0(t)=G_{0,r}(-\log_r t),\qquad 0<t<g.
\]
This proves \eqref{eq:fb0Ar}. The periodicity follows from
$-\log_r(rt)=-\log_rt-1$.
The function $G_{0,r}$ is bounded above and away from zero:
on the fixed interval $[rg,g)$ this follows from
$1\leq\fb_0(A;t)\leq\fb_0(A;rg)$.

Next we turn to $\fb_1$.
For $r\neq1/4$, let
\[
a:=\frac4{1-4r}.
\]
Note that the affine function $t\mapsto a+\frac43t$ satisfies the second
renewal equation in \eqref{eq:FW-renewal}. 
Hence, subtracting it gives
\[
\fb_1(A;rt)-a-\frac43rt
=
4r\left(\fb_1(A;t)-a-\frac43t\right).
\]
Therefore
\[
h_1(t):= t^{D-1}\left(\fb_1(A;t)-a-\frac43t\right)
\]
is also invariant under $t\mapsto rt$.
Similarly as in the discussion for $h_0$ above, $h_1$
can now be represented by a $1$-periodic function
$G_{1,r}$. This proves \eqref{eq:fb1Ar}.
The periodic function $G_{1,r}$ is bounded, because
$\fb_1(A;t)$ is bounded on $[rg,g)$.

The periodic functions $G_{0,r}$ and, for $r\neq1/4$,
$G_{1,r}$ are nonconstant.
Indeed, the interval $[rg,p/2)$ is nonempty, and on this
interval $t/r\geq g$. Equations \eqref{eq:FW-renewal}
and \eqref{eq:FW-exterior} give
\[
\fb_0(A;t)=5,\qquad
\fb_1(A;t)=4+8r-4t.
\]
Consequently, on this interval,
\[
G_{0,r}(-\log_rt)=\frac{16}{3}t^D
\]
and, for $r\neq1/4$,
\[
G_{1,r}(-\log_rt)
=
t^{D-1}\left(4+8r-a-\frac{16}{3}t\right).
\]
The first expression is strictly increasing, while the
second is not constant on any interval, as is easily
seen by differentiation.

For completeness, the sign of $G_{1,r}$ for $r\neq1/4$
depends on $r$.
If $r>1/4$, then $a<0$ and, for $rg\leq t<g$,
\[
\fb_1(A;t)-a-\frac43t
\geq 2-a-\frac43g>0.
\]
Thus $G_{1,r}$ is positive and bounded away from zero.
If $r<1/4$, the total first support mass equals the sum
of all frame lengths,
\[
\mu_1(A;N(A))
=
4\sum_{n=0}^{\infty}(4r)^n
=\frac4{1-4r}=a.
\]
Hence $\fb_1(A;t)\leq a$ and
\[
\fb_1(A;t)-a-\frac43t\leq-\frac43t.
\]
Thus $G_{1,r}$ is negative in this case, with its absolute
value bounded away from zero on $[rg,g)$ and, by periodicity,
on $\R$.

In the critical case when $r=1/4$, subtracting $\frac43t$ from the second
renewal equation in \eqref{eq:FW-renewal} gives
\[
\fb_1(A;rt)-\frac43rt
=
4+\fb_1(A;t)-\frac43t.
\]
Since
\[
\frac4{\log4}
\left(\log\frac1{rt}-\log\frac1t\right)=4,
\]
subtracting also $\frac4{\log4}\log(1/t)$ leaves a function
invariant under $t\mapsto rt$.
Consequently, \eqref{eq:FW-critical} follows with a
bounded $1$-periodic function $G_{1,1/4}$.
This function is also nonconstant. Indeed, on the interval
$[rg,p/2)$, we have
\[
G_{1,1/4}(-\log_{1/4}t)
=6-\frac{16}{3}t-\frac4{\log4}\log\frac1t,
\]
whose derivative does not vanish identically.

This completes the derivation of the basic functions
and justifies the assertions stated above.
$\lozenge$
\end{new}
\end{examp}

In the next example we construct compact sets $A\subset\R^2$ for which the leading asymptotic term in the general Steiner formula is governed solely by the support measure $\mu_0(A,\cdot)$, meaning that the Minkowski dimension will be determined by the $0$-th basic scaling exponent.
In other words, we will construct sets $A$ for which $\bex_1(A)<\bex_0(A)=\dim_M A$. This also shows that, in general, one does not have an increasing order relation for the basic scaling exponents (compare to \eqref{eq:increasing-exp}, where such relation is stated for the support scaling exponents $\usex_i$).
It would be interesting to characterize those compact sets $A\subseteq\R^d$ for which $\bex_0(A)\leq\bex_1(A)\leq\ldots\leq \bex_{d-1}(A)$ holds.

\begin{examp}[Enclosed fractal dust] \label{ex:FD}
The building blocks of the {\em enclosed fractal dust} are boundaries of squares (`frames'), within which points are placed equidistantly on a grid. The side lengths $\ell_j$ of the $j$-th frame and the number of points $n_j^2$ placed inside will be chosen appropriately for each $j\in\N$.
The construction that we will present can be generalized considerably by choosing different sequences $(\ell_j)_{j\in\N}$ and $(n_j)_{j\in\N}$. For instance, $(\ell_j)$ might be chosen to be the $a$-string (Minkowski measurable case) or a self-similar string \cite{LF12} in order for $A$ to become Minkowski nonmeasurable.
Here, we will choose the side lengths and the number of points in a specific way in order to guarantee that the union of all squares has finite area and that it can be perfectly packed in a larger square.

For each $j\in\N$, let $K_j$ be a closed square in $\R^2$ of some side length $\ell_j>0$. We do not specify at the moment, where exactly these squares are located, but we will assume throughout that the interiors of these squares are pairwise disjoint, that is, they form a packing.

We denote by $\mathcal L:=(\ell_j)_{j\in\N}$ the (possibly unbounded) fractal string generated by the side lengths $\ell_j$.
We impose the restriction that $\sum_{j=1}^{\infty}\ell_j^2<\infty$ or, equivalently, that $\bigcup_{j=1}^{\infty}K_j$ has finite area. Then it is possible that the squares $K_j$ are all located inside some compact set $K$ and this will be another assumption. That is, we suppose that there is a square $K\subset\R^2$ such that $\overline{\bigcup_j K_j}= K$.

\begin{figure}[ht]
	\begin{tikzpicture}[scale=3]
	\draw (0, 0) rectangle (1,1);
		\equidistantPoints{xshift=1cm}{0.69}{1}{3};
		\equidistantPoints{xshift=1cm+0.69cm}{0.56}{2}{3}
		\equidistantPoints{xshift=1cm+0.69cm+0.56cm}{0.48}{3}{3}
		\equidistantPoints{yshift=1cm}{0.43}{4}{3}
		\equidistantPoints{yshift=1cm,xshift=0.43cm}{0.39}{5}{3}
		\equidistantPoints{yshift=0.69cm,xshift=1cm}{0.37}{6}{3}
		\equidistantPoints{yshift=0.56cm,xshift=1cm+0.69cm}{0.33}{7}{3}
		\equidistantPoints{yshift=0.48cm,xshift=1cm+0.69cm+0.56cm}{0.31}{8}{3}
		\equidistantPoints{yshift=0.69cm,xshift=1cm+0.37cm}{0.29}{9}{3}
		\equidistantPoints{yshift=0.69cm+0.37cm,xshift=1cm}{0.28}{10}{3}
		\equidistantPoints{yshift=1cm+0.43cm}{0.26}{11}{3}
	\end{tikzpicture}
	\caption{The first 12 building blocks of the `Riemann string' enclosed fractal dust $A$ with $\alpha=17/32$ and $m=1$. Note that the placements depicted here will not produce a perfect packing into a square and are for visualization purposes only.
	The placement procedure for perfect packing is much more complicated; see \cite{Jan23}.\label{fig:dust}}
\end{figure}
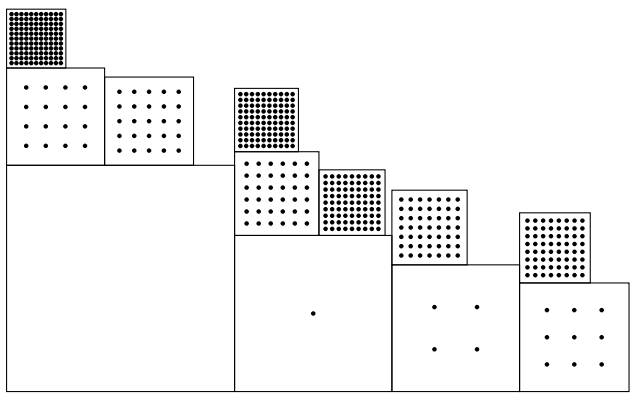

This can be done e.g.\ by choosing the fractal string $\mathcal{L}$ to be a Riemann-type string:
\begin{equation}
	\ell_j:=j^{-\alpha},\quad\textrm{where}\quad \alpha\in\left(\frac{1}{2},\frac{2}{3}\right].
\end{equation}
Indeed, it is known from \cite{MR4227795} that, for the above values of the parameter $\alpha$, a perfect packing into a single large square $K$ exists.
In this case the large square $K$ has area exactly $\sum_{j=1}^{\infty}j^{-2\alpha}=\zeta(2\alpha)$, i.e., $K$ has side length $\sqrt{\zeta(2\alpha)}$ where $\zeta$ is the Riemann zeta function.
(Note that, by \cite{PAM_doi:10.1007/s00454-023-00523-y}, for any $\alpha\in(1/2,1)$, there exists an index $N_{\alpha}\in\N$ such that the squares of side lengths $j^{-\alpha}$, $j\geq N_{\alpha}$, can be perfectly packed into a square of area $\sum_{j=N_{\alpha}}^{\infty}j^{-2\alpha}$.
Hence the current example could be easily modified to this generality if we would discard the first $N_{\alpha}-1$ squares in our construction.)

Fix a sequence $(n_j)_{j\in\N}$ of nonnegative integers. For each $j\in\N$, we place $n_j^2$ equidistant points in the interior of $K_j$.
More precisely, the points are placed to form a square lattice inside $K_j$ with minimal distance 
\begin{equation}\label{eq:r_j}
	r_j:=\frac{\ell_j}{n_j+1}.
\end{equation}

Let $A_j$ be the union of $\pa K_j$ and the $n_j^2$ points placed inside $K_j$. The \emph{enclosed fractal dust} $A$ is defined to be the closure of the union of the $A_j$, $A:=\overline{\bigcup_{j=1}^{\infty}A_j}$.

We now fix $m\in\N$ and choose the integers $n_j^2$, i.e., the number of equidistant points in each square, such that
\begin{equation}
	n_j=\ell_j^{-m/\alpha}-1=j^m-1.
\end{equation}

\begin{new}
For the present purpose of determining the basic exponents of $A$, it suffices to estimate the basic
functions up to positive multiplicative constants.
We write $f(t)\asymp g(t)$, as $t\searrow 0$, if there exist constants $c,C>0$ such that 
$c\,g(t)\leq f(t)\leq C\,g(t)$ for all sufficiently small $t>0$.

Recall that
\[
\ell_j=j^{-\alpha},\qquad
n_j=j^m-1,\qquad
r_j=\frac{\ell_j}{n_j+1}=j^{-(\alpha+m)},
\]
where $1/2<\alpha\leq2/3$ and $m\geq1$ is an integer.


We claim that, for any $t>0$,
\begin{align} \label{eq:ex7.5_beta_0}
1+\sum_{j:r_j>4t}n_j^2
\leq \fb_0(A;t)
\leq 1+\sum_{j:r_j>t}n_j^2
\end{align}
and
\begin{align} \label{eq:ex7.5_beta_1}
    \frac12\sH^1(\pa K)+\sum_{j:r_j>4t}\ell_j
\leq \fb_1(A;t)
\leq \frac12\sH^1(\pa K)+2\sum_{j:r_j>t}\ell_j.
\end{align}
To verify these bounds, fix $t>0$. First observe that the exterior normals of $K$ contribute mass $1$ to $\fb_0(A;t)$
and $\frac12\sH^1(\pa K)$ to $\fb_1(A;t)$, explaining the first summand in each of the four bounds (while the second summand will be related to the interiors
of the $K_j$ in each case). Inside each square $K_j$, the only contributing normal pairs are
those based at inserted points and those pointing inward
from sides of the squares; there are no inward normal pairs at
the square vertices.
By \eqref{eq:support-formula}, at an inserted point the contribution to $\mu_0(A;\cdot)$
is $1$, 
while $\mu_1(A;\cdot)$ vanishes.
The contribution of any side of any $K_j$ to $\mu_1(A;\cdot)$ is half the side length and to $\mu_0(A;\cdot)$ it is zero.
In particular, both support measures are nonnegative, so
$\fb_i(A;t)=\fbv_i(A;t)$ for $i=0,1$. (This clarifies also, why we are proving bounds for $\fb_i$ rather than $\fbv_i$.)

For the lower bounds in \eqref{eq:ex7.5_beta_0} and \eqref{eq:ex7.5_beta_1}, consider any $j$ such that $t<r_j/4$.
Each inserted point in $K_j$ has distance at least $r_j$ to
the rest of $A$, so for all its normal directions the local reach is larger than $t$. 
The inserted points in $K_j$ therefore
contribute exactly mass $n_j^2$ to $\fb_0(A;t)$. Summing over those $j$ provides the lower bound in \eqref{eq:ex7.5_beta_0}. 
Moreover, all normal pairs with foot point in the middle half of
a side of $K_j$ have local reach larger than $t$.
Indeed, their foot points
are at distance at least $\ell_j/4\geq r_j/4>t$ from any
vertex of $K_j$ and the distance to any inserted point is at least $r_j>2t$.
The four middle half-sides of $K_j$ contribute mass $\tfrac12\cdot4\cdot\ell_j/2=\ell_j$ to
$\fb_1(A;t)$, which gives the lower bound in \eqref{eq:ex7.5_beta_1}.

For the upper bounds, observe that the maximal possible contribution of $K_j$ to $\fb_0(A,t)$ is $n_j^2$ (mass 1 for each inserted point) and to $\fb_1(A;t)$ it is $2\ell_j$ (half the lengths of $\partial K_j$). There is certainly no contribution, if $t\geq r_j$, since any normal pair related to $K_j$ has local reach less than $r_j$. Summing over all the remaining $j$ and assuming maximal contributions of those $K_j$ provides the upper bounds claimed in \eqref{eq:ex7.5_beta_0} and \eqref{eq:ex7.5_beta_1}.


For $j\geq2$, $(j^m-1)^2\asymp j^{2m}$, and elementary
integral comparisons give
\begin{align}\label{eq:ex7.5_sum_est}
\sum_{j=2}^{J}(j^m-1)^2\asymp J^{2m+1},
\qquad
\sum_{j=1}^{J}j^{-\alpha}\asymp J^{1-\alpha}
\quad(J\to\infty).
\end{align}
Since
$
r_j=j^{-(\alpha+m)},
$
the condition $r_j>t$ in the sums of the upper bounds in \eqref{eq:ex7.5_beta_0} and \eqref{eq:ex7.5_beta_1} is equivalent to
$
j<t^{-1/(\alpha+m)}.
$
Similarly, the condition $r_j>4t$ in the lower bounds is equivalent to
$
j<(4t)^{-1/(\alpha+m)}
=4^{-1/(\alpha+m)}t^{-1/(\alpha+m)}.
$
Thus the two cutoffs differ only by a fixed multiplicative
constant in the index. Consequently, both sums are truncated
at an index of order $t^{-1/(\alpha+m)}$.
Combining this with \eqref{eq:ex7.5_sum_est} yields
\[
\fb_0(A;t)=\fbv_0(A;t)
\asymp t^{-\frac{2m+1}{\alpha+m}},
\qquad
\fb_1(A;t)=\fbv_1(A;t)
\asymp t^{-\frac{1-\alpha}{\alpha+m}}
\quad(t\searrow 0),
\]
By the definitions of the basic scaling exponents, this proves
\[
\bex_0(A)=\frac{2m+1}{\alpha+m}
>
\bex_1(A)=\frac{m+1}{\alpha+m}>1.
\]

The set $A$ has area zero, since the squares $K_j$ fill $K$ and $A$ contains only finitely many points
inside each of them. Consequently, Theorem~\ref{thm:basic-exp} 
gives
\[
\dim_M A=\dimout_M A=\bex_0(A).
\]
For the parameters $\alpha=17/32$ and $m=1$ used in
Figure~\ref{fig:dust}, the two basic exponents are
$96/49$ and $64/49$, respectively.

These two-sided estimates for the basic functions in \eqref{eq:ex7.5_beta_0} and \eqref{eq:ex7.5_beta_1} are sufficient to establish the phenomenon
of interest: the zeroth basic exponent $\ubex_0$ can be strictly larger
than the first $\ubex_1$.
A more detailed analysis of the
basic functions is done in \cite{RaWi2} using fractal zeta functions.
$\lozenge$
\end{new}
\end{examp}

\section{Subsets of \texorpdfstring{$\R$}{R}.}\label{sec:R1}

Although the support measures in \cite{HugLasWeil} are developed for dimension $d\geq 2$, one can include dimension $d=1$ into the discussion without difficulty and in fact beneficially as we briefly demonstrate. As an example we analyze the classic ternary Cantor set below.

Recall that the complement of any closed set $A\subset \R$ consists of countably many open intervals (possibly two of them unbounded). A point $x\in\partial A$ will be in the image of the metric projection onto $A$ if and only if it is an endpoint of such an interval. We write $\partial^+ A$ for the set of these endpoints. This is consistent with the notation in \cite{HugLasWeil}, where it is called the \emph{positive boundary} of $A$, cf.~\cite[p.~251]{HugLasWeil}. Note that $\partial^+ A$ is countable. Since there are only two directions in $\R$, any unit normal of $A$ at $x\in\partial^+ A$ will be either $-1$ or $+1$, depending on whether a complementary interval is located to the left or the right of $x$, respectively (or on both sides). Hence\begin{new}
$$
N(A)=\left\{(x,n):
x\in\partial^+A,\ n\in\{-1,1\},\
\{x+sn:0<s\leq\eps\}\subset A^c
\text{ for some }\eps>0
\right\}.
$$
\end{new}In dimension $d=1$ generalized principal curvatures are not defined (and not needed). However, it is consistent with higher dimensions to define, for any closed set $A\subseteq \R$, the symmetric function $H_0(A,\cdot)$ to be equal to 1 on $N(A)$. The support measure $\mu_0(A;\cdot)$ is then defined by
\begin{equation}
\begin{aligned}
    \mu_0(A;\cdot)&:=\frac{1}{2}\int_{N(A)}\1\{(x,u)\in\cdot\}\mathscr{H}^0(\d (x,u))\\
    &\,=\frac{1}{2}\#\{x:(x,-1)\in N(A)\cap\cdot\}+\frac{1}{2}\#\{x:(x,1)\in N(A)\cap\cdot\},
    \end{aligned}
\end{equation}
where $\mathscr{H}^0(\d (x,u))$ is the counting measure on $N(A)$ and $\#$ denotes the cardinality of a set; compare also to \cite[Prop.\ 4.1]{HugLasWeil}.

Suppose now that $A\subseteq \R$ is nonempty and compact. Then there exist $a_m,a_M\in A$, the minimal and maximal value of $A$, respectively, such that $A\subset[a_m,a_M]$.
Obviously, $(a_m,-1),(a_M,1)\in N(A)$ and the local reach is infinite at those points, which adds a summand 1 to the basic function $\fb_0(A;\cdot)$ of $A$.
Hence, we obtain, for any $t>0$,
\begin{equation}
\begin{new}
\begin{aligned}
    \fb_0(A;t)=&1+\frac{1}{2}\#\left\{x\in \partial^+ A\setminus\{a_m\} : \dl(A,x,-1)>t\right\}\\
    &+\frac{1}{2}\#\left\{x\in \partial^+ A\setminus\{a_M\} : \dl(A,x,1)>t\right\}.
\end{aligned}
\end{new}
\end{equation}
The nonempty compact set $A$ can also be described in terms of its fractal string representation as in \cite{LF12}. Recall that the complement of $A$ consists of a countable union of bounded open intervals $I_1,I_2,\ldots$ (plus the two unbounded ones, which we do not count here). Assume that they are ordered according to their lengths (in non-increasing order). In the terminology of \cite{LF12}, the sequence\footnote{The sequence can be  finite or empty, for instance if $A$ is a bounded segment.
However, for clarity of exposition, here we will implicitly assume that $A$ is such that the sequence is infinite.} of the lengths $\mathcal{L}_A=(\ell_j)_{j\in\N}$ of these intervals is called the \emph{fractal string} associated to $A$. Note that in $\mathcal{L}_A$ each length appears with the multiplicity corresponding to the number of complementary intervals of $A$ with that length.\footnote{Note that, since $A$ is compact, all the multiplicities are necessarily finite.}

Observe that each complementary interval $I_j=(x_j',x_j)$ of $A$ (with length $\ell_j$) contributes the two pairs $(x_j',1),(x_j,-1)$ to $N(A)$, both with local reach $\ell_j/2$. Those two pairs will contribute to $\fb_0(A;t)$ if and only if $|x_j-x_j'|=\ell_j>2t$.
Thus, for any $t>0$,
\begin{equation}\label{eq:b0string}
    \fb_0(A;t)=1+\#\{j\geq 1:\ell_j>2t\}=1+\sum_j \1{\{2t<\ell_j\}}.
\end{equation}
Recall now from \cite[\S1.1]{LF12} the {\em geometric counting function} $N_{\mathcal{L}_A}(x):=\#\{j:\ell_j^{-1}\leq x\}$ of the fractal string $\mathcal{L}_A$.
It is closely related to the volume of the parallel set of $A$ as well as to the theory of fractal zeta functions and complex dimensions developed in \cite{LF12}.
It is easy to see that, for any $t>0$,
\begin{equation}
    \label{eq:fb0-NL}
    \fb_0(A;t)=1+N_{\mathcal{L}_A}\left(\tfrac{1}{2t}\right)-w(2t),
\end{equation}
where $w(x)$ denotes the multiplicity of the length $x$ in the fractal string $\mathcal{L}_A$ (being equal to zero if the length $x$ does not appear in $\mathcal{L}_A$).
Note that there are at most countably many $t\in(0,+\infty)$ such that $2t$ is a length in $\mathcal{L}_A$, that is, such that $w(2t)$ is nonzero. Hence we have
$\fb_0(A;t)=1+N_{\mathcal{L}_A}\left(\tfrac{1}{2t}\right)$ almost everywhere (in Lebesgue sense).

We now record the one-dimensional Steiner formula in a form
that makes explicit the relation between the basic function,
the geometric counting function and the tube volume.
Although elementary, this formulation will be useful below
when relating support measures to the geometric zeta function.
\begin{proposition}[General Steiner formula in dimension $d=1$]\label{thm:stein1}
Let $A\subseteq \R$ be a nonempty compact set and let $\mathcal{L}_A=(\ell_j)_{j\in\N}$ be the associated fractal string as defined above.
Then, for any $\e>0$ the volume of the parallel set of $A$ is given by
    \begin{equation}
    \label{eq:stein_d=1}
    V(A_\e\setminus A)=2\int_0^{\e}\fb_0(A;t)\,\d t =2\e+2\int_0^{\e}N_{\mathcal{L}_A}\left(\tfrac{1}{2t}\right)\d t,
\end{equation}
where $N_{\mathcal{L}_A}$ is the geometric counting function of $\mathcal{L}_A$.
\end{proposition}

\begin{remark}
    A version of Proposition \ref{thm:stein1} was already stated and proven directly in \cite[Prop.~1]{klimes2021reading} (without reference to support measures and modulo the term $2\e$ which corresponds to the outer contribution of the minimum and maximum of $A$).
    However, in \cite[Prop.~1]{klimes2021reading} it was implicitly assumed that all of the lengths of $\mathcal{L}_A$ are mutually distinct since $A$ was an orbit of a discrete dynamical system.
    This assumption is not needed in general as we show here.

    We provide a direct proof of Proposition \ref{thm:stein1} using the well-known (and elementary) formula
    \begin{equation}\label{eq:stein_d=1old}
        V(A_\e\setminus A)=2\e+2\e\cdot N_{\mathcal{L}_A}\left(\tfrac{1}{2\e}\right)+\sum_{j:\ell_j<2\e}\ell_j, \quad \e>0.
    \end{equation}
   \end{remark}

   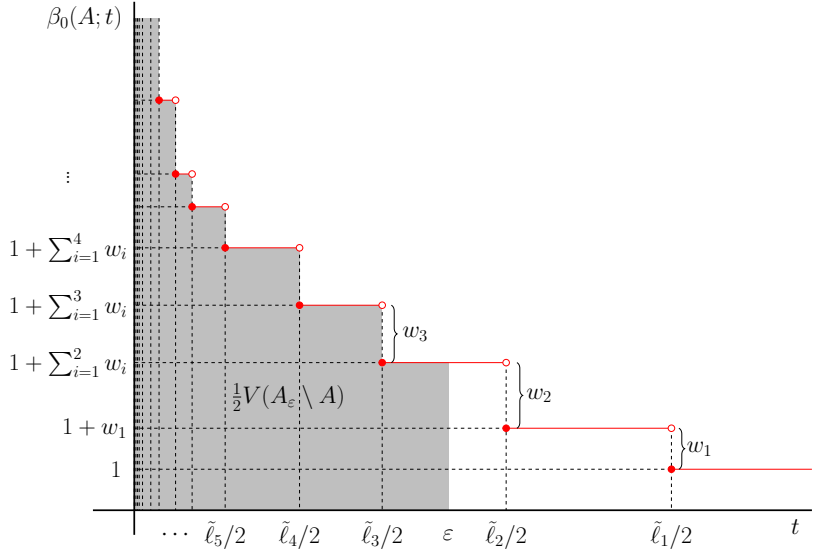
\begin{figure}[!ht]
\centering
\resizebox{0.75\textwidth}{!}{%
\begin{circuitikz}
\tikzstyle{every node}=[font=\large]

\node [font=\huge] at (25,3.25) {$t$};
\node [font=\huge] at (3.5,18.75) {$\fb_0(A;t)$};
\draw [short,red] (21.25,5) -- (25.5,5);
\fill[fill=lightgray, draw=none] (5,18.75) -- (5.75,18.75) -- (5.75,3.75) -- (5,3.75) -- cycle;
\fill[fill=lightgray, draw=none] (5.75,16.25) -- (6.25,16.25) -- (6.25,3.75) -- (5.75,3.75) -- cycle;
\fill[fill=lightgray, draw=none] (6.25,14) -- (6.75,14) -- (6.75,3.75) -- (6.25,3.75) -- cycle;
\fill[fill=lightgray, draw=none] (6.75,13) -- (7.75,13) -- (7.75,3.75) -- (6.75,3.75) -- cycle;
\fill[fill=lightgray, draw=none] (7.75,11.75) -- (10,11.75) -- (10,3.75) -- (7.75,3.75) -- cycle;
\fill[fill=lightgray, draw=none] (10,10) -- (12.5,10) -- (12.5,3.75) -- (10,3.75) -- cycle;
\fill[fill=lightgray, draw=none] (12.5,8.25) -- (14.5,8.25) -- (14.5,3.75) -- (12.5,3.75) -- cycle;
\draw [short,line width=1.5pt] (3.75,3.75) -- (25.5,3.75);
\draw [short,line width=1.5pt] (5,3) -- (5,19.25);
\draw [short,red] (16.25,6.25) -- (21.25,6.25);
\node [font=\huge] at (22.1,5.625) {$w_1$};
\draw [decorate,
	decoration = {brace,mirror,
		raise=5pt,
		amplitude=5pt}]  (21.3,5) --  (21.3,6.25);
\draw [short,red] (12.5,8.25) -- (16.25,8.25);
\node [font=\huge] at (17.25,7.25) {$w_2$};
\draw [decorate,
	decoration = {brace,mirror,
		raise=5pt,
		amplitude=5pt}]  (16.4,6.25)  --  (16.4,8.25);
\draw [short,red] (10,10) -- (12.5,10);
\node [font=\huge] at (13.5,9.125) {$w_3$};
\draw [decorate,
	decoration = {brace,mirror,
		raise=5pt,
		amplitude=5pt}]  (12.6,8.25) --  (12.6,10);
\draw [line width=0.1pt, dashed] (21.25,6.25) -- (21.25,3.75);
\node [font=\huge] at (21.25,3) {$\tilde{\ell}_1/2$};
\draw [line width=0.1pt, dashed] (16.25,8.25) -- (16.25,3.75);
\node [font=\huge] at (16.25,3) {$\tilde{\ell}_2/2$};
\draw [line width=0.1pt, dashed] (12.5,10) -- (12.5,3.75);
\node [font=\huge] at (12.5,3) {$\tilde{\ell}_3/2$};
\draw [line width=0.1pt, dashed] (10,11.75) -- (10,3.75);
\node [font=\huge] at (10,3) {$\tilde{\ell}_4/2$};
\draw [line width=0.1pt, dashed] (7.75,13) -- (7.75,3.75);
\node [font=\huge] at (7.75,3) {$\tilde{\ell}_5/2$};
\node [font=\huge] at (14.5,3) {$\eps$};
\node [font=\huge] at (9.66,7.2) {$\frac{1}{2}V(A_{\eps}\setminus A)$};
\draw [line width=0.1pt, dashed] (6.75,14) -- (6.75,3.75);
\node [font=\huge] at (6.3,3) {$\cdots$};
\draw [line width=0.1pt, dashed] (6.25,16.25) -- (6.25,3.75);
\draw [line width=0.1pt, dashed] (5.75,18.75) -- (5.75,3.75);
\draw [line width=0.1pt, dashed] (21.25,5) -- (5,5);
\draw [line width=0.1pt, dashed] (16.25,6.25) -- (5,6.25);
\draw [line width=0.1pt, dashed] (12.5,8.25) -- (5,8.25);
\draw [line width=0.1pt, dashed] (10,10) -- (5,10);
\draw [line width=0.1pt, dashed] (7.75,11.75) -- (5,11.75);
\draw [line width=0.1pt, dashed] (6.75,13) -- (5,13);
\draw [line width=0.1pt, dashed] (6.25,14) -- (5,14);
\draw [line width=0.1pt, dashed] (5.75,16.25) -- (5,16.25);
\draw [line width=0.1pt, dashed] (5.25,18.75) -- (5.25,3.75);
\draw [line width=0.1pt, dashed] (5.5,18.75) -- (5.5,3.75);
\draw [line width=0.1pt, dashed] (5.10,18.75) -- (5.10,3.75);
\draw [line width=0.1pt, dashed] (5,18.75) -- (5,3.75);
\draw [line width=0.1pt, dashed] (5.05,18.75) -- (5.05,3.75);
\draw [line width=0.1pt, dashed] (5.15,18.75) -- (5.15,3.75);
\node [font=\huge] at (4.45,5) {1};
\node [font=\huge] at (3.8,6.25) {$1+w_1$};
\node [font=\huge] at (3,8.25) {$1+\sum_{i=1}^2w_i$};
\node [font=\huge] at (3,10) {$1+\sum_{i=1}^3w_i$};
\node [font=\huge] at (3,11.75) {$1+\sum_{i=1}^4w_i$};
\node [font=\Huge] at (3,14) {$\vdots$};
\draw node [draw=red, circle, fill=white, inner sep=2pt] at (21.25,6.25) {};
\draw node [draw=red, circle, fill=red, inner sep=2pt] at (16.25,6.25) {};
\draw node [draw=red, circle, fill=white, inner sep=2pt] at (16.25,8.25) {};
\draw node [draw=red, circle, fill=red, inner sep=2pt] at (12.5,8.25) {};
\draw node [draw=red, circle, fill=red, inner sep=2pt] at (21.25,5) {};
\draw node [draw=red, circle, fill=white, inner sep=2pt] at (12.5,10) {};
\draw node [draw=red, circle, fill=red, inner sep=2pt] at (10,10) {};
\draw [short,red] (7.75,11.75) -- (10,11.75);
\draw node [draw=red, circle, fill=white, inner sep=2pt] at (10,11.75) {};
\draw node [draw=red, circle, fill=red, inner sep=2pt] at (7.75,11.75) {};
\draw [short,red] (6.75,13) -- (7.75,13);
\draw node [draw=red, circle, fill=white, inner sep=2pt] at (7.75,13) {};
\draw node [draw=red, circle, fill=red, inner sep=2pt] at (6.75,13) {};
\draw [short,red] (6.25,14) -- (6.75,14);
\draw node [draw=red, circle, fill=white, inner sep=2pt] at (6.75,14) {};
\draw node [draw=red, circle, fill=red, inner sep=2pt] at (6.25,14) {};
\draw [short,red] (5.75,16.25) -- (6.25,16.25);
\draw node [draw=red, circle, fill=white, inner sep=2pt] at (6.25,16.25) {};
\draw node [draw=red, circle, fill=red, inner sep=2pt] at (5.75,16.25) {};
\end{circuitikz}
}%

\caption{Visualization of Proposition \ref{thm:stein1} for some compact set $A\subset\R$. The graph of the basic function $t\mapsto \fb_0(A;t)$ is depicted as well as half of the volume, i.e.\ of the length $V(A_{\e}\setminus A)$ (the gray area under the graph between the abscissae 0 and $\e$). The numbers $\tilde{\ell}_m$, $m\in\N$, are the distinct lengths appearing in the fractal string $\mathcal{L}_A$ of $A$, while $w_m\in\N$ is the multiplicity of the length $\tilde{\ell}_m$ in $\mathcal{L}_A$.
Note that the graph of the counting function $t\mapsto N_{\mathcal{L}_A}(1/(2t))$ can be obtained by interchanging the white circles with the red ones and shifting the depicted graph vertically by $-1$. \label{fig:fbv0}
}
\end{figure}

\begin{new}
\begin{proof}[Proof of Proposition \ref{thm:stein1}]
By \eqref{eq:b0string} and Tonelli's theorem,
\begin{align*}
2\int_0^\e \fb_0(A;t)\,\d t
&=2\e+
  2\sum_j\int_0^\e \1{\{2t<\ell_j\}}\,\d t=2\e+\sum_j\min\{2\e,\ell_j\}\\
&=2\e+
  2\e N_{\mathcal{L}_A}\left(\tfrac{1}{2\e}\right)
  +\sum_{j:\ell_j<2\e}\ell_j=V(A_\e\setminus A),
\end{align*}
where the last equality is \eqref{eq:stein_d=1old}.
Tonelli's theorem applies because all summands are nonnegative.
The second equality in \eqref{eq:stein_d=1} follows from
\eqref{eq:fb0-NL}, since $w(2t)=0$ for almost every $t>0$.
\end{proof}
\end{new}

A geometric interpretation of the above proof is that each bounded complementary interval of length $\ell_j$
contributes two endpoint normals, each of mass $1/2$
and local reach $\ell_j/2$. Its contribution to
$2\int_0^\e\fb_0(A;t)\,\d t$ is therefore
$
2\int_0^\e\1{\{2t<\ell_j\}}\,\d t
=\min\{2\e,\ell_j\},
$
which is exactly the length covered inside that interval
by the $\e$-parallel set of $A$.
The two unbounded complementary intervals contribute $2\e$.
Summing these contributions gives the same formula directly
and explains the area under the graph in Figure~\ref{fig:fbv0}.

For later reference, we record the corresponding one-sided
derivatives of the tube function. This also makes explicit
the distinction between the strict reach cutoff in $\fb_0$
and the non-strict cutoff in $N_{\mathcal{L}_A}$,
including the effect of repeated gap lengths.
\begin{corollary}[Differentiability of the tube function in $\R$]\label{cor:diffR1}
    Let $A\subseteq \R$ be a nonempty compact set and let ${\mathcal{L}}_A=({\ell}_j)_{j\in\N}$ be the associated fractal string.
Then the tube function $\e\mapsto V_A(\e):=V(A_\e\setminus A)$ is piecewise linear (and thus piecewise differentiable) and, for each $\e\in(0,\infty)\setminus\{\ell_j/2:j\in\N\}$,
\begin{equation}
    V_A'(\e)=2\fb_0(A;\e)=2+2N_{\mathcal{L}_A}\left(\tfrac{1}{2\e}\right).
\end{equation}
Moreover, right and left derivatives exist for all $\e>0$ and satisfy 
$$
V_A^{(+)}(\e)=2\fb_0(A;\e) \quad \text{ and } \quad 
V_A^{(-)}(\e)=2+2N_{\mathcal{L}_A}\left(\tfrac{1}{2\e}\right),
$$ respectively.
Furthermore, if $\tilde{\mathcal{L}}_A=(\tilde{\ell}_m)_{m\in\N}$ denotes the fractal string consisting of all the {\em distinct} finite lengths of the disjoint open intervals that comprise the complement of $A$, and $w_m$ the multiplicity of $\tilde{\ell}_m$ for $m\in\N$, then,
for each $\e\in(\tilde{\ell}_{m+1}/{2},{\tilde{\ell}_m}/{2})$, $$V_A'(\e)=2+2\sum_{k=1}^mw_k.$$
In particular,
$V_A^{(-)}(\e)-V_A^{(+)}(\e)=2w(2\e)$.
\end{corollary}

\begin{proof}
 The assertion follows directly from \eqref{eq:stein_d=1} and the definitions of $\fb_0$ and $N_{\mathcal{L}_A}$.
    Note that $t\mapsto \fb_0(A;t)$ is constant on $[\tilde{\ell}_{m+1}/2,\tilde{\ell}_{m}/2)$ and equal to $1+\sum_{k=1}^mw_k$ while $t\mapsto N_{\mathcal{L}_A}((2t)^{-1})$ is constant on $(\tilde{\ell}_{m+1}/2,\tilde{\ell}_{m}/2]$ and equal to $\sum_{k=1}^mw_k$.
\end{proof}

We next make explicit the relation between the
one-dimensional support-measure description and the
geometric zeta function of a fractal string from \cite{LF12}.
The identities below are elementary; their purpose here
is to identify the same basic function as the quantity
underlying both the tube volume and the geometric zeta
function. They also provide a one-dimensional starting
point for the higher-dimensional connections considered
in the forthcoming paper \cite{RaWi2}.

Recall that for a compact set $A\subseteq\R$, the {\em geometric zeta function} of the fractal fractal string $\mathcal{L}_A$ is defined as the Dirichlet series
\begin{equation}
    \label{eq:geo-zeta}
    \zeta_{\mathcal{L}_A}(s):=\sum_{j\in\N}\ell_j^s=\sum_{m\in\N}w_m\tilde{\ell}_m^s,
\end{equation}
absolutely convergent for all $s\in\C$ such that $\re s>\udimout_MA$, and hence, holomorphic in the corresponding half-plane; see \cite[Ch.\ 1]{LF12}.
The corresponding theory of complex dimensions then studies the poles and other singularities of $\zeta_{\mathcal{L}_A}$ giving them geometric meaning that reflects the fractality of the set $A$.
It is not difficult to see directly that $\zeta_{\mathcal{L}_A}$ can be also expressed in terms of the Mellin transform of the geometric counting function; see \cite[Proof of Thm.\ 1.17]{LF12}:
\begin{equation}
    \label{eq:geo-zeta-count}
    \zeta_{\mathcal{L}_A}(s)=s\int_0^{\infty}\tau^{-s-1}N_{\mathcal{L}_A}(\tau)\,\d \tau.
\end{equation}
Now one can use a change of variables and \eqref{eq:fb0-NL} in order to obtain a representation for $\zeta_{\mathcal{L}_A}$ in terms of a Mellin integral of the basic function $\fb_0(A;\cdot)$:
\begin{equation}
    \begin{aligned}
        \zeta_{\mathcal{L}_A}(s)&=2^ss\int_0^{\infty}t^{s-1}N_{\mathcal{L}_A}(\tfrac{1}{2t})\,\d t=2^ss\int_0^{\e}t^{s-1}N_{\mathcal{L}_A}(\tfrac{1}{2t})\,\d t\\
        &=-2^s\e^s+2^ss\int_0^{\e}t^{s-1}\fb_0(A;t)\,\d t,
    \end{aligned}
\end{equation}
where $\e\geq\ell_1/2$ can be taken arbitrary since $N_{\mathcal{L}_A}(\tfrac{1}{2t})=0$ for all $t>\ell_1/2$. The integral representation is valid (initially) for all $s\in\C$ such that $\re s>1$.
Thus, after subtracting the contribution of the two
unbounded complementary intervals, the basic function
determines the geometric zeta function through a Mellin
transform. Together with \eqref{eq:stein_d=1}, this identifies
the tube volume and the geometric zeta function as two
integral transforms of the same support-measure quantity.
This explicit correspondence is the reason for including
these formulas here. It also motivates the higher-dimensional
relations to be developed in \cite{RaWi2}; however, no result from
that forthcoming work is needed for the present argument.

\medskip

Let us discuss the classical ternary Cantor set as an illustrating example.
\begin{examp}[The ternary Cantor set]
The ternary Cantor set $C$ is constructed starting from the closed unit interval by removing its (open) middle third interval in the first step and doing the same with all the remaining intervals in each step.
The fractal string $\mathcal{L}_C$ associated to $C$ contains length $3^{-k}$ with multiplicity $2^{k-1}$ for each $k\in\N$.
For $t<1/6$, one has $\#\{j\geq 1:\ell_j>2t\}=2^n-1$, where $n:=n(t):=\lceil -\log_3(2t)\rceil-1$ and $\lceil\cdot\rceil$ is the ceiling function.
Hence, for any $t<1/6$,
\begin{equation}\label{eq:fb0-pom}
\begin{aligned}
    \fb_0(C;t)&=2^n=2^{\lceil -\log_3(2t)\rceil-1}=2^{\{\log_3(2t)\}-\log_3(2t)-1}\\
    &=2^{\{\log_3(2t)\}-1}\cdot (2t)^{-\log_32} ,
\end{aligned}
\end{equation}
where $\{\cdot\}$ is the fractional part function.
We conclude that $\bex_0(C)=\log_32=\dimout_M C=\dim_M C$ as expected. The basic content $\BC_0^{\log_3 2}(C)$ does not exist as a limit due to the multiplicatively periodic function $t\mapsto 2^{-\log_32}\cdot 2^{\{\log_3(2t)\}-1}$, 
while the upper and lower basic contents exist and are equal to the supremum and the minimum of this function, respectively.

We can now recover the fractal tube formula from \cite[Eq.\ (1.14)]{LF12} for the Cantor set (for $\e<1/6$) by applying the general Steiner formula from Proposition \ref{thm:stein1}, which yields that $V(C_\e\setminus C)=2\int_0^{\e}\fb_0(C;t)\,\d t$.
For computing the integral we use the Fourier series expansion of the one-periodic function $u\mapsto b^{\{u\}}$ given by $\lim_{N\to +\infty}\sum_{n=-N}^N\frac{(1-b)\E^{2\pi\I n u}}{-\log b+2\pi\I n}$ and apply it to \eqref{eq:fb0-pom} in order to obtain
\begin{equation}
    \fb_0(C;t)= \frac{(2t)^{-D}}{2\log 3}\lim_{N\to +\infty}\sum_{n=-N}^{N}\frac{(2t)^{\I n\mathbf{p}}}{D-\I n\mathbf{p}},\quad t\in(0,\tfrac{1}{6})\setminus\{3^{-k}/2:k\in\N\},
\end{equation}
where $D:=\log_32$ and $\mathbf{p}:=\frac{2\pi}{\log 3}$.\footnote{Note that the Fourier series representation of $\fb_0(C;t)$ is conditionally pointwise convergent except at points at which $\fb_0(C;\cdot)$ has jump discontinuities, i.e., at the set $\{3^{-k}/2:k\in\N\}$ where it converges to the average value of the corresponding left and right limit of $\fb_0(C;\cdot)$.}
Now we integrate the above series term by term
in order to obtain
\begin{equation}
    V(C_\e\setminus C)=\frac{(2\e)^{1-D}}{2\log 3}\sum_{n\in\Z}\frac{(2\e)^{\I n\mathbf{p}}}{(D-\I n\mathbf{p})(1-D+\I n\mathbf{p})}, \quad \e\in(0,\tfrac{1}{6})
\end{equation}
which recovers the fractal tube formula \cite[Eq.~(1.14)]{LF12} (modulo the term $2\e$ which corresponds to the contribution of the two endpoints of $C$).
\begin{new}This is justified by the standard theorem on termwise integration of Fourier series
(see, e.g., \cite[\S9.3, Thm.~2]{EdwardsPenney}), applied to the
periodic factor after the logarithmic change of variables; the resulting series is absolutely convergent.\end{new}
Hence, by well known facts about the Fourier series, the fractal tube formula is valid pointwise for all $\e\in(0,1/6)$ and moreover, the series is absolutely and, hence, uniformly convergent on $(0,1/6)$.
In fact, the formula is actually valid on $(0,1/6]$ which can be verified directly and extends by continuity to $\e=0$ which is in agreement with the fact that $\e\mapsto V(C_\e\setminus C)$ is a continuous function on $[0,+\infty)$.
Finally, the reader can easily convince oneself that $V(C_\e\setminus C)=1+2\e$ for $\e\in[1/6,\infty)$ which gives us a complete description of the fractal tube formula for the Cantor set $C$.
\end{examp}

Finally, we remark that even in the one-dimensional context, the outer Minkowski dimension is really the notion which is encoded in the  basic function $\fb_0$.
To clarify this, one can consider a segment in $\R$ or, more interestingly, Smith-Volterra-Cantor sets (see \cite{Smith1875} or, e.g., \cite[Example 4.92]{LaRa24}) for which the Minkowski dimension equals 1 but the outer Minkowski dimension is strictly smaller:
Fix some $a>3$. Starting from the closed unit interval, the Smith-Volterra-Cantor set $SVC_a$ is constructed by removing in the first step an open interval of length $1/a$ centered at $1/2$. Then, for each $k>1$, 
a centered open interval of length $1/a^k$ is removed from each remaining interval.
It is easy to check that  $SVC_a$ has nonzero Lebesgue measure, hence it is of Minkowski dimension 1.
On the other hand, for such a set, by using \eqref{eq:b0string} and analogous reasoning as above, one obtains that $\bex_0(SVC_a)=\log_a2=\dimout_M(SVC_a)<1=\dim_M(SVC_a)$.

\section{Outlook: Basic zeta functions}
\label{page:zeta}

The examples in Section~\ref{sec:ex} illustrate that
the direct computation of basic functions and exponents
can require substantial effort.
The discussion of subsets of $\R$ in Section~\ref{sec:R1}
also revealed a close connection between basic functions
and geometric zeta functions through Mellin transforms.

In a follow-up paper \cite{RaWi2} we study how the present
results can be connected with the theory of fractal zeta
functions \cite{FZF} via \emph{basic zeta functions},
which allow one to obtain further information about
the geometric structure of fractal sets.
Here we briefly present some of the ideas.

Recall from \cite[Ch.\ 2]{FZF} that the \emph{distance zeta function} $\z_A\colon\C\to\C$ of a nonempty compact set $A\subseteq\R^d$ is defined by
\begin{equation}
	\label{eq:distance-zeta}
	\z_A(s):=\int_{A_{\e}\setminus A}\dist(z,A)^{s-d}\d z,\quad s\in\C,
\end{equation}
where $\e>0$ is fixed.
In fact, in \cite[Ch.\ 2]{FZF}, the function $\zeta_A$ is defined via integration over $A_\e$, under the assumption that $A$ has zero Lebesgue measure. In this case, the definitions are clearly equivalent and thus \eqref{eq:distance-zeta} represents a straightforward generalization to arbitrary compact sets.\footnote{In the terminology of \cite[Ch.\ 4]{FZF}, we are working here with the relative fractal drum (RFD) $(A,\Omega)$ with $\Omega:=A_{\e}\setminus A$. Note that $\udimout_MA=\udim_M(A,A_\eps\setminus A)$, where the latter is the relative upper Minkowski dimension of the RFD $(A,A_\eps\setminus A)$.}

We recall some properties of $\zeta_A$ and refer to \cite[Ch.\ 2]{FZF} for details.
The integral \eqref{eq:distance-zeta} is absolutely convergent and defines a holomorphic function in the open half-plane $\{\re s>\udimout_MA\}.$
Furthermore, for any real $s<\udimout_MA$, the integral is divergent. Under the additional assumptions that $D:=\dimout_MA$ exists, $\lMinkout{D}(A)>0$ and $\zeta_A$ possesses a meromorphic continuation to some larger half-plane containing $D$, the number $D$ is a pole of $\zeta_A$; see \cite[Thm. 2.1.11 and 4.1.7]{FZF}.

Given any open set $W\supseteq\{\re s>\udimout_MA\}$ in which $\zeta_A$ possesses a meromorphic continuation (usually called a window), one defines the set of \emph{complex dimensions} of $A$ \emph{visible in $W$} as the set of poles of $\zeta_A$.
It is easy to see that the complex dimensions do not depend on the parameter $\e$, since varying $\eps>0$ in \eqref{eq:distance-zeta} amounts to adding an entire function.
Hence, the complex dimensions can be thought of as a generalization of the classical Minkowski dimension and although they are defined analytically via the distance zeta function, they carry some geometric meaning.
In particular, one can derive asymptotic Steiner-type formulas for $A$ expressing the tube volume $V(A_\e\setminus A)$ in terms of the complex dimensions and the corresponding residues; see \cite[Ch.\ 5]{FZF}.

We show in \cite{RaWi2} that the distance zeta function can be decomposed into a sum of \emph{basic zeta functions} arising from the support measures, that is,
\begin{equation}
		\label{eq:eqf-shell}
		\z_A(s)=\sum_{i=0}^{d-1}\o_{d-i}\breve{\z}_{A,i}(s),		
	\end{equation}
	where, for $i\in I_d$, the {\em $i$-th basic zeta function} $\breve{\z}_{A,i}$ of $A$ is defined by the following Mellin-type integral:
	\begin{equation}
		\label{eq:support-zeta}
		\breve{\z}_{A,i}(s)=\breve{\z}_{A,i}(s;\e):=\int_0^{\e}t^{s-i-1}\fb_i(t)\d t.
	\end{equation}
Here we use the same fixed $\e>0$ as in \eqref{eq:distance-zeta}. We show that the integral defining $\breve{\z}_{A,i}$ is absolutely convergent in the open half-plane $\{\re s>\ov{\bex}_i(A)\}$.

Then one can use a Mellin inversion technique, similarly as in \cite[Ch.\ 5]{FZF}, in order to reconstruct the $i$-th basic function $\fb_i(A;\cdot)$ as a generalized asymptotic series. From this representation one can deduce easily the corresponding basic exponent but also much more, like the corresponding upper and lower basic contents as well as higher order asymptotic terms and possible non-powerlike scaling laws.
By using two types of functional equations obtained in \cite{RaWi2}, we are able to obtain closed forms of the $i$-th basic zeta functions $\breve{\z}_{A,i}$ without explicit knowledge of the basic functions $\fb_i$, which demonstrates the advantage of using the zeta function technique.

\medskip

As an illustration of this approach, consider again
the Sierpi\'nski gasket $SG$ from Example~\ref{ex:SG}.
For all $t\in(0,g)$, where $g=1/(4\sqrt{3})$,
one obtains the following Fourier representation:
$$
\begin{aligned}
\fb_1(SG;t)&=t^{1-\log_23}\frac{3\sqrt{3}}{\log 2}\sum_{k\in\Z}\frac{(4\sqrt{3})^{-\nu_k}\mathrm{e}^{2\pi\I k\log_{1/2}t}}{\nu_k(\nu_k-1)}    +   \frac{3\sqrt{3}}{2}t,
\end{aligned}
$$
where $\nu_k:=\log_23+\frac{2\pi\I k}{\log 2}$ for all $k\in\Z$. 
This Fourier series 
converges absolutely and uniformly, and hence defines a continuous, multiplicatively periodic and non-constant function $G$ on $(0,g)$.
We conclude that $\bex_1(SG)=\log_23=\dim_M(SG)$, and that the basic content $\BC_1^{\log_23}(SG)$ does not exist.
Moreover, $\uBC_1^{\log_23}(SG)$ and $\lBC_1^{\log_23}(SG)$ are related to the maximum and the minimum of $G$, respectively.

Furthermore, applying the general Steiner formula \eqref{eq:v-par}, we can recover the fractal tube formula given in \cite[Example 5.4.12]{FZF}, i.e., an exact expression for $V(SG_t)$ in terms of the poles of $\z_A$.

\begin{new}
\subsection*{Use of generative AI}
The first submitted version of the manuscript has been written without any AI support. During the preparation of the revision, the authors used ChatGPT (OpenAI) to review the mathematical correctness. In addition, the strengthening of Theorem \ref{thm:support-exp} (more precisely, the strengthening of the $\leq$-relation in \eqref{eq:s<max} to an equality) and the completion of Example \ref{ex:FW} were AI-assisted. All mathematical statements, proofs, citations, and final formulations were independently checked and approved by the authors, who assume full responsibility for the content.
\end{new}

\bibliographystyle{abbrv}
\bibliography{GR-bib}

\def\cprime{$'$} \def\cprime{$'$}
\begin{thebibliography}{10}

\bibitem{AmCoVi08}
L.~Ambrosio, A.~Colesanti, and E.~Villa.
\newblock Outer {M}inkowski content for some classes of closed sets.
\newblock {\em Math. Ann.}, 342(4):727--748, 2008.

\bibitem{Ba13}
A.~R. Backes.
\newblock A new approach to estimate lacunarity of texture images.
\newblock {\em Pattern Recognition Letters}, 34(13):1455--1461, 2013.

\bibitem{Barnsley88}
M.~Barnsley.
\newblock {\em Fractals everywhere}.
\newblock Academic Press, Inc., Boston, MA, 1988.

\bibitem{BaDe85}
M.~F. Barnsley and S.~Demko.
\newblock Iterated function systems and the global construction of fractals.
\newblock {\em Proc. Roy. Soc. London Ser. A}, 399(1817):243--275, 1985.

\bibitem{BZ13}
T.~J. Bohl and M.~Z\"ahle.
\newblock Curvature-direction measures of self-similar sets.
\newblock {\em Geom. Dedicata}, 167:215--231, 2013.

\bibitem{CaBr86}
J.~Brossard and R.~Carmona.
\newblock Can one hear the dimension of a fractal?
\newblock {\em Comm. Math. Phys.}, 104(1):103--122, 1986.

\bibitem{bruno2012shape}
O.~M. Bruno, R.~E. Plotze, M.~de~F{\'a}tima~de Oliveira, and L.~da~F.~Costa.
\newblock Shape analysis using fractal dimension: A curvature based approach.
\newblock {\em Chaos: An Interdisciplinary Journal of Nonlinear Science},
  22(4):043103, 2012.

\bibitem{CaGa13}
C.~Cattani and G.~Pierro.
\newblock On the fractal geometry of dna by the binary image analysis.
\newblock {\em Bulletin of Mathematical Biology}, 75(9):1544--1570, 2013.

\bibitem{DEMES13}
J.~J. {de Mesquita Sá Junior}, A.~{Ricardo Backes}, and P.~{César Cortez}.
\newblock Color texture classification based on gravitational collapse.
\newblock {\em Pattern Recognition}, 46(6):1628--1637, 2013.

\bibitem{EdwardsPenney}
C.~H. Edwards and D.~E. Penney.
\newblock {\em Differential Equations and Boundary Value Problems: Computing
  and Modeling}.
\newblock Pearson Prentice Hall, 4 edition, 2008.

\bibitem{Fed59}
H.~Federer.
\newblock Curvature measures.
\newblock {\em Trans. Amer. Math. Soc.}, 93:418--491, 1959.

\bibitem{MR3024316}
J.~M. Fraser.
\newblock Inhomogeneous self-similar sets and box dimensions.
\newblock {\em Studia Math.}, 213(2):133--156, 2012.

\bibitem{LapHer21}
H.~Herichi and M.~L. Lapidus.
\newblock {\em Quantized number theory, fractal strings and the {R}iemann
  hypothesis---from spectral operators to phase transitions and universality},
  volume~4 of {\em Fractals and Dynamics in Mathematics, Science, and the Arts:
  Theory and Applications}.
\newblock World Scientific Publishing Co. Pte. Ltd., Hackensack, NJ, [2021]
  \copyright 2021.
\newblock Research monograph.

\bibitem{HugLast00}
D.~Hug and G.~Last.
\newblock On support measures in {Minkowski} spaces and contact distributions
  in stochastic geometry.
\newblock {\em Ann. Probab.}, 28(2):796--850, 2000.

\bibitem{HugLasWeil}
D.~Hug, G.~Last, and W.~Weil.
\newblock A local {S}teiner-type formula for general closed sets and
  applications.
\newblock {\em Math. Z.}, 246(1-2):237–272, 2004.

\bibitem{HugSan2022}
D.~Hug and M.~Santilli.
\newblock Curvature measures and soap bubbles beyond convexity.
\newblock {\em Advances in Mathematics}, 411:108802, 2022.

\bibitem{Hut81}
J.~E. Hutchinson.
\newblock Fractals and self similarity.
\newblock {\em Indiana University Mathematics Journal}, 30(5):713--747, 1981.

\bibitem{DiIeva24}
A.~D. Ieva, editor.
\newblock {\em The Fractal Geometry of the Brain}, volume~36 of {\em Advances
  in Neurobiology}.
\newblock Springer Cham, 2024.
\newblock Published: 12 March 2024. Hardcover ISBN: 978-3-031-47605-1,
  Softcover ISBN: 978-3-031-47608-2.

\bibitem{Jan23}
J.~Januszewski.
\newblock A simple method for perfect packing of squares of sidelengths
  $n^{-1/2-\epsilon}$.
\newblock {\em Mathematische Semesterberichte}, 70:17–23, 2023.

\bibitem{MR4227795}
J.~Januszewski and {\L}.~Zielonka.
\newblock A note on perfect packing of squares and cubes.
\newblock {\em Acta Mathematica Hungarica}, 163(2):530–537, 2021.

\bibitem{klimes2021reading}
M.~Klime\v{s}, P.~Marde\v{s}ic, G.~Radunovi\'{c}, and M.~Resman.
\newblock {Reading analytic invariants of parabolic diffeomorphisms from their
  orbits}.
\newblock {\em Ann.\ Sc.\ Norm.\ Super.\ Pisa Cl.\ Sci.}, 26(2):1017--1047,
  2025.

\bibitem{Kneser51}
M.~Kneser.
\newblock \"uber den {R}and von {P}arallelk\"orpern.
\newblock {\em Math. Nachr.}, 5:241--251, 1951.

\bibitem{Lap08}
M.~L. Lapidus.
\newblock {\em In search of the {R}iemann zeros}.
\newblock American Mathematical Society, Providence, RI, 2008.
\newblock Strings, fractal membranes and noncommutative spacetimes.

\bibitem{LaPom}
M.~L. Lapidus and C.~Pomerance.
\newblock The {Riemann} zeta-function and the one-dimensional {Weyl}-{Berry}
  conjecture for fractal drums.
\newblock {\em Proc. Lond. Math. Soc. (3)}, 66(1):41–69, 1993.

\bibitem{LaPo96}
M.~L. Lapidus and C.~Pomerance.
\newblock Counterexamples to the modified {W}eyl-{B}erry conjecture on fractal
  drums.
\newblock {\em Math. Proc. Cambridge Philos. Soc.}, 119(1):167--178, 1996.

\bibitem{LaRa24}
M.~L. Lapidus and G.~Radunovi{\'c}.
\newblock {\em An invitation to fractal geometry: fractal dimensions,
  self-similarity, and fractal curves}, volume 247 of {\em Grad. Stud. Math.}
\newblock Providence, RI: American Mathematical Society (AMS), 2024.

\bibitem{lapidus2017distance}
M.~L. Lapidus, G.~Radunović, and D.~Žubrinić.
\newblock Distance and tube zeta functions of fractals and arbitrary compact
  sets.
\newblock {\em Advances in Mathematics}, 307:1215–1267, 2017.

\bibitem{FZF}
M.~L. Lapidus, G.~Radunović, and D.~Žubrinić.
\newblock {\em Fractal zeta functions and fractal drums}.
\newblock Springer Monographs in Mathematics. Springer, Cham, 2017.
\newblock Higher-dimensional theory of complex dimensions.

\bibitem{lapidus2018fractal}
M.~L. Lapidus, G.~Radunović, and D.~Žubrinić.
\newblock Fractal tube formulas for compact sets and relative fractal drums:
  Oscillations, complex dimensions and fractality.
\newblock {\em Journal of Fractal Geometry}, 5(1):1–119, 2018.

\bibitem{LF12}
M.~L. Lapidus and M.~van Frankenhuijsen.
\newblock {\em Fractal Geometry, Complex Dimensions and Zeta Functions:
  Geometry and Spectra of Fractal Strings}.
\newblock Springer New York, 2013.

\bibitem{Last06}
G.~Last.
\newblock On mean curvature functions of {B}rownian paths.
\newblock {\em Stochastic Process. Appl.}, 116(12):1876--1891, 2006.

\bibitem{LuVi16}
L.~Lussardi and E.~Villa.
\newblock A general formula for the anisotropic outer {M}inkowski content of a
  set.
\newblock {\em Proc. Roy. Soc. Edinburgh Sect. A}, 146(2):393--413, 2016.

\bibitem{Ma94-lacun}
B.~B. Mandelbrot.
\newblock Measures of fractal lacunarity: {M}inkowski content and alternatives.
\newblock In {\em Fractal geometry and stochastics ({F}insterbergen, 1994)},
  volume~37 of {\em Progr. Probab.}, pages 15--42. Birkh\"auser, Basel, 1995.

\bibitem{MRR2022}
P.~Marde{\v{s}}i{\'c}, G.~Radunovi{\'c}, and M.~Resman.
\newblock Fractal zeta functions of orbits of parabolic diffeomorphisms.
\newblock {\em Anal. Math. Phys.}, 12(5):70, 2022.
\newblock Id/No 114.

\bibitem{MCKEAN-1967}
H.~P. {M}c{K}ean, {J}r. and I.~M. {S}inger.
\newblock {C}urvature and the eigenvalues of the {L}aplacian.
\newblock {\em {J}ournal of {D}ifferential {G}eometry}, 1(1-2):43--69, 1967.

\bibitem{RaWi2}
G.~Radunović and S.~Winter.
\newblock From support measures to fractal zeta functions.
\newblock Unpublished manuscript.

\bibitem{RaWi2010}
J.~Rataj and S.~Winter.
\newblock On volume and surface area of parallel sets.
\newblock {\em Indiana Univ. Math. J.}, 59(5):1661--1685, 2010.

\bibitem{RZ19}
J.~Rataj and M.~Z\"ahle.
\newblock {\em Curvature measures of singular sets}.
\newblock Springer Monographs in Mathematics. Springer, Cham, 2019.

\bibitem{Res2013}
M.~Resman.
\newblock {{\(\varepsilon\)}}-neighborhoods of orbits and formal classification
  of parabolic diffeomorphisms.
\newblock {\em Discrete Contin. Dyn. Syst.}, 33(8):3767--3790, 2013.

\bibitem{Res2014}
M.~Resman.
\newblock $\varepsilon$-neighbourhoods of orbits of parabolic diffeomorphisms
  and cohomological equations.
\newblock {\em Nonlinearity}, 27(12):3005, nov 2014.

\bibitem{Smith1875}
H.~J.~S. Smith.
\newblock On the integration of discontinuous functions.
\newblock {\em Proc. Lond. Math. Soc.}, 6:140--153, 1875.

\bibitem{spodarev2015estimation}
E.~Spodarev, P.~Straka, and S.~Winter.
\newblock Estimation of fractal dimension and fractal curvatures from digital
  images.
\newblock {\em Journal of Mathematical Imaging and Vision}, 53(2):171--188,
  2015.

\bibitem{Stacho79}
L.~L. Stach\'o.
\newblock On curvature measures.
\newblock {\em Acta Sci. Math. (Szeged)}, 41(1-2):191--207, 1979.

\bibitem{Stacho76}
L.~L. Stachó.
\newblock On the volume function of parallel sets.
\newblock {\em Acta Universitatis Szegediensis. Acta Scientiarum
  Mathematicarum}, 38(3-4):365–374, 1976.

\bibitem{StoerBulirsch}
J.~Stoer and R.~Bulirsch.
\newblock {\em Introduction to Numerical Analysis}.
\newblock Springer, 3 edition, 2002.

\bibitem{PAM_doi:10.1007/s00454-023-00523-y}
T.~Tao.
\newblock Perfectly packing a square by squares of nearly harmonic sidelength.
\newblock {\em {Discrete \& Computational Geometry}}, page 1–12, 2023.

\bibitem{VAN_DEN_BERG-1994}
M.~van~den Berg and J.~F. {L}e {G}all.
\newblock {M}ean curvature and the heat equation.
\newblock {\em {M}athematische {Z}eitschrift}, 215(1):437--464, {J}an 1994.

\bibitem{Vi09}
E.~Villa.
\newblock On the outer {M}inkowski content of sets.
\newblock {\em Ann. Mat. Pura Appl. (4)}, 188(4):619--630, 2009.

\bibitem{Vi25}
T.~Vison\`a.
\newblock Intersections of randomly translated sets.
\newblock {\em J. Theoret. Probab.}, 38(1):Paper No. 3, 17, 2025.

\bibitem{Wi08}
S.~Winter.
\newblock Curvature measures and fractals.
\newblock {\em Dissertationes Math.}, 453:66, 2008.

\bibitem{Wi19}
S.~Winter.
\newblock Localization results for {M}inkowski contents.
\newblock {\em J. Lond. Math. Soc. (2)}, 99(2):553--582, 2019.

\bibitem{WiZa13}
S.~Winter and M.~Z\"ahle.
\newblock Fractal curvature measures of self-similar sets.
\newblock {\em Adv. Geom.}, 13(2):229--244, 2013.

\bibitem{Zahle2011}
M.~Z{\"a}hle.
\newblock Lipschitz-{Killing} curvatures of self-similar random fractals.
\newblock {\em Trans. Am. Math. Soc.}, 363(5):2663--2684, 2011.

\bibitem{zahle86}
M.~Zähle.
\newblock Integral and current representation of {F}ederer's curvature
  measures.
\newblock {\em Arch. Math.}, 46:557–567, 1986.

\end{thebibliography}
\end{document}